\documentclass[12pt, reqno]{amsart}

\usepackage[margin=2cm]{geometry}
\usepackage{amsmath, amssymb, amsfonts, amstext, verbatim, amsthm, mathrsfs}
\usepackage{microtype}
\usepackage[all,arc,knot,poly,cmtip]{xy}
\usepackage{mathtools}
\usepackage[colorlinks=true,linkcolor=blue,citecolor=blue,urlcolor=blue,citebordercolor={0 0 1},urlbordercolor={0 0 1},linkbordercolor={0 0 1}]{hyperref} 
\usepackage{enumerate}
\usepackage{stmaryrd}
\usepackage{subcaption}
\usepackage{tikz}
\usepackage{setspace}
\usepackage{graphicx}
\theoremstyle{plain}
\newtheorem{theorem}{Theorem}[section]
\newtheorem{proposition}[theorem]{Proposition}
\newtheorem{lemma}[theorem]{Lemma}
\newtheorem{corollary}[theorem]{Corollary}

\theoremstyle{definition}
\newtheorem{definition}[theorem]{Definition}

\DeclareMathOperator{\Hom}{\operatorname{Hom}}
\DeclareMathOperator{\Tr}{\operatorname{Tr}}
\DeclareMathOperator{\Sk}{\operatorname{Sk}}
\DeclareMathOperator{\Spec}{\operatorname{Spec}}
\DeclareMathOperator{\Db}{\operatorname{D^b}}
\DeclareMathOperator{\Coh}{\operatorname{Coh}}
\DeclareMathOperator{\gr}{\operatorname{gr}}
\DeclareMathOperator{\Gr}{\operatorname{Gr}}
\DeclareMathOperator{\twocolim}{\operatorname{2-colim}}
\DeclareMathOperator{\wt}{\operatorname{wt}}
\DeclareMathOperator{\supp}{\operatorname{supp}}
\DeclareMathOperator{\tw}{\operatorname{tw}}
\DeclareMathOperator{\Stab}{\operatorname{Stab}}

\begin{document}

\title{Monoidal categorification of genus zero skein algebras}
\author{Dylan G.L. Allegretti, Hyun Kyu Kim, and Peng Shan}
\address{Yau Mathematical Sciences Center, Department of Mathematical Sciences, Tsinghua~University, Beijing 100084, China}
\email{dylanallegretti@tsinghua.edu.cn, pengshan@tsinghua.edu.cn}
\address{School of Mathematics, Korea Institute for Advanced Study, 85 Hoegi-ro, Dongdaemun-gu, Seoul 02455, Republic of Korea}
\email{hkim@kias.re.kr}

\date{}

\maketitle

\begin{abstract}
We prove a conjecture of the first and third named authors relating the Kauffman~bracket skein~algebra of a genus zero surface with boundary to a quantized $K$-theoretic Coulomb~branch. As a consequence, we see that our skein algebra arises as the Grothendieck ring of the bounded derived category of equivariant coherent sheaves on the Braverman--Finkelberg--Nakajima variety of triples with monoidal structure defined by the convolution product. We thus give a monoidal categorification of the skein algebra, partially answering a question posed by D.~Thurston.
\end{abstract}

\setcounter{tocdepth}{1}
\tableofcontents

\section{Introduction}

Skein algebras of surfaces and skein modules of three-manifolds are basic objects of study in quantum topology. The most famous example of a skein algebra is the Kauffman~bracket skein~algebra, whose definition goes back to Przytycki~\cite{P91} and Turaev~\cite{T91} in the early 1990s. Given an oriented surface~$S$, the Kauffman bracket skein algebra of~$S$ is an algebra generated by isotopy classes of framed links in the three-manifold $S\times[0,1]$, modulo the same skein relations used to compute the Kauffman bracket invariant in knot theory. This algebra can also be seen as a quantization of the $\mathrm{SL}_2$-character variety of the surface~$S$.

In an influential 2014 paper~\cite{T14}, D.~Thurston showed that a classical limit of the Kauffman~bracket skein algebra possesses a canonical basis such that the structure constants for multiplication of basis elements are positive integers. In the same paper, Thurston asked whether the skein algebra has a natural categorification, saying ``the existence of a positive basis suggests the presence of a `nice' categorification, where product becomes a monoidal tensor product and sum becomes direct sum, or possibly a composition series.''

Several authors have approached the categorification of skein algebras~\cite{APS04,QW21} using constructions similar to Khovanov's categorification of the Jones~polynomial~\cite{K00}, and these methods have led to important results on canonical bases and their positivity properties~\cite{Q22}. In the present paper, we take a different approach, using tools from mathematical physics and geometric representation theory. Our results show that there exists a monoidal category whose monoidal product categorifies the multiplication in the Kauffman bracket skein algebra of a family of surfaces.

In previous work~\cite{AS24a}, inspired by the physical ideas of Gaiotto, Moore, and Neitzke~\cite{G12,GMN13a,GMN13b}, the first and third named authors described a relationship between skein algebras of surfaces and quantized $K$-theoretic Coulomb branches. The latter are noncommutative algebras introduced by Braverman, Finkelberg, and Nakajima~\cite{BFN18}. One of the main conjectures from~\cite{AS24a} states that the skein~algebra of a genus zero surface with boundary is isomorphic to such a quantized Coulomb~branch. The main result of the present paper is a proof of this conjecture.

The quantized Coulomb branch is, essentially by definition, the Grothendieck ring of the bounded derived category of equivariant coherent sheaves on an infinite-dimensional space, equipped with a monoidal product called the convolution product~\cite{BFN18}. Thus, for a genus zero surface with boundary, we obtain a categorification of the skein algebra by a triangulated monoidal category. This categorification connects two seemingly different mathematical constructions, and we believe it will provide valuable new tools and perspectives for studying the positivity phenomena that originally motivated Thurston's proposal.

\subsection{Skein algebras and quantized Coulomb branches}

Recall that if $S$ is a compact oriented surface with boundary, then the $\mathrm{SL}_2$-character variety of~$S$ is a space parametrizing principal $\mathrm{SL}_2$-bundles with flat connection on~$S$. More~concretely, it can be defined by considering representations $\pi_1(S)\rightarrow\mathrm{SL}_2$ of the fundamental group into~$\mathrm{SL}_2$. The set $\Hom(\pi_1(S),\mathrm{SL}_2)$ of all such representations has the structure of an affine scheme, and the $\mathrm{SL}_2$-character variety is the affine GIT~quotient 
\[
\mathcal{M}_{\mathrm{flat}}(S,\mathrm{SL}_2)=\Hom(\pi_1(S),\mathrm{SL}_2)\sslash\mathrm{SL}_2
\]
of this scheme by the conjugation action of~$\mathrm{SL}_2$. Note that any loop $\gamma\subset S$ defines a regular function $\Tr_\gamma$ on $\mathcal{M}_{\mathrm{flat}}(S,\mathrm{SL}_2)$ which maps the conjugacy class of a representation $\rho:\pi_1(S)\rightarrow\mathrm{SL}_2$ to the trace $\Tr_\gamma(\rho)=\Tr\rho(\gamma)$.

When the surface $S$ has nonempty boundary and $\gamma_1,\dots,\gamma_n\subset S$ are loops freely homotopic to the boundary components, it is quite natural to consider the relative character variety, which is a subspace of $\mathcal{M}_{\mathrm{flat}}(S,\mathrm{SL}_2)$ parametrizing representations $\rho:\pi_1(S)\rightarrow\mathrm{SL}_2$ that map $\gamma_i$ to a fixed conjugacy class for each~$i$. Since the conjugacy class of a semisimple complex $2\times2$ matrix of determinant one is determined by its trace, we may choose an $n$-tuple $\boldsymbol{\lambda}=(\lambda_1,\dots,\lambda_n)\in(\mathbb{C}^*)^n$ and define the relative character variety to be the closed subscheme $\mathcal{M}_{\mathrm{flat}}^{\boldsymbol{\lambda}}(S,\mathrm{SL}_2)\subset\mathcal{M}_{\mathrm{flat}}(S,\mathrm{SL}_2)$ cut out by the equations $\Tr_{\gamma_i}=\lambda_i+\lambda_i^{-1}$.

The Kauffman~bracket skein~algebra of~$S$ is a noncommutative deformation of the algebra of regular functions on~$\mathcal{M}_{\mathrm{flat}}(S,\mathrm{SL}_2)$. More~precisely, it is a noncommutative algebra $\Sk_A(S)$ over the ring~$\mathbb{C}[A^{\pm1}]$ such that we recover the algebra of regular functions by specializing~$A$ to~$-1$. We recall its definition in Section~\ref{sec:GenusZeroSkeinAlgebrasAndTheirGenerators}. We also define the relative skein algebra $\Sk_{A,\boldsymbol{\lambda}}(S)$. It is a noncommutative algebra over $\mathbb{C}[A^{\pm1},\lambda_1^{\pm1},\dots,\lambda_n^{\pm1}]$ such that we recover the algebra of regular functions on the relative character variety $\mathcal{M}_{\mathrm{flat}}^{\boldsymbol{\lambda}}(S,\mathrm{SL}_2)$ by specializing~$A$ to~$-1$ and specializing the~$\lambda_i$ to nonzero complex numbers. We refer the reader to~\cite{AS24a} for more details about the relationship between these skein algebras and the character varieties.

Our main goal in the present paper is to relate the skein algebra of a genus zero surface to an associated quantized Coulomb branch. In general, a Coulomb branch is a space that physicists associate to a compact Lie group~$G_{\mathrm{cpt}}$, called the gauge group, and a quaternionic representation~$M$ of~$G_{\mathrm{cpt}}$. In their pioneering work~\cite{BFN18}, Braverman, Finkelberg, and Nakajima replaced the group~$G_\mathrm{cpt}$ by its complexification~$G$ and gave a precise mathematical definition of this space under the assumption that the representation $M$ is of cotangent type, that is, $M\cong N\oplus N^*$ for some complex representation~$N$ of~$G$.

Their definition involves a remarkable space $\mathcal{R}_{G,N}$ associated to $G$ and~$N$ called the variety of triples. As we review in Section~\ref{sec:QuantizedCoulombBranchesAndCategorification}, it is an ind-scheme generalizing the affine Grassmannian of~$G$. In the examples relevant for the present paper, the action of~$G$ on~$N$ extends to an action of~$\widetilde{G}=G\times F$ where $F\cong(\mathbb{C}^*)^n$ is an algebraic torus called a flavor symmetry group. The variety of triples admits an action of the group $\widetilde{G}_\mathcal{O}$ of $\mathcal{O}$-points of~$\widetilde{G}$ where $\mathcal{O}=\mathbb{C}\llbracket z\rrbracket$. One can therefore define the $\widetilde{G}_{\mathcal{O}}$-equivariant $K$-theory $K^{\widetilde{G}_{\mathcal{O}}}(\mathcal{R}_{G,N})$ of the variety of triples. Here we complexify all $K$-groups so that the latter is a complex vector space. Braverman, Finkelberg, and Nakajima~\cite{BFN18} showed that this $K$-theory has a natural convolution product which is commutative, and they defined the $K$-theoretic Coulomb branch to be the spectrum 
\[
\mathcal{M}_C(G,F,N)=\Spec K^{\widetilde{G}_{\mathcal{O}}}(\mathcal{R}_{G,N})
\]
of the resulting commutative $\mathbb{C}$-algebra.

There is a canonical noncommutative deformation of this algebra, whose underlying vector space is the $\widetilde{G}_{\mathcal{O}}\rtimes\mathbb{C}^*$-equivariant $K$-theory $K^{\widetilde{G}_{\mathcal{O}}\rtimes\mathbb{C}^*}(\mathcal{R}_{G,N})$ where the $\mathbb{C}^*$-factor acts on the variety of triples by loop rotation. It was shown in~\cite{BFN18} that this vector space admits a convolution product, making it into a noncommutative algebra over $K^{\mathbb{C}^*\times F}(\mathrm{pt})\cong\mathbb{C}[q^{\pm1},t_1^{\pm1},\dots,t_n^{\pm1}]$. We will replace $\mathbb{C}^*$ by its double cover and thereby extend scalars to~$\mathbb{C}[q^{\pm\frac{1}{2}},t_1^{\pm1},\dots,t_n^{\pm1}]$. The resulting noncommutative algebra is called a quantized $K$-theoretic Coulomb branch.

Below we consider the surface $S=S_{0,n+2}$ obtained from the 2-sphere by removing $n+2$ open disks with pairwise disjoint closures where $n\geq1$. In the previous paper~\cite{AS24a}, the first and third named authors described a group $\widetilde{G}=G\times F$ and representation~$N$ of~$\widetilde{G}$ associated to such a surface~$S$. We recall the definitions of~$\widetilde{G}$ and~$N$ in Section~\ref{sec:TheGaugeAndFlavorGroups}. We prove the following result.

\begin{theorem}
\label{thm:intromain}
Let $\widetilde{G}=G\times F$ and~$N$ be the group and representation associated to a surface~$S$ of genus zero, and identify $\Bbbk\coloneqq\mathbb{C}[A^{\pm1},\lambda_i^{\pm1}]\cong\mathbb{C}[q^{\pm\frac{1}{2}},t_i^{\pm1}]$ by mapping $A\mapsto q^{-\frac{1}{2}}$ and $\lambda_i\mapsto t_i$ for all~$i$. Then there is a $\Bbbk$-algebra isomorphism 
\[
\Sk_{A,\boldsymbol{\lambda}}(S)\cong K^{\widetilde{G}_{\mathcal{O}}\rtimes\mathbb{C}^*}(\mathcal{R}_{G,N})
\]
from the skein algebra to the quantized Coulomb branch.
\end{theorem}

Theorem~\ref{thm:intromain} was first conjectured in~\cite{AS24a} where the cases $S=S_{0,3}$ and $S=S_{0,4}$ were proved. The algebra $\Sk_{A,\boldsymbol{\lambda}}(S_{0,4})$ is closely related to the spherical double affine Hecke~algebra~(DAHA) of type~$(C_1^\vee,C_1)$. The spherical DAHA has a faithful representation, the so called polynomial representation, and there is an analogous representation of the quantized Coulomb branch defined using the so called dressed minuscule monopole operators. In~\cite{AS24a} the $S=S_{0,4}$ case of Theorem~\ref{thm:intromain} was proved by comparing these representations.

A key ingredient in the proof of Theorem~\ref{thm:intromain} in the present paper is the construction of a ``polynomial representation'' for the skein algebra of a general surface. This is obtained by modifying a construction of Detcherry and Santharoubane~\cite{DS25}. This representation is explicitly computable using the fusion rules from~\cite{MV94}. Another key ingredient is the construction of special generators for the algebra $\mathscr{S}_{A,\boldsymbol{\lambda}}(S)$ obtained from~$\Sk_{A,\boldsymbol{\lambda}}(S)$ by extending scalars to rational functions in~$A$ and~$\lambda_i$. This is nontrivial and is established thanks to a crucial computation in Proposition~\ref{prop:generators}. We construct analogous generators for the algebra $\mathscr{A}(G,F,N)$ obtained from $K^{\widetilde{G}_{\mathcal{O}}\rtimes\mathbb{C}^*}(\mathcal{R}_{G,N})$ by extending scalars to rational functions in~$q^{\frac{1}{2}}$ and~$t_i$. By comparing the polynomial representation with the representation defined by monopole operators, we construct a $\mathbb{C}$-algebra isomorphism $\mathscr{S}_{A,\boldsymbol{\lambda}}(S)\cong\mathscr{A}(G,F,N)$. We define natural filtrations of these algebras, and by studying the induced map on associated graded algebras, we argue that this isomorphism restricts to an isomorphism of the integral forms as in Theorem~\ref{thm:intromain}.

\subsection{Categorification of skein algebras}

To categorify the skein algebra of a genus zero surface, we consider the bounded derived category $\Db\Coh^{\widetilde{G}_{\mathcal{O}}\rtimes\mathbb{C}^*}(\mathcal{R}_{G,N})$ of $\widetilde{G}_{\mathcal{O}}\rtimes\mathbb{C}^*$-equivariant coherent sheaves on the variety of triples~$\mathcal{R}_{G,N}$. We review the definition of this category in Section~\ref{sec:QuantizedCoulombBranchesAndCategorification} following the approach of Varagnolo and Vasserot~\cite{VV10}. Essentially by definition, the equivariant $K$-theory $K^{\widetilde{G}_{\mathcal{O}}\rtimes\mathbb{C}^*}(\mathcal{R}_{G,N})$ is the Grothendieck group of $\Db\Coh^{\widetilde{G}_{\mathcal{O}}\rtimes\mathbb{C}^*}(\mathcal{R}_{G,N})$. One can view the convolution product as a monoidal structure on this category, and so we obtain the following immediate consequence of Theorem~\ref{thm:intromain}.

\begin{corollary}
\label{cor:categorification}
Let $\widetilde{G}=G\times F$ and~$N$ be the group and representation associated to a surface~$S$ of genus zero. Then there is a $\mathbb{C}$-algebra isomorphism 
\[
\Sk_{A,\boldsymbol{\lambda}}(S)\cong K_0\left(\Db\Coh^{\widetilde{G}_{\mathcal{O}}\rtimes\mathbb{C}^*}(\mathcal{R}_{G,N})\right)
\]
from the skein algebra to the Grothendieck ring of the equivariant derived category.
\end{corollary}

It is worth noting that Cautis and Williams~\cite{CW23a,CW23b,CW23c} have developed a sophisticated framework for studying equivariant coherent sheaves on ind-schemes. In~\cite{CW23a} they constructed a t-structure on the equivariant derived category of the variety of triples. The heart of this t-structure is an abelian monoidal category called the category of Koszul-perverse coherent sheaves. It gives a mathematical realization of the category of half-BPS line defects of a four-dimensional $\mathcal{N}=2$ gauge theory~\cite{GMN13b}. Our results imply that the skein algebra is the Grothendieck~ring of this abelian monoidal category. The classes of simple objects in the category give a canonical basis, which we expect to be related to the bases studied by Thurston~\cite{T14}.

\subsection{Further directions}

We conclude this introduction by mentioning some further problems and directions for future research.

\subsubsection{Higher genus surfaces}

One obvious problem is to generalize our results to higher genus surfaces. However, one encounters new phenomena already in genus~one. Though it is possible to associate a group $\widetilde{G}=G\times F$ and representation~$N$ to any compact oriented surface~$S$ of genus~one with boundary, Conjecture~1.4 of~\cite{AS24a} states that the associated quantized Coulomb~branch $K^{\widetilde{G}_{\mathcal{O}}\rtimes\mathbb{C}^*}(\mathcal{R}_{G,N})$ is identified not with the full skein algebra $\Sk_{A,\boldsymbol{\lambda}}(S)$ but with a $\mathbb{Z}_2$-invariant subalgebra. It is therefore not immediately clear how to categorify the full skein algebra in this case, but we believe it should be possible to do this after enlarging the flavor symmetry group (see the remark following Proposition~5.5 in~\cite{AS24b}). We will revisit the genus~one case in a future publication.

For surfaces of genus $>1$, the situation is more mysterious. We believe that for many of these surfaces the representation of the gauge group is not of cotangent type, and so the technology of~\cite{BFN18} may not be sufficient to define the associated Coulomb branches. Recently, Teleman~\cite{T22} has proposed a generalization of the $K$-theoretic Coulomb branch and its categorification for theories of noncotangent type, but we have not studied whether this approach is applicable in our situation. Another possible approach to the higher genus case would be to develop a $K$-theoretic version of the gluing formula in Proposition~5.23 of~\cite{BFN19b}.

\subsubsection{Canonical bases and positivity}

As we have discussed, the Kauffman bracket skein algebra admits interesting canonical bases. Thanks to the work of many authors~\cite{B23,FG06,FG00,MQ23,Q22,T14}, these bases are now known to be positive, meaning that the product of two basis elements can be expanded as a $\mathbb{Z}_{\geq0}[A^{\pm1}]$-linear combination of basis elements. It was this positivity property that led to Thurston's question about categorification of the skein algebra~\cite{T14}.

We believe that our categorification will provide a valuable tool for understanding this positivity. Indeed, Cautis and Williams~\cite{CW23a} have used the derived category of equivariant coherent sheaves on the variety of triples to construct a canonical basis for the quantized Coulomb branch which is manifestly positive. We expect that this canonical basis coincides with one of the known canonical bases of the skein algebra under the isomorphism of Theorem~\ref{thm:intromain}. We hope to return to this in a future publication.

\subsubsection{Relation to other categorifications}

We believe that the results of the present paper should be related to the problem of categorifying cluster algebras. Indeed, there is a well known class of cluster algebras associated to surfaces with boundary~\cite{FST08}. The cluster algebra associated to such a surface is closely related to a space of twisted local systems on the surface~\cite{FG06}, and its quantization is closely related to the Kauffman bracket skein algebra~\cite{M16}. It would be interesting to use our methods to define monoidal categorifications in the sense of~\cite{HL10,KKKO18} for this class of cluster algebras. We refer to~\cite{CW23a,SS19} for more on the relationship between Coulomb branches and cluster algebras.

There should also be connections between our results and the work of Queffelec and Wedrich~\cite{QW21}. These authors considered a skein algebra called the $\mathfrak{gl}_2$-skein algebra, which is related to the Kauffman bracket skein algebra by a natural surjective algebra homomorphism. They used the theory of foams to construct a category whose Grothendieck group is naturally isomorphic to the underlying module of the $\mathfrak{gl}_2$-skein algebra. It would be interesting to relate their category to the one considered here.

\subsection*{Acknowledgements.}
We thank Haimiao~Chen and Renaud~Detcherry for answering questions and Kazuhiro~Hikami for generously sharing his unpublished work with us. HK was supported by KIAS Individual Grant MG047204 at the Korea Institute for Advanced Study. PS was supported by NSFC Grant 12225108 and by the New Cornerstone Science Foundation through the Xplorer~Prize.

\section{Genus zero skein algebras and their generators}
\label{sec:GenusZeroSkeinAlgebrasAndTheirGenerators}

In this section, we define the Kauffman bracket skein~algebra and prove a result about generation of the skein algebra of a genus zero surface with boundary.

\subsection{The skein module}
\label{sec:TheSkeinModule}

Let $M$ be any three-dimensional smooth manifold with (possibly empty) boundary. By a \emph{tangle} in $M$, we mean a compact unoriented one-dimensional submanifold $L\subset M$ with boundary such that $\partial L=L\cap\partial M$. A \emph{link} is a tangle $L$ with $\partial L=\emptyset$. A \emph{framing} for a tangle~$L\subset M$ is a continuous map $L\rightarrow TM$ that assigns a tangent vector $v_p\in T_pM\setminus T_pL$ to each point $p\in L$. A tangle in~$M$ equipped with a framing is said to be \emph{framed}, and an \emph{isotopy} of two framed tangles in~$M$ is an ambient isotopy of the underlying tangles in~$M$ that fixes the boundary pointwise and preserves the framings.

Let us write $\widehat{\mathcal{L}}_A(M)$ for the $\mathbb{C}[A^{\pm1}]$-module freely generated by isotopy classes of framed tangles in~$M$ and write $\mathcal{L}_A(M)$ for the submodule generated by isotopy classes of framed links. We will consider the relations in these modules depicted in Figure~\ref{fig:skeinrelations}. Each of the pictures appearing in these relations represents a framed tangle in~$M$. We depict only the portion of the tangle that lies in a small ball $B\subset M$ shaded in gray. For every point $p\in L\cap B$, the associated framing vector $v_p$ is assumed to be pointing toward the reader. The framed tangles appearing in a given relation are assumed to be identical in the complement of~$B$.

\begin{figure}[ht]
\begin{subfigure}{\textwidth}
\[
\begin{tikzpicture}[baseline=(current bounding box.center)]
\clip(-1.125,-1) rectangle (1.125,1);
\draw[fill=gray,opacity=0.2] (0,0) circle (1);
\draw[black, thick] (-0.707,-0.707) -- (0.707,0.707);
\draw[black, thick] (-0.707,0.707) -- (-0.1,0.1);
\draw[black, thick] (0.1,-0.1) -- (0.707,-0.707);
\draw[gray, opacity=0.2, thick] (0,0) circle (1);
\end{tikzpicture}
= A
\begin{tikzpicture}[baseline=(current bounding box.center)]
\clip(-1.125,-1) rectangle (1.125,1);
\draw[fill=gray,opacity=0.2] (0,0) circle (1);
\draw[black, thick] plot [smooth, tension=1] coordinates { (0.707,0.707) (0.25,0) (0.707,-0.707)};
\draw[black, thick] plot [smooth, tension=1] coordinates { (-0.707,0.707) (-0.25,0) (-0.707,-0.707)};
\draw[gray, opacity=0.2, thick] (0,0) circle (1);
\end{tikzpicture}
+ A^{-1}
\begin{tikzpicture}[baseline=(current bounding box.center)]
\clip(-1.125,-1) rectangle (1.125,1);
\draw[fill=gray,opacity=0.2] (0,0) circle (1);
\draw[black, thick] plot [smooth, tension=1] coordinates { (-0.707,0.707) (0,0.25) (0.707,0.707)};
\draw[black, thick] plot [smooth, tension=1] coordinates { (-0.707,-0.707) (0,-0.25) (0.707,-0.707)};
\draw[gray, opacity=0.2, thick] (0,0) circle (1);
\end{tikzpicture}
\]
\caption{\label{subfig:resolution}}
\end{subfigure}
\begin{subfigure}{\textwidth}
\[
\begin{tikzpicture}[baseline=(current bounding box.center)]
\clip(-1.125,-1) rectangle (1.125,1);
\draw[fill=gray,opacity=0.2] (0,0) circle (1);
\draw[black, thick] (0,0) circle (0.5);
\draw[gray, opacity=0.2, thick] (0,0) circle (1);
\end{tikzpicture}
= -(A^2+A^{-2})
\begin{tikzpicture}[baseline=(current bounding box.center)]
\clip(-1.25,-1) rectangle (1.25,1);
\draw[fill=gray,opacity=0.2] (0,0) circle (1);
\draw[gray, opacity=0.2, thick] (0,0) circle (1);
\end{tikzpicture}
\]
\caption{\label{subfig:unknot}}
\end{subfigure}
\caption{The Kauffman bracket skein relations.\label{fig:skeinrelations}}
\end{figure}
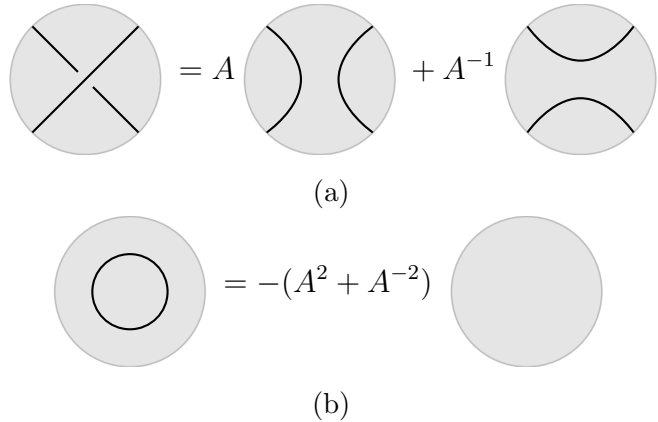

\begin{definition}
The \emph{extended (Kauffman bracket) skein module} $\widehat{\Sk}_A(M)$ is the quotient of~$\widehat{\mathcal{L}}_A(M)$ by the relations depicted in Figure~\ref{fig:skeinrelations}. The \emph{(Kauffman bracket) skein module} $\Sk_A(M)$ is the quotient of~$\mathcal{L}_A(M)$ by these relations.
\end{definition}

\subsection{The skein algebra}

We will be especially interested in the case where the three-manifold has the form $S\times[0,1]$ for an oriented smooth surface~$S$. In this case, we say that a framed~link $L\subset S\times[0,1]$ has the \emph{blackboard framing} if, for each point $p\in L$, the associated framing vector $v_p\in T_p(S\times[0,1])$ is parallel to the $[0,1]$ factor and points towards~1. Given any framed link~$L\subset S\times[0,1]$, we can find an isotopic framed link $L'\subset S\times[0,1]$ that has the blackboard framing. Consequently, we can describe the isotopy class of~$L$ by drawing the projection of~$L'$ to~$S$ and, if two points of~$L'$ project to the same point, indicating their order in the $[0,1]$~direction. In the following, we will always describe isotopy classes of framed links in this way.

The skein module $\Sk_A(S\times[0,1])$ in this case has the natural structure of a $\mathbb{C}[A^{\pm1}]$-algebra. Indeed, if $L_1$,~$L_2\subset S\times[0,1]$ are framed links, we can rescale their $[0,1]$-coordinates so that these links lie in~$S\times[0,\frac{1}{2})$ and~$S\times(\frac{1}{2},1]$, respectively. Then we define $L_2L_1$ to be the isotopy class of the union of the rescaled links in~$S\times[0,1]$, which depends only on the isotopy classes of~$L_1$ and~$L_2$. Extending $\mathbb{C}[A^{\pm1}]$-bilinearly, we get a product on~$\Sk_A(S\times[0,1])$ called the \emph{superposition product}.

\begin{definition}
The \emph{(Kauffman bracket) skein algebra} of~$S$, denoted $\Sk_A(S)$, is the skein module $\Sk_A(S\times[0,1])$ equipped with its superposition product.
\end{definition}

We will also need a modified version of the skein algebra when $S$ is a compact surface with nonempty boundary. Let $S=S_{g,n}$ be the surface obtained from a closed surface of genus~$g$ by removing $n$ open disks with pairwise disjoint closures, and let $\boldsymbol{\lambda}=(\lambda_1,\dots,\lambda_n)$ be an $n$-tuple of formal variables. Then we define the \emph{relative skein algebra} $\Sk_{A,\boldsymbol{\lambda}}(S)$ to be the quotient of $\Sk_A(S)\otimes_{\mathbb{C}[A^{\pm1}]}\mathbb{C}[A^{\pm1},\lambda_1^{\pm1},\dots,\lambda_n^{\pm1}]$ by all relations as in Figure~\ref{fig:puncturerelation}. Each picture in this relation represents a framed link in~$S\times[0,1]$. We depict only the portion of the link that projects to a neighborhood $U\subset S$ of the $i$th boundary component, shaded in gray, and the links on either side of the relation are assumed to be identical outside of~$U$.

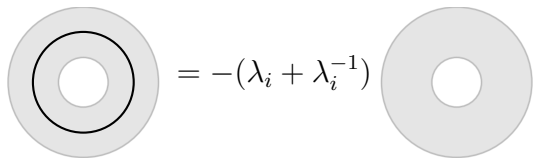
\begin{figure}[ht]
\[
\begin{tikzpicture}[baseline=(current bounding box.center)]
\clip(-1.125,-1) rectangle (1.125,1);
\draw[fill=gray,opacity=0.2,even odd rule] (0,0) circle (1) (0,0) circle (0.33);
\draw[black, thick] (0,0) circle (0.67);
\draw[gray, opacity=0.2, thick] (0,0) circle (1);
\draw[gray, opacity=0.2, thick] (0,0) circle (0.33);
\end{tikzpicture}
= -(\lambda_i+\lambda_i^{-1})
\begin{tikzpicture}[baseline=(current bounding box.center)]
\clip(-1.125,-1) rectangle (1.125,1);
\draw[fill=gray,opacity=0.2,even odd rule] (0,0) circle (1) (0,0) circle (0.33);
\draw[gray, opacity=0.2, thick] (0,0) circle (1);
\draw[gray, opacity=0.2, thick] (0,0) circle (0.33);
\end{tikzpicture}
\]
\caption{Additional relation associated to a boundary component.\label{fig:puncturerelation}}
\end{figure}

The skein algebra $\Sk_A(S)$ defined above becomes the algebra of regular functions on the $\mathrm{SL}_2$-character variety of~$S$ when we specialize to~$A=-1$. Similarly, the relative skein algebra $\Sk_{A,\boldsymbol{\lambda}}(S)$ becomes the algebra of regular functions on the relative $\mathrm{SL}_2$-character variety when we specialize to $A=-1$ and specialize the $\lambda_i$ to nonzero complex numbers~\cite{AS24a}. In the present paper, we will consider variants of these algebras obtained by localizing scalars. Namely, we will write $\mathscr{S}_A(S)$ for the algebra obtained from $\Sk_A(S)$ be extending scalars from $\mathbb{C}[A^{\pm1}]$ to $\mathbb{C}(A)$ and write $\mathscr{S}_{A,\boldsymbol{\lambda}}(S)$ for the algebra obtained from $\Sk_{A,\boldsymbol{\lambda}}(S)$ by extending scalars from $\mathbb{C}[A^{\pm1},\lambda_1^{\pm1},\dots,\lambda_n^{\pm1}]$ to $\mathbb{C}(A,\lambda_1,\dots,\lambda_n)$. Abusing terminology, we refer to $\mathscr{S}_A(S)$ and $\mathscr{S}_{A,\boldsymbol{\lambda}}(S)$ as the skein algebra and relative skein algebra, respectively.

\subsection{Generators for skein algebras}
\label{sec:GeneratorsForSkeinAlgebras}

In this subsection, we will describe a special generating set for the relative skein algebra $\mathscr{S}_{A,\boldsymbol{\lambda}}(S)$ when $S$ has genus zero.

\subsubsection{}

Let us write $S=S_{0,n+2}$ for a surface of genus zero with $n+2$ boundary components and number the boundary components of this surface by the integers $0,1,\dots,n+1$ as shown in Figure~\ref{subfig:puncturedsphere}. Let us also consider a collection of simple, pairwise nonintersecting paths $c_i$ on~$S$ where, for each integer~$i$ considered modulo~$n+2$, the path $c_i$ connects the $i$th and $(i+1)$st boundary components.

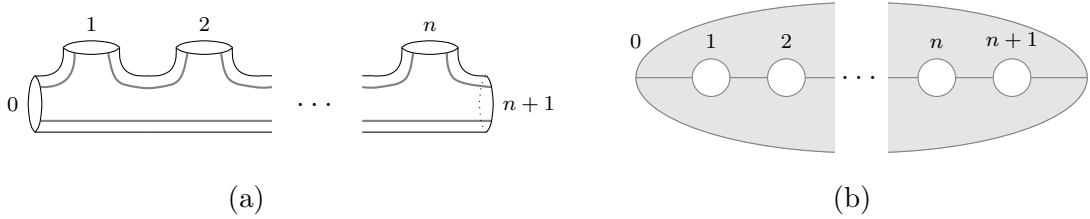
\begin{figure}[ht]
\begin{subfigure}{0.45\textwidth}
\begin{center}
\begin{tikzpicture}[scale=1.5]
\clip(-3.25,-0.5) rectangle (2.25,0.8);
\draw[black, thin] plot [smooth, tension=1] coordinates { (0.5,0.25) (0.7,0.3) (0.75,0.5)};
\draw[black, thin] plot [smooth, tension=1] coordinates { (1.5,0.25) (1.3,0.3) (1.25,0.5)};
\draw[black, thin] plot [smooth, tension=1] coordinates { (-0.5,0.25) (-0.7,0.3) (-0.75,0.5)};
\draw[black, thin] plot [smooth, tension=1] coordinates { (-1.5,0.25) (-1.3,0.3) (-1.25,0.5)};
\draw[black, thin] plot [smooth, tension=1] coordinates { (-1.5,0.25) (-1.7,0.3) (-1.75,0.5)};
\draw[black, thin] plot [smooth, tension=1] coordinates { (-2.5,0.25) (-2.3,0.3) (-2.25,0.5)};
\draw[gray, thick] plot [smooth, tension=1] coordinates { (-2.45,0.15) (-2.22,0.22) (-2.15,0.45)};
\draw[gray, thick] plot [smooth, tension=0.7] coordinates { (-1.85,0.45) (-1.78,0.22) (-1.6,0.15) (-1.4,0.15) (-1.22,0.22) (-1.15,0.45) };
\draw[gray, thick] plot [smooth, tension=0.7] coordinates { (-0.4,0.15) (-0.55,0.15) (-0.78,0.22) (-0.85,0.45)};
\draw[gray, thick] plot [smooth, tension=0.7] coordinates { (0.4,0.15) (0.55,0.15) (0.78,0.22) (0.85,0.45)};
\draw[gray, thick] plot [smooth, tension=0.7] coordinates { (1.55,0.15) (1.22,0.22) (1.15,0.45)};
\draw[gray, thick] (-2.45,-0.15) -- (-0.4,-0.15);
\draw[gray, thick] (0.4,-0.15) -- (1.55,-0.15);
\draw (-1.5,-0.25) -- (-2.5,-0.25);
\draw (-0.4,-0.25) -- (-1.5,-0.25);
\draw (0.4,-0.25) -- (1.5,-0.25);
\draw (-0.4,0.25) -- (-0.5,0.25);
\draw (0.4,0.25) -- (0.5,0.25);
\draw[black, thin] plot [smooth, tension=1] coordinates { (-1.75,0.5) (-1.875,0.45) (-2.125,0.45) (-2.25,0.5)};
\draw[black, thin] plot [smooth, tension=1] coordinates { (-1.75,0.5) (-1.875,0.55) (-2.125,0.55) (-2.25,0.5)};
\draw[black, thin] plot [smooth, tension=1] coordinates { (-0.75,0.5) (-0.875,0.45) (-1.125,0.45) (-1.25,0.5)};
\draw[black, thin] plot [smooth, tension=1] coordinates { (-0.75,0.5) (-0.875,0.55) (-1.125,0.55) (-1.25,0.5)};
\draw[black, thin] plot [smooth, tension=1] coordinates { (0.75,0.5) (0.875,0.45) (1.125,0.45) (1.25,0.5)};
\draw[black, thin] plot [smooth, tension=1] coordinates { (0.75,0.5) (0.875,0.55) (1.125,0.55) (1.25,0.5)};
\draw[black, thin] plot [smooth, tension=1] coordinates { (-2.5,-0.25) (-2.55,-0.125) (-2.55,0.125) (-2.5,0.25)};
\draw[black, thin] plot [smooth, tension=1] coordinates { (-2.5,-0.25) (-2.45,-0.125) (-2.45,0.125) (-2.5,0.25)};
\draw[black, thin, dotted] plot [smooth, tension=1] coordinates { (1.5,-0.25) (1.45,-0.125) (1.45,0.125) (1.5,0.25)};
\draw[black, thin] plot [smooth, tension=1] coordinates { (1.5,-0.25) (1.55,-0.125) (1.55,0.125) (1.5,0.25)};
\node at (0,0) {$\dots$};
\node[left] at (-2.55,0) {\tiny $0$};
\node[above] at (-2,0.55) {\tiny $1$};
\node[above] at (-1,0.55) {\tiny $2$};
\node[above] at (1,0.55) {\tiny $n$};
\node[right] at (1.55,0) {\tiny $n+1$};
\end{tikzpicture}
\end{center}
\caption{\label{subfig:puncturedsphere}}
\end{subfigure}
\begin{subfigure}{0.45\textwidth}
\begin{center}
\begin{tikzpicture}
\draw[fill=gray,opacity=0.2] (0,0) ellipse (3cm and 1cm);
\draw[gray] (0,0) ellipse (3cm and 1cm);
\draw[gray,fill=white] (-2,0) circle (0.25);
\draw[gray,fill=white] (-1,0) circle (0.25);
\draw[gray,fill=white] (1,0) circle (0.25);
\draw[gray,fill=white] (2,0) circle (0.25);
\draw[gray] (-3,0) -- (-2.25,0);
\draw[gray] (-1.75,0) -- (-1.25,0);
\draw[gray] (-0.75,0) -- (0.75,0);
\draw[gray] (1.25,0) -- (1.75,0);
\draw[gray] (2.25,0) -- (3,0);
\fill[white] (-0.35,-1.1) rectangle (0.35,1.1);
\node at (0,0) {$\dots$};
\node[above] at (-3,0.25) {\tiny $0$};
\node[above] at (-2,0.25) {\tiny $1$};
\node[above] at (-1,0.25) {\tiny $2$};
\node[above] at (1,0.25) {\tiny $n$};
\node[above] at (2,0.25) {\tiny $n+1$};
\end{tikzpicture}
\end{center}
\caption{\label{subfig:planarview}}
\end{subfigure}
\caption{A genus zero surface.\label{fig:puncturedsphere}}
\end{figure}

When doing calculations in the skein algebra, it is often useful to think of~$S$ as the surface obtained from a closed disk~$D$ in the plane by removing $n+1$ open disks whose closures are disjoint from each other and from the boundary of~$D$. This is illustrated in Figure~\ref{subfig:planarview}, which also shows the curves~$c_i$ and the numbering of the boundary components.

Note that $S$ is a union of two disks $D_+$ and $D_-$ glued along the curves~$c_i$ where $D_+$ denotes the top half and $D_-$ denotes the bottom half of Figure~\ref{subfig:planarview}. For each $i=1,\dots,n+1$, we write $Z_i^+=D_+\cap Z_i$ where $Z_i$ is the closure of a small tubular neighborhood of the $i$th boundary component of~$S$. For any subset $I=\{i_1,\dots,i_k\}\subset\{1,\dots,n+1\}$, we fix a subsurface $S_I\subset S$ in the interior of~$D$, such that $S_I$ is homeomorphic to~$S_{0,k+1}$, its boundary contains the boundary components of~$S$ numbered $i_1,\dots,i_k$, and its intersection with~$D_+$ is $\bigcup_{i\in I}Z_i^+$. We then consider the curve $\alpha_I=\partial S_I\setminus\partial S$. Examples of $\alpha_I$ include the curves in Figures~\ref{subfig:rho}, \ref{subfig:zeta}, and~\ref{subfig:mu} below. The following result is a consequence of the presentation of the skein algebra obtained in~\cite{C22}.

\begin{theorem}[\cite{C22}, Theorem~5.3]
\label{thm:Chen}
$\Sk_A(S)$ is generated by the curves $\alpha_I$ for $I\subset\{1,\dots,n+1\}$.
\end{theorem}

\subsubsection{}

We will now begin to prove a more refined result on generation of skein algebras. In the following, we say that a subset $I\subset\{1,\dots,n+1\}$ is \emph{complete} if it has the form $I=\{i,i+1,i+2\dots,j\}$ for some integers $i$,~$j$.

\begin{lemma}
\label{lem:basecase}
For any complete subset $I\subset\{1,\dots,n+1\}$, the element $\alpha_I$ is contained in the $\mathbb{C}(A)$-subalgebra of~$\mathscr{S}_A(S)$ generated by the curves $\sigma_{i,j}$, $\gamma_i$, and $\delta_i$ depicted in Figure~\ref{fig:generators}.
\end{lemma}

\begin{figure}[ht]
\begin{subfigure}{\textwidth}
\begin{center}
\begin{tikzpicture}[scale=1.5]
\clip(-3.25,-0.55) rectangle (3.25,0.8);
\draw[black, thin] plot [smooth, tension=1] coordinates { (0.5,0.25) (0.7,0.3) (0.75,0.5)};
\draw[black, thin] plot [smooth, tension=1] coordinates { (1.5,0.25) (1.3,0.3) (1.25,0.5)};
\draw[black, thin] plot [smooth, tension=1] coordinates { (-0.5,0.25) (-0.7,0.3) (-0.75,0.5)};
\draw[black, thin] plot [smooth, tension=1] coordinates { (-1.5,0.25) (-1.3,0.3) (-1.25,0.5)};
\draw[black, thin] plot [smooth, tension=1] coordinates { (1.5,0.25) (1.7,0.3) (1.75,0.5)};
\draw[black, thin] plot [smooth, tension=1] coordinates { (2.5,0.25) (2.3,0.3) (2.25,0.5)};
\draw[black, thin] plot [smooth, tension=1] coordinates { (-1.5,0.25) (-1.7,0.3) (-1.75,0.5)};
\draw[black, thin] plot [smooth, tension=1] coordinates { (-2.5,0.25) (-2.3,0.3) (-2.25,0.5)};
\draw (-1.5,-0.25) -- (-2.6,-0.25);
\draw (-0.4,-0.25) -- (-1.5,-0.25);
\draw (0.4,-0.25) -- (1.5,-0.25);
\draw (-2.5,0.25) -- (-2.6,0.25);
\draw (-0.4,0.25) -- (-0.5,0.25);
\draw (0.4,0.25) -- (0.5,0.25);
\draw (2.5,0.25) -- (2.6,0.25);
\draw (1.5,-0.25) -- (2.6,-0.25);
\draw[black, thin] plot [smooth, tension=1] coordinates { (-1.75,0.5) (-1.875,0.45) (-2.125,0.45) (-2.25,0.5)};
\draw[black, thin] plot [smooth, tension=1] coordinates { (-1.75,0.5) (-1.875,0.55) (-2.125,0.55) (-2.25,0.5)};
\draw[black, thin] plot [smooth, tension=1] coordinates { (-0.75,0.5) (-0.875,0.45) (-1.125,0.45) (-1.25,0.5)};
\draw[black, thin] plot [smooth, tension=1] coordinates { (-0.75,0.5) (-0.875,0.55) (-1.125,0.55) (-1.25,0.5)};
\draw[black, thin] plot [smooth, tension=1] coordinates { (0.75,0.5) (0.875,0.45) (1.125,0.45) (1.25,0.5)};
\draw[black, thin] plot [smooth, tension=1] coordinates { (0.75,0.5) (0.875,0.55) (1.125,0.55) (1.25,0.5)};
\draw[black, thin] plot [smooth, tension=1] coordinates { (1.75,0.5) (1.875,0.45) (2.125,0.45) (2.25,0.5)};
\draw[black, thin] plot [smooth, tension=1] coordinates { (1.75,0.5) (1.875,0.55) (2.125,0.55) (2.25,0.5)};
\draw[black, thick] plot [smooth, tension=1] coordinates { (-1.5,0.25) (-1.25,0) (-0.4,-0.12)};
\draw[black, thick, dotted] plot [smooth, tension=1] coordinates { (-1.5,0.25) (-1,0.1) (-0.4,0.05)};
\draw[black, thick] plot [smooth, tension=1] coordinates { (1.5,0.25) (1.25,0) (0.4,-0.12)};
\draw[black, thick, dotted] plot [smooth, tension=1] coordinates { (1.5,0.25) (1,0.1) (0.4,0.05)};
\node[below left] at (-1.5,0.25) {\tiny $\sigma_{i,j}$};
\node at (-3,0) {$\dots$};
\node at (0,0) {$\dots$};
\node at (3,0) {$\dots$};
\node[above] at (-2,0.55) {\tiny $i-1$};
\node[above] at (-1,0.55) {\tiny $i$};
\node[above] at (1,0.55) {\tiny $j$};
\node[above] at (2,0.55) {\tiny $j+1$};
\end{tikzpicture}
\end{center}
\caption{\label{subfig:sigma}}
\end{subfigure}
\begin{subfigure}{\textwidth}
\begin{center}
\begin{tikzpicture}[scale=1.5]
\clip(-3.25,-0.55) rectangle (2.25,0.8);
\draw[black, thin] plot [smooth, tension=1] coordinates { (0.5,0.25) (0.7,0.3) (0.75,0.5)};
\draw[black, thin] plot [smooth, tension=1] coordinates { (1.5,0.25) (1.3,0.3) (1.25,0.5)};
\draw[black, thin] plot [smooth, tension=1] coordinates { (-0.5,0.25) (-0.7,0.3) (-0.75,0.5)};
\draw[black, thin] plot [smooth, tension=1] coordinates { (-1.5,0.25) (-1.3,0.3) (-1.25,0.5)};
\draw[black, thin] plot [smooth, tension=1] coordinates { (-1.5,0.25) (-1.7,0.3) (-1.75,0.5)};
\draw[black, thin] plot [smooth, tension=1] coordinates { (-2.5,0.25) (-2.3,0.3) (-2.25,0.5)};
\draw (-1.5,-0.25) -- (-2.5,-0.25);
\draw (-0.4,-0.25) -- (-1.5,-0.25);
\draw (0.4,-0.25) -- (1.5,-0.25);
\draw (-0.4,0.25) -- (-0.5,0.25);
\draw (0.4,0.25) -- (0.5,0.25);
\draw[black, thin] plot [smooth, tension=1] coordinates { (-1.75,0.5) (-1.875,0.45) (-2.125,0.45) (-2.25,0.5)};
\draw[black, thin] plot [smooth, tension=1] coordinates { (-1.75,0.5) (-1.875,0.55) (-2.125,0.55) (-2.25,0.5)};
\draw[black, thin] plot [smooth, tension=1] coordinates { (-0.75,0.5) (-0.875,0.45) (-1.125,0.45) (-1.25,0.5)};
\draw[black, thin] plot [smooth, tension=1] coordinates { (-0.75,0.5) (-0.875,0.55) (-1.125,0.55) (-1.25,0.5)};
\draw[black, thin] plot [smooth, tension=1] coordinates { (0.75,0.5) (0.875,0.45) (1.125,0.45) (1.25,0.5)};
\draw[black, thin] plot [smooth, tension=1] coordinates { (0.75,0.5) (0.875,0.55) (1.125,0.55) (1.25,0.5)};
\draw[black, thin] plot [smooth, tension=1] coordinates { (-2.5,-0.25) (-2.55,-0.125) (-2.55,0.125) (-2.5,0.25)};
\draw[black, thin] plot [smooth, tension=1] coordinates { (-2.5,-0.25) (-2.45,-0.125) (-2.45,0.125) (-2.5,0.25)};
\draw[black, thick, dotted] plot [smooth, tension=1] coordinates { (-1.5,-0.25) (-1.55,-0.125) (-1.55,0.125) (-1.5,0.25)};
\draw[black, thick] plot [smooth, tension=1] coordinates { (-1.5,-0.25) (-1.45,-0.125) (-1.45,0.125) (-1.5,0.25)};
\draw[black, thick, dotted] plot [smooth, tension=1] coordinates { (-0.5,-0.25) (-0.55,-0.125) (-0.55,0.125) (-0.5,0.25)};
\draw[black, thick] plot [smooth, tension=1] coordinates { (-0.5,-0.25) (-0.45,-0.125) (-0.45,0.125) (-0.5,0.25)};
\draw[black, thick, dotted] plot [smooth, tension=1] coordinates { (0.5,-0.25) (0.45,-0.125) (0.45,0.125) (0.5,0.25)};
\draw[black, thick] plot [smooth, tension=1] coordinates { (0.5,-0.25) (0.55,-0.125) (0.55,0.125) (0.5,0.25)};
\draw[black, thin, dotted] plot [smooth, tension=1] coordinates { (1.5,-0.25) (1.45,-0.125) (1.45,0.125) (1.5,0.25)};
\draw[black, thin] plot [smooth, tension=1] coordinates { (1.5,-0.25) (1.55,-0.125) (1.55,0.125) (1.5,0.25)};
\node at (0,0) {$\dots$};
\node[below] at (-1.5,-0.25) {\tiny $\gamma_1$};
\node[below] at (-0.5,-0.25) {\tiny $\gamma_2$};
\node[below] at (0.5,-0.25) {\tiny $\gamma_{n-1}$};
\end{tikzpicture}
\end{center}
\caption{\label{subfig:gamma}}
\end{subfigure}
\begin{subfigure}{\textwidth}
\begin{center}
\begin{tikzpicture}[scale=1.5]
\clip(-3.25,-0.55) rectangle (2.25,0.8);
\draw[black, thin] plot [smooth, tension=1] coordinates { (0.5,0.25) (0.7,0.3) (0.75,0.5)};
\draw[black, thin] plot [smooth, tension=1] coordinates { (1.5,0.25) (1.3,0.3) (1.25,0.5)};
\draw[black, thin] plot [smooth, tension=1] coordinates { (-0.5,0.25) (-0.7,0.3) (-0.75,0.5)};
\draw[black, thin] plot [smooth, tension=1] coordinates { (-1.5,0.25) (-1.3,0.3) (-1.25,0.5)};
\draw[black, thin] plot [smooth, tension=1] coordinates { (-1.5,0.25) (-1.7,0.3) (-1.75,0.5)};
\draw[black, thin] plot [smooth, tension=1] coordinates { (-2.5,0.25) (-2.3,0.3) (-2.25,0.5)};
\draw (-1.5,-0.25) -- (-2.5,-0.25);
\draw (-0.4,-0.25) -- (-1.5,-0.25);
\draw (0.4,-0.25) -- (1.5,-0.25);
\draw (-0.4,0.25) -- (-0.5,0.25);
\draw (0.4,0.25) -- (0.5,0.25);
\draw[black, thin] plot [smooth, tension=1] coordinates { (-1.75,0.5) (-1.875,0.45) (-2.125,0.45) (-2.25,0.5)};
\draw[black, thin] plot [smooth, tension=1] coordinates { (-1.75,0.5) (-1.875,0.55) (-2.125,0.55) (-2.25,0.5)};
\draw[black, thick] plot [smooth, tension=1] coordinates { (-1.75,0.35) (-1.875,0.3) (-2.125,0.3) (-2.25,0.35)};
\draw[black, thick, dotted] plot [smooth, tension=1] coordinates { (-1.75,0.35) (-1.875,0.4) (-2.125,0.4) (-2.25,0.35)};
\draw[black, thin] plot [smooth, tension=1] coordinates { (-0.75,0.5) (-0.875,0.45) (-1.125,0.45) (-1.25,0.5)};
\draw[black, thin] plot [smooth, tension=1] coordinates { (-0.75,0.5) (-0.875,0.55) (-1.125,0.55) (-1.25,0.5)};
\draw[black, thick] plot [smooth, tension=1] coordinates { (-0.75,0.35) (-0.875,0.3) (-1.125,0.3) (-1.25,0.35)};
\draw[black, thick, dotted] plot [smooth, tension=1] coordinates { (-0.75,0.35) (-0.875,0.4) (-1.125,0.4) (-1.25,0.35)};
\draw[black, thin] plot [smooth, tension=1] coordinates { (0.75,0.5) (0.875,0.45) (1.125,0.45) (1.25,0.5)};
\draw[black, thin] plot [smooth, tension=1] coordinates { (0.75,0.5) (0.875,0.55) (1.125,0.55) (1.25,0.5)};
\draw[black, thick] plot [smooth, tension=1] coordinates { (0.75,0.35) (0.875,0.3) (1.125,0.3) (1.25,0.35)};
\draw[black, thick, dotted] plot [smooth, tension=1] coordinates { (0.75,0.35) (0.875,0.4) (1.125,0.4) (1.25,0.35)};
\draw[black, thin] plot [smooth, tension=1] coordinates { (-2.5,-0.25) (-2.55,-0.125) (-2.55,0.125) (-2.5,0.25)};
\draw[black, thin] plot [smooth, tension=1] coordinates { (-2.5,-0.25) (-2.45,-0.125) (-2.45,0.125) (-2.5,0.25)};
\draw[black, thin, dotted] plot [smooth, tension=1] coordinates { (1.5,-0.25) (1.45,-0.125) (1.45,0.125) (1.5,0.25)};
\draw[black, thin] plot [smooth, tension=1] coordinates { (1.5,-0.25) (1.55,-0.125) (1.55,0.125) (1.5,0.25)};
\draw[black, thick] plot [smooth, tension=1] coordinates { (1.35,-0.25) (1.4,-0.126) (1.4,0.127) (1.35,0.258)};
\draw[black, thick, dotted] plot [smooth, tension=1] coordinates { (1.35,-0.25) (1.3,-0.126) (1.3,0.127) (1.35,0.258)};
\draw[black, thick] plot [smooth, tension=1] coordinates { (-2.35,-0.25) (-2.3,-0.126) (-2.3,0.127) (-2.35,0.258)};
\draw[black, thick, dotted] plot [smooth, tension=1] coordinates { (-2.35,-0.25) (-2.4,-0.126) (-2.4,0.127) (-2.35,0.258)};
\node at (0,0) {$\dots$};
\node[below] at (-2.35,-0.25) {\tiny $\delta_0$};
\node[below] at (-2,0.35) {\tiny $\delta_1$};
\node[below] at (-1,0.35) {\tiny $\delta_2$};
\node[below] at (1,0.35) {\tiny $\delta_n$};
\node[below] at (1.35,-0.25) {\tiny $\delta_{n+1}$};
\end{tikzpicture}
\end{center}
\caption{\label{subfig:delta}}
\end{subfigure}
\caption{Curves on a genus zero surface, $1\leq i, j\leq n$.\label{fig:generators}}
\end{figure}
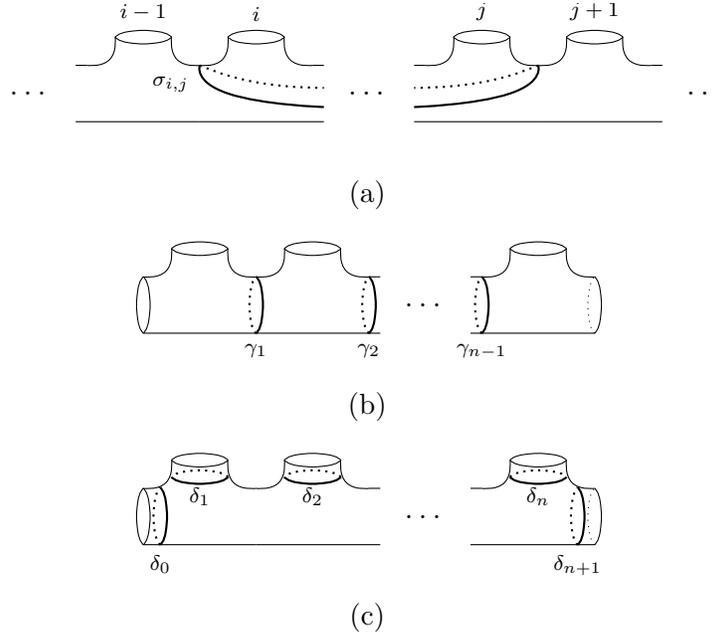

\begin{proof}
For any $i=1,\dots,n+1$, we have $\alpha_{\{i\}}=\delta_i$, and therefore the lemma is true when $|I|=1$. On the other hand, for any $2\leq r\leq n+1$, we have 
\begin{align*}
\alpha_{\{1,2,\dots,r\}} &= \sigma_{1,r}, \\
\alpha_{\{2,3,\dots,r+1\}} &= \sigma_{2,r+1}, \\
&\vdots \\
\alpha_{\{n+2-r,\dots,n+1\}} &= \sigma_{n+2-r,n+1}.
\end{align*}
(We note that the last of these elements can also be identified with $\gamma_{n+1-r}$ when $r<n+1$ and with $\delta_0$ when $r=n+1$.) Hence the lemma is true when $|I|=r$.
\end{proof}

\begin{lemma}
\label{lem:induction}
If $I_1,\dots,I_r\subset\{1,\dots,n+1\}$ are complete subsets, then $\alpha_{I_1\cup\dots\cup I_r}$ is contained in the $\mathbb{C}(A)$-subalgebra of~$\mathscr{S}_A(S)$ generated by the curves $\sigma_{i,j}$, $\gamma_i$, and~$\delta_i$ in Figure~\ref{fig:generators}.
\end{lemma}

\begin{proof}
Let us write $\mathcal{A}\subset\mathscr{S}_A(S)$ for the $\mathbb{C}(A)$-subalgebra of~$\mathscr{S}_A(S)$ generated by the curves illustrated in Figure~\ref{fig:generators}. If $I_1\subset\{1,\dots,n+1\}$ is a complete set, then $\alpha_{I_1}\in\mathcal{A}$ by Lemma~\ref{lem:basecase}. Assume inductively that the lemma is true for all $r\leq k$, and let $I_1,\dots,I_{k+1}\subset\{1,\dots,n+1\}$ be complete subsets. We may assume without loss of generality that the~$I_t$ are pairwise disjoint and $i_1<\dots<i_{k+1}$ when $i_t\in I_t$ for $1\leq t\leq k+1$. For convenience, let us write $I\coloneqq I_1\cup\dots\cup I_k$. If $I_k\cup I_{k+1}$ is complete, then $I\cup I_{k+1}=I_1\cup\dots I_{k-1}\cup(I_k\cup I_{k+1})$ is a union of $k$ complete subsets, and hence, by the inductive assumption, $\alpha_{I_1\cup\dots\cup I_{k+1}}\in\mathcal{A}$. If $I_k\cup I_{k+1}$ is not complete, let 
\[
J=\{j\in\mathbb{Z}:i_k<j<i_{k+1}\text{ for all $i_k\in I_k$ and $i_{k+1}\in I_{k+1}$}\}.
\]
One can check using the skein relation that 
\[
\alpha_{J\cup I_{k+1}}\cdot\alpha_{I\cup J}=A^2\alpha_{I\cup I_{k+1}}+\alpha_J\cdot\alpha_{I\cup J\cup I_{k+1}}+\alpha_I\cdot\alpha_{I_{k+1}}+A^{-2}\beta
\]
for some loop~$\beta$ on~$S$, and similarly 
\[
\alpha_{I\cup J}\cdot\alpha_{J\cup I_{k+1}}=A^2\beta+\alpha_I\cdot\alpha_{I_{k+1}}+\alpha_J\cdot\alpha_{I\cup J\cup I_{k+1}}+A^{-2}\alpha_{I\cup I_{k+1}}.
\]
These relations imply 
\begin{align*}
A^2\alpha_{J\cup I_{k+1}}&\cdot\alpha_{I\cup J} - A^{-2}\alpha_{I\cup J}\cdot\alpha_{J\cup I_{k+1}} \\
&= (A^4-A^{-4})\alpha_{I\cup I_{k+1}}+(A^2-A^{-2})\left(\alpha_J\cdot\alpha_{I\cup J\cup I_{k+1}}+\alpha_I\cdot\alpha_{I_{k+1}}\right).
\end{align*}
Since the subsets $J$, $I_{k+1}$, and $J\cup I_{k+1}$ are complete, we have $\alpha_J$, $\alpha_{I_{k+1}}$, $\alpha_{J\cup I_{k+1}}\in\mathcal{A}$. By the inductive assumption and the fact that the subsets $I_k\cup J$ and $I_k\cup I_{k+1}\cup J$ are complete, we also have $\alpha_{I\cup J}$, \ $\alpha_{I\cup J\cup I_{k+1}}$, \ $\alpha_I\in\mathcal{A}$. Hence $\alpha_{I_1\cup\dots\cup I_{k+1}}=\alpha_{I\cup I_{k+1}}\in\mathcal{A}$ and the lemma follows by induction.
\end{proof}

\begin{proposition}
\label{prop:firstgenerators}
$\mathscr{S}_A(S)$ is generated as a $\mathbb{C}(A)$-algebra by the curves $\sigma_{i,j}$, $\gamma_i$, and~$\delta_i$ in Figure~\ref{fig:generators}.
\end{proposition}

\begin{proof}
Since any subset of $\{1,\dots,n+1\}$ is a union of complete subsets, this follows from Lemma~\ref{lem:induction} and Theorem~\ref{thm:Chen}.
\end{proof}

\subsubsection{}

The following is the main result we need about generation of skein algebras. The proof uses a method that we learned from an unpublished preprint of Hikami~\cite{H24}, in which he considered the case $S=S_{0,5}$.

\begin{proposition}
\label{prop:generators}
$\mathscr{S}_{A,\boldsymbol{\lambda}}(S)$ is generated as a $\mathbb{C}(A,\lambda_0,\dots,\lambda_{n+1})$-algebra by the curves $\sigma_{i,i+1}$ and~$\gamma_i$ in Figure~\ref{fig:generators}.
\end{proposition}

\begin{proof}
By Proposition~\ref{prop:firstgenerators} and the relation in Figure~\ref{fig:puncturerelation}, we know that $\mathscr{S}_{A,\boldsymbol{\lambda}}(S)$ is generated as a $\mathbb{C}(A,\lambda_0,\dots,\lambda_{n+1})$-algebra by the curves $\sigma_{i,j}$ and~$\gamma_i$. Thus we only need to show that each curve $\sigma_{i,j}$ is contained in the $\mathbb{C}(A,\lambda_0,\dots,\lambda_{n+1})$-subalgebra generated by curves of the form $\sigma_{i,i+1}$ and~$\gamma_i$. Given a curve~$\sigma_{i,j}$ with $j-i=k-1$, we can cut~$S$ along the loops $\gamma_{i-1}$ and $\gamma_j$ to get a subsurface of~$S$ homeomorphic to~$S_{0,k+2}$, and the curve $\sigma_{i,j}$ is identified with the curve~$\sigma_{1,k}$ on this subsurface. Thus it suffices to show that, for any $k\geq1$, the curve $\sigma_{1,k}\in\Sk_{A,\boldsymbol{\lambda}}(S_{0,k+2})$ is contained in the $\mathbb{C}(A,\lambda_1,\dots,\lambda_k)$-subalgebra generated by the~$\sigma_{i,i+1}$ and~$\gamma_i$.

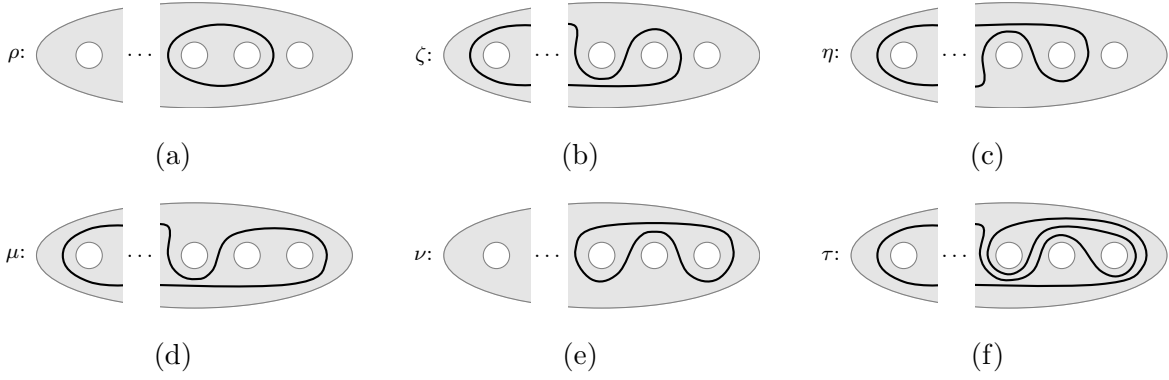
\begin{figure}[ht]
\begin{subfigure}{0.3\textwidth}
\begin{center}
\begin{tikzpicture}[scale=0.7]
\clip(-3.75,-1) rectangle (3,1.5);
\draw[fill=gray,opacity=0.2] (0,0) ellipse (3cm and 1cm);
\draw[gray] (0,0) ellipse (3cm and 1cm);
\draw[gray,fill=white] (-2,0) circle (0.25);
\draw[gray,fill=white] (0,0) circle (0.25);
\draw[gray,fill=white] (1,0) circle (0.25);
\draw[gray,fill=white] (2,0) circle (0.25);
\draw[black, thick] plot [smooth cycle, tension=0.8] coordinates { (0,0.5) (-0.5,0) (0,-0.5) (1,-0.5) (1.5,0) (1,0.5) };
\fill[white] (-1.35,-1.1) rectangle (-0.65,1.1);
\node at (-1,0) {\tiny $\dots$};
\node[left] at (-3,0) {\tiny $\rho$:};
\end{tikzpicture}
\end{center}
\caption{\label{subfig:rho}}
\end{subfigure}
\begin{subfigure}{0.3\textwidth}
\begin{center}
\begin{tikzpicture}[scale=0.7]
\clip(-3.75,-1) rectangle (3,1.5);
\draw[fill=gray,opacity=0.2] (0,0) ellipse (3cm and 1cm);
\draw[gray] (0,0) ellipse (3cm and 1cm);
\draw[gray,fill=white] (-2,0) circle (0.25);
\draw[gray,fill=white] (0,0) circle (0.25);
\draw[gray,fill=white] (1,0) circle (0.25);
\draw[gray,fill=white] (2,0) circle (0.25);
\draw[black, thick] plot [smooth, tension=0.8] coordinates { (-1,-0.55) (-2,-0.5) (-2.5,0) (-2,0.5) (-1,0.55) };
\draw[black, thick] plot [smooth, tension=0.7] coordinates { (-1,0.55) (-0.5,0.55) (-0.5,0) (-0.33,-0.33) (0,-0.45) (0.33,-0.3) (0.65,0.35) (1.2,0.45) (1.5,-0.1) (1,-0.55) (-1,-0.55) };
\fill[white] (-1.35,-1.1) rectangle (-0.65,1.1);
\node at (-1,0) {\tiny $\dots$};
\node[left] at (-3,0) {\tiny $\zeta$:};
\end{tikzpicture}
\end{center}
\caption{\label{subfig:zeta}}
\end{subfigure}
\begin{subfigure}{0.3\textwidth}
\begin{center}
\begin{tikzpicture}[scale=0.7]
\clip(-3.75,-1) rectangle (3,1.5);
\draw[fill=gray,opacity=0.2] (0,0) ellipse (3cm and 1cm);
\draw[gray] (0,0) ellipse (3cm and 1cm);
\draw[gray,fill=white] (-2,0) circle (0.25);
\draw[gray,fill=white] (0,0) circle (0.25);
\draw[gray,fill=white] (1,0) circle (0.25);
\draw[gray,fill=white] (2,0) circle (0.25);
\draw[black, thick] plot [smooth, tension=0.8] coordinates { (-1,-0.55) (-2,-0.5) (-2.5,0) (-2,0.5) (-1,0.55) };
\draw[black, thick] plot [smooth, tension=0.7] coordinates { (-1,-0.55) (-0.5,-0.55) (-0.5,0) (-0.33,0.33) (0,0.45) (0.33,0.3) (0.65,-0.35) (1.2,-0.45) (1.5,0.1) (1,0.55) (-1,0.55) };
\fill[white] (-1.35,-1.1) rectangle (-0.65,1.1);
\node at (-1,0) {\tiny $\dots$};
\node[left] at (-3,0) {\tiny $\eta$:};
\end{tikzpicture}
\end{center}
\caption{\label{subfig:eta}}
\end{subfigure}
\begin{subfigure}{0.3\textwidth}
\begin{center}
\begin{tikzpicture}[scale=0.7]
\clip(-3.75,-1) rectangle (3,1.5);
\draw[fill=gray,opacity=0.2] (0,0) ellipse (3cm and 1cm);
\draw[gray] (0,0) ellipse (3cm and 1cm);
\draw[gray,fill=white] (-2,0) circle (0.25);
\draw[gray,fill=white] (0,0) circle (0.25);
\draw[gray,fill=white] (1,0) circle (0.25);
\draw[gray,fill=white] (2,0) circle (0.25);
\draw[black, thick] plot [smooth, tension=0.8] coordinates { (-1,-0.55) (-2,-0.5) (-2.5,0) (-2,0.5) (-1,0.55) };
\draw[black, thick] plot [smooth, tension=0.7] coordinates { (-1,0.55) (-0.5,0.55) (-0.5,0) (-0.33,-0.33) (0,-0.45) (0.33,-0.3) (0.75,0.4) (2.1,0.45) (2.5,-0.1) (1.8,-0.55) (-1,-0.55) };
\fill[white] (-1.35,-1.1) rectangle (-0.65,1.1);
\node at (-1,0) {\tiny $\dots$};
\node[left] at (-3,0) {\tiny $\mu$:};
\end{tikzpicture}
\end{center}
\caption{\label{subfig:mu}}
\end{subfigure}
\begin{subfigure}{0.3\textwidth}
\begin{center}
\begin{tikzpicture}[scale=0.7]
\clip(-3.75,-1) rectangle (3,1.5);
\draw[fill=gray,opacity=0.2] (0,0) ellipse (3cm and 1cm);
\draw[gray] (0,0) ellipse (3cm and 1cm);
\draw[gray,fill=white] (-2,0) circle (0.25);
\draw[gray,fill=white] (0,0) circle (0.25);
\draw[gray,fill=white] (1,0) circle (0.25);
\draw[gray,fill=white] (2,0) circle (0.25);
\draw[black, thick] plot [smooth cycle, tension=0.7] coordinates { (-0.5,0.1) (-0.2,-0.45) (0.3,-0.35) (0.67,0.3) (1,0.45) (1.33,0.3) (1.65,-0.35) (2.2,-0.45) (2.5,0.1) (2,0.55) (0,0.55) };
\fill[white] (-1.35,-1.1) rectangle (-0.65,1.1);
\node at (-1,0) {\tiny $\dots$};
\node[left] at (-3,0) {\tiny $\nu$:};
\end{tikzpicture}
\end{center}
\caption{\label{subfig:nu}}
\end{subfigure}
\begin{subfigure}{0.3\textwidth}
\begin{center}
\begin{tikzpicture}[scale=0.7]
\clip(-3.75,-1) rectangle (3,1.5);
\draw[fill=gray,opacity=0.2] (0,0) ellipse (3cm and 1cm);
\draw[gray] (0,0) ellipse (3cm and 1cm);
\draw[gray,fill=white] (-2,0) circle (0.25);
\draw[gray,fill=white] (0,0) circle (0.25);
\draw[gray,fill=white] (1,0) circle (0.25);
\draw[gray,fill=white] (2,0) circle (0.25);
\draw[black, thick] plot [smooth, tension=0.8] coordinates { (-1,-0.55) (-2,-0.5) (-2.5,0) (-2,0.5) (-1,0.55) };
\draw[black, thick] plot [smooth, tension=0.7] coordinates { (-1,0.55) (-0.5,0.55) (-0.55,0) (-0.35,-0.35) (0,-0.47) (0.45,-0.3) (0.75,0.3) (1.25,0.3) (1.75,-0.3) (2.2,-0.35) (2.4,0.05) (2,0.4) (0.75,0.5) (0.25,-0.25) (-0.2,-0.3) (-0.4,0.1) (0,0.55) (1.25,0.7) (2.35,0.45) (2.55,-0.2) (1.7,-0.55) (-1,-0.55) };
\fill[white] (-1.35,-1.1) rectangle (-0.65,1.1);
\node at (-1,0) {\tiny $\dots$};
\node[left] at (-3,0) {\tiny $\tau$:};
\end{tikzpicture}
\end{center}
\caption{\label{subfig:tau}}
\end{subfigure}
\caption{Curves in the proof of Proposition~\ref{prop:generators}.\label{fig:skeinsplanarview}}
\end{figure}

This is true if $k=1$, since in that case $\sigma_{1,k}=-\lambda_1-\lambda_1^{-1}$ by the relation in Figure~\ref{fig:puncturerelation}, and it is obviously true if $k=2$. Assume inductively that the statement is true for all $k\leq K$. Let us write $\sigma\coloneqq\sigma_{1,K+1}$, $\sigma'=\sigma_{1,K}$, $\sigma''=\sigma_{1,K-1}\in\mathscr{S}_{A,\boldsymbol{\lambda}}(S_{0,K+3})$ and consider the curves on~$S_{0,K+3}$ depicted in Figure~\ref{fig:skeinsplanarview}. By the relation in Figure~\ref{subfig:resolution}, we have the identities 
\begin{align*}
\sigma'\rho &= A^{-2}\zeta+\delta_K\sigma+\delta_{K+1}\sigma''+A^2\eta, \\
\mu\nu &= A^{-2}\tau+\delta_0\delta_{K+2}+\delta_K\zeta+A^2\sigma, 
\end{align*}
and similarly, 
\begin{align*}
\rho\sigma' &= A^{-2}\eta+\delta_{K+1}\sigma''+\delta_K\sigma+A^2\zeta, \\
\nu\mu &= A^{-2}\sigma+\delta_K\zeta+\delta_0\delta_{K+2}+A^2\tau.
\end{align*}
We can combine these into a single matrix equation 
\begin{equation}
\label{eqn:skeinmatrix}
\begin{pmatrix}
\sigma'\rho-\delta_{K+1}\sigma'' \\
\rho\sigma'-\delta_{K+1}\sigma'' \\
\mu\nu-\delta_0\delta_{K+2} \\
\nu\mu-\delta_0\delta_{K+2}
\end{pmatrix}
=
\underbrace{
\begin{pmatrix}
A^{-2} & A^2 & \delta_K & 0 \\
A^2 & A^{-2} & \delta_K & 0 \\
\delta_K & 0 & A^2 & A^{-2} \\
\delta_K & 0 & A^{-2} & A^2
\end{pmatrix}}_P
\begin{pmatrix}
\zeta \\
\eta \\
\sigma \\
\tau
\end{pmatrix}.
\end{equation}
The determinant of~$P$ is not identically zero, so there is an inverse~$P^{-1}$ with coefficients in $\mathbb{C}(A,\lambda_1,\dots,\lambda_{k+1})$. In particular, we can express $\sigma$ as a polynomial in~$\mu$, $\nu$, $\sigma'$, $\sigma''$, $\rho$, and the~$\delta_i$. The elements $\rho$, and~$\delta_i$ are obviously contained in the subalgebra $\mathcal{A}\subset\mathscr{S}_{A,\boldsymbol{\lambda}}(S_{0,K+3})$ generated by the $\sigma_{i,i+1}$ and~$\gamma_i$. The inductive assumption implies $\sigma'$,~$\sigma''\in\mathcal{A}$. Lemma~\ref{lem:induction} implies $\mu\in\mathcal{A}$, since it has the form $\alpha_I$ for a subset $I\subset\{1,\dots,K+1\}$. Finally, by applying the diffeomorphism of~$S_{0,K+3}$ given by rotation through 180~degrees and arguing as in Lemma~\ref{lem:induction}, we see that $\nu\in\mathcal{A}$. Hence $\sigma=\sigma_{1,K+1}\in\mathcal{A}$, and the proposition follows by induction.
\end{proof}

\subsection{The associated graded algebra}
\label{sec:TheAssociatedGradedAlgebraSkein}

In the proof of our main result, we will consider a filtration on the skein algebra of $S=S_{0,n+2}$. To define it, let $[L]$ be the isotopy class of a framed link $L\subset S\times[0,1]$, and let $d_i=\min|\{L'\cap(\gamma_i\times[0,1]):L'\in[L]\}|$ where $\gamma_i$ is the curve in Figure~\ref{fig:generators}. Then we define the \emph{degree} of~$[L]$ to be the vector $\deg([L])=(d_i)_{i=1}^{n-1}\in2\mathbb{Z}_{\geq0}^{n-1}$.

Given elements $\mathbf{d}=(d_i)$ and $\mathbf{e}=(e_i)$ of $2\mathbb{Z}_{\geq0}^{n-1}$, we will write $\mathbf{d}\leq\mathbf{e}$ if we have $d_i\leq e_i$ for every~$i$. We will write $\mathbf{d}<\mathbf{e}$ if $\mathbf{d}\leq\mathbf{e}$ and $\mathbf{d}\neq\mathbf{e}$. The relation $\leq$ defines a partial order on the set~$2\mathbb{Z}_{\geq0}^{n-1}$. For any $\mathbf{d}\in2\mathbb{Z}_{\geq0}^{n-1}$, we will write $\Sk_{A,\boldsymbol{\lambda}}(S)_{\leq\mathbf{d}}\subset\Sk_{A,\boldsymbol{\lambda}}(S)$ for the $\mathbb{C}[A^{\pm1},\lambda_0^{\pm1},\dots,\lambda_{n+1}^{\pm1}]$-submodule spanned by isotopy classes of framed links of degree~$\leq\mathbf{d}$. We will write $\Sk_{A,\boldsymbol{\lambda}}(S)_{<\mathbf{d}}$ for the $\mathbb{C}[A^{\pm1},\lambda_0^{\pm1},\dots,\lambda_{n+1}^{\pm1}]$-submodule spanned by isotopy classes of framed links of degree $<\mathbf{d}$. The submodules $\Sk_{A,\boldsymbol{\lambda}}(S)_{\leq\mathbf{d}}$ define a filtration of~$\Sk_{A,\boldsymbol{\lambda}}(S)$, and the superposition product maps 
\[
\Sk_{A,\boldsymbol{\lambda}}(S)_{\leq\mathbf{d}}\times\Sk_{A,\boldsymbol{\lambda}}(S)_{\leq\mathbf{e}}\rightarrow\Sk_{A,\boldsymbol{\lambda}}(S)_{\leq\mathbf{d}+\mathbf{e}}
\]
so that $\Sk_{A,\boldsymbol{\lambda}}(S)$ becomes a filtered algebra. We will write 
\[
\gr^\bullet\Sk_{A,\boldsymbol{\lambda}}(S)\coloneqq\bigoplus_{\mathbf{d}\in2\mathbb{Z}_{\geq0}^{n-1}}\Sk_{A,\boldsymbol{\lambda}}(S)_{\leq\mathbf{d}}/\Sk_{A,\boldsymbol{\lambda}}(S)_{<\mathbf{d}}
\]
for the associated graded algebra.

By a \emph{multicurve} $C$ on a surface~$S$, we mean the image of a smooth embedding $(S^1)^{\sqcup k}\rightarrow S$ such that no component of~$C$ is nullhomotopic or homotopic to a boundary component. Note that we can view a multicurve on~$S$ as an element of~$\Sk_{A,\boldsymbol{\lambda}}(S)$. Moreover, by applying skein relations, we can write any element of~$\Sk_{A,\boldsymbol{\lambda}}(S)$ as a $\mathbb{C}[A^{\pm1},\lambda_0^{\pm1},\dots,\lambda_{n+1}^{\pm1}]$-linear combination of multicurves. From this we see that the degree~$\mathbf{d}$ homogeneous component $\gr^{\mathbf{d}}\Sk_{A,\boldsymbol{\lambda}}(S)\subset\gr^\bullet\Sk_{A,\boldsymbol{\lambda}}(S)$ has a $\mathbb{C}[A^{\pm1},\lambda_0^{\pm1},\dots,\lambda_{n+1}^{\pm1}]$-module basis consisting of multicurves of degree~$\mathbf{d}$ in~$\Sk_{A,\boldsymbol{\lambda}}(S)$.

Let us write $\mathcal{C}(S)$ for the set of isotopy classes of multicurves on a surface~$S$. Recall that a \emph{pants decomposition} of~$S$ is a collection of disjoint simple closed curves on~$S$ such that any component of the complement of these curves in~$S\setminus\partial S$ is homeomorphic to the interior of~$S_{0,3}$. After choosing a pants decomposition~$\mathcal{P}$ of~$S$ and some additional data described in Section~1.2 of~\cite{PH92}, a result known as Dehn's theorem provides an injective map 
\begin{equation}
\label{eqn:DehnThurston}
I:\mathcal{C}(S)\hookrightarrow(\mathbb{Z}_{\geq0}\times\mathbb{Z})^{\mathcal{P}}.
\end{equation}
Let $C\in\mathcal{C}(S)$ be an isotopy class of multicurves on~$S$, and write $I(C)=(l_\gamma,t_\gamma)_{\gamma\in\mathcal{P}}$ for its image under this embedding. Then the integers $l_\gamma$ and $t_\gamma$ are called the \emph{Dehn--Thurston coordinates} of~$C$. By definition, $l_\gamma$ is the geometric intersection number of~$C$ with~$\gamma$, while $t_\gamma$ measures how much the multicurve ``twists'' around~$\gamma$.

\begin{lemma}
\label{lem:specialmulticurve}
Take notation as in Figure~\ref{fig:generators}, and consider the pants decomposition $\mathcal{P}=\{\gamma_1,\dots,\gamma_{n-1}\}$ of $S_{0,n+2}$. Then for any $\mathbf{d}=(d_i)\in2\mathbb{Z}_{\geq0}^{n-1}$, the homogeneous component $\gr^{\mathbf{d}}\Sk_{A,\boldsymbol{\lambda}}(S)$ has a basis consisting of multicurves $C$ such that $I(C)=(d_i,t_i)_{i=1}^{n-1}$ for some $t_i\in\mathbb{Z}$. It contains an element which is a union of the curves~$\sigma_{i,j}$.
\end{lemma}

\begin{proof}
The first statement is immediate from the above discussion. The second statement holds because we can get a union of the $\sigma_{i,j}$ by gluing the curves in Figure~1.2.2 of~\cite{PH92} with suitable~$t_i$.
\end{proof}

\section{The quantum trace map}

In this section, we define a normalized version of the quantum trace map introduced by Detcherry and Santharoubane~\cite{DS25} and compute it in some special cases.

\subsection{Skeins from ribbon graphs}
\label{sec:SkeinsFromRibbonGraphs}

We will start by describing how a ribbon graph encodes elements in the skein module of a three-manifold.

\subsubsection{}
\label{sec:CurvesOnSurfaces}

Consider a compact oriented two-dimensional smooth manifold $\Sigma$ with corners and a finite set $\Pi\subset\partial\Sigma$ of marked points on the boundary. By a \emph{simple arc} on~$(\Sigma,\Pi)$, we mean the image of a smooth embedding $[0,1]\rightarrow\Sigma$ such that 0 and~1 map into~$\Pi$ and no point of $(0,1)$ maps to~$\partial\Sigma$. The condition that the map is an embedding means in particular that the images of~0 and~1 are distinct. By a \emph{simple loop} on~$\Sigma$, we mean the image of a smooth embedding $S^1\rightarrow\Sigma\setminus\partial\Sigma$. By a \emph{diagram} on~$(\Sigma,\Pi)$, we mean a union of pairwise nonintersecting simple arcs and simple loops on~$(\Sigma,\Pi)$, such that each point of~$\Pi$ is an endpoint of exactly one of the arcs. We say that two diagrams on~$(\Sigma,\Pi)$ are isotopic if they are related by an ambient isotopy of~$\Sigma$ that fixes~$\partial\Sigma$ pointwise.

To any pair $(\Sigma,\Pi)$ as above, we can associate the $\mathbb{C}(A)$-vector space $\mathcal{D}_A(\Sigma,\Pi)$ spanned by isotopy classes of diagrams on~$(\Sigma,\Pi)$. This construction is well behaved with respect to cutting and gluing of surfaces. Indeed, suppose $\Sigma$ is obtained by gluing surfaces~$\Sigma_1$ and~$\Sigma_2$ along a segment $I\subset\Sigma$ which lies on the boundary of both~$\Sigma_1$ and~$\Sigma_2$ and whose endpoints are not in~$\Pi$. Let $\Omega\subset I$ be a finite set of points in the interior of~$I$, and define 
\[
\Pi_k = (\Pi\cap\Sigma_k)\cup\Omega, \qquad k=1,2.
\]
If $\delta_k$ is a diagram on~$(\Sigma_k,\Pi_k)$ for $k=1,2$, then we get a diagram on~$(\Sigma,\Pi)$ by gluing~$\delta_1$ and~$\delta_2$. Extending $\mathbb{C}(A)$-bilinearly, we get a map 
\begin{equation}
\label{eqn:gluingmap}
\mathcal{D}_A(\Sigma_1,\Pi_1)\otimes\mathcal{D}_A(\Sigma_2,\Pi_2)\longrightarrow\mathcal{D}_A(\Sigma,\Pi).
\end{equation}
One can define a similar gluing map in the case where $\Sigma$ is obtained by identifying two segments on the boundary of a single surface.

\subsubsection{}
\label{sec:rectanglecase}

As an example, we can consider the case $\Sigma=[0,1]\times[0,1]$. Given an integer $c\geq0$, let $\Pi\subset\partial\Sigma$ be a set consisting of $c$ distinct points on $(0,1)\times\{0\}$ and $c$ distinct points on $(0,1)\times\{1\}$. Then the $\mathbb{C}(A)$-vector subspace of $\mathcal{D}_A(\Sigma,\Pi)$ spanned by diagrams without simple loops is the underlying vector space for the Temperley-Lieb algebra (see~\cite{MV94} and references given there). The Temperley-Lieb algebra contains a canonical idempotent element $\mathrm{JW}_c$ known as the Jones-Wenzl idempotent, given by an explicit recursion relation~\cite{MV94}. For example, when $c=1,2,3$, the Jones-Wenzl idempotents are illustrated in Figure~\ref{fig:JonesWenzlexamples}. In this figure, the gray shaded regions represent the rectangle $\Sigma$, and we use the notation 
\begin{equation}
\label{eqn:quantuminteger}
[k]\coloneqq\frac{A^{2k}-A^{-2k}}{A^2-A^{-2}}\in\mathbb{C}(A)
\end{equation}
for the $k$th \emph{quantum integer}.

\begin{figure}[ht]
\begin{align*}
&\mathrm{JW}_1=
\begin{tikzpicture}[baseline=(current bounding box.center)]
\clip(-1,-1) rectangle (1,1);
\draw[fill=gray,opacity=0.2] (-0.75,-0.75) rectangle (0.75,0.75);
\draw[black, thick] (0,-0.75) -- (0,0.75);
\end{tikzpicture} \\
&\mathrm{JW}_2=
\begin{tikzpicture}[baseline=(current bounding box.center)]
\clip(-1,-1) rectangle (1,1);
\draw[fill=gray,opacity=0.2] (-0.75,-0.75) rectangle (0.75,0.75);
\draw[black, thick] (-0.25,-0.75) -- (-0.25,0.75);
\draw[black, thick] (0.25,-0.75) -- (0.25,0.75);
\end{tikzpicture}
+\frac{[1]}{[2]}
\begin{tikzpicture}[baseline=(current bounding box.center)]
\clip(-1,-1) rectangle (1,1);
\draw[fill=gray,opacity=0.2] (-0.75,-0.75) rectangle (0.75,0.75);
\draw[black, thick] plot [smooth, tension=2] coordinates { (-0.25,0.75) (0,0.25) (0.25,0.75)};
\draw[black, thick] plot [smooth, tension=2] coordinates { (-0.25,-0.75) (0,-0.25) (0.25,-0.75)};
\end{tikzpicture} \\
&\mathrm{JW}_3=
\begin{tikzpicture}[baseline=(current bounding box.center)]
\clip(-1,-1) rectangle (1,1);
\draw[fill=gray,opacity=0.2] (-0.75,-0.75) rectangle (0.75,0.75);
\draw[black, thick] (-0.5,-0.75) -- (-0.5,0.75);
\draw[black, thick] (0,-0.75) -- (0,0.75);
\draw[black, thick] (0.5,-0.75) -- (0.5,0.75);
\end{tikzpicture}
+\frac{[2]}{[3]}
\begin{tikzpicture}[baseline=(current bounding box.center)]
\clip(-1,-1) rectangle (1,1);
\draw[fill=gray,opacity=0.2] (-0.75,-0.75) rectangle (0.75,0.75);
\draw[black, thick] plot [smooth, tension=2] coordinates { (-0.5,0.75) (-0.25,0.25) (0,0.75)};
\draw[black, thick] plot [smooth, tension=2] coordinates { (-0.5,-0.75) (-0.25,-0.25) (0,-0.75)};
\draw[black, thick] (0.5,-0.75) -- (0.5,0.75);
\end{tikzpicture}
+\frac{[2]}{[3]}
\begin{tikzpicture}[baseline=(current bounding box.center)]
\clip(-1,-1) rectangle (1,1);
\draw[fill=gray,opacity=0.2] (-0.75,-0.75) rectangle (0.75,0.75);
\draw[black, thick] (-0.5,-0.75) -- (-0.5,0.75);
\draw[black, thick] plot [smooth, tension=2] coordinates { (0,0.75) (0.25,0.25) (0.5,0.75)};
\draw[black, thick] plot [smooth, tension=2] coordinates { (0,-0.75) (0.25,-0.25) (0.5,-0.75)};
\end{tikzpicture}
+\frac{[1]}{[3]}
\begin{tikzpicture}[baseline=(current bounding box.center)]
\clip(-1,-1) rectangle (1,1);
\draw[fill=gray,opacity=0.2] (-0.75,-0.75) rectangle (0.75,0.75);
\draw[black, thick] plot [smooth, tension=2] coordinates { (-0.5,0.75) (-0.25,0.25) (0,0.75)};
\draw[black, thick] plot [smooth, tension=2] coordinates { (0,-0.75) (0.25,-0.25) (0.5,-0.75)};
\draw[black, thick] plot [smooth, tension=1] coordinates { (-0.5,-0.75) (-0.3,-0.125) (0.3,0.125) (0.5,0.75)};
\end{tikzpicture}
+\frac{[1]}{[3]}
\begin{tikzpicture}[baseline=(current bounding box.center)]
\clip(-1,-1) rectangle (1,1);
\draw[fill=gray,opacity=0.2] (-0.75,-0.75) rectangle (0.75,0.75);
\draw[black, thick] plot [smooth, tension=2] coordinates { (0,0.75) (0.25,0.25) (0.5,0.75)};
\draw[black, thick] plot [smooth, tension=2] coordinates { (-0.5,-0.75) (-0.25,-0.25) (0,-0.75)};
\draw[black, thick] plot [smooth, tension=1] coordinates { (-0.5,0.75) (-0.3,0.125) (0.3,-0.125) (0.5,-0.75)};
\end{tikzpicture}
\end{align*}
\caption{Some Jones--Wenzl idempotents.\label{fig:JonesWenzlexamples}}
\end{figure}
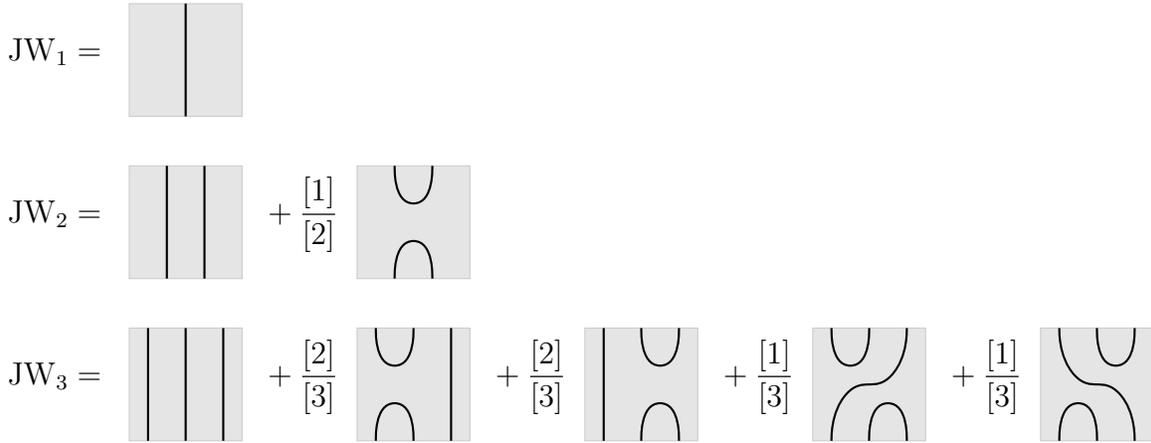

\subsubsection{}
\label{sec:trianglecase}

We can also consider a closed triangle $\Sigma$  with sides $e$, $f$, and~$g$. Suppose we are given integers $c(e)$, $c(f)$,~$c(g)\geq0$ such that the sum $c(e)+c(f)+c(g)$ is even and such that the triangle inequality $c(r)\leq c(s)+c(t)$ is satisfied whenever $\{r,s,t\}=\{e,f,g\}$. It follows from these conditions that the numbers 
\[
i=\frac{c(f)+c(g)-c(e)}{2}, \quad j=\frac{c(e)+c(g)-c(f)}{2}, \quad k=\frac{c(e)+c(f)-c(g)}{2}
\]
are nonnegative integers. Let $\Pi\subset\partial\Sigma$ be a set consisting of $c(s)$ distinct points in the interior of each side~$s$ of~$\Sigma$. For convenience, let us represent a collection of $m$ parallel curves on~$\Sigma$ by a single curve labeled by the integer~$m$. Then up to isotopy there is a unique diagram on~$(\Sigma,\Pi)$ consisting of the nonintersecting arcs in Figure~\ref{fig:trianglearcs} connecting points of~$\Pi$ and no other curves.

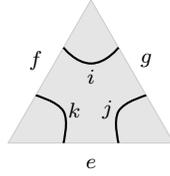
\begin{figure}[ht]
\begin{tikzpicture}[scale=1.1]
\draw[fill=gray,opacity=0.2] (-1,0) -- (0,1.732) -- (1,0) -- cycle;
\draw[black, thick] plot [smooth, tension=1] coordinates { (-0.333,1.155) (0,0.962) (0.333,1.155)};
\draw[black, thick] plot [smooth, tension=0.75] coordinates { (-0.667,0.577) (-0.333,0.385) (-0.333,0)};
\draw[black, thick] plot [smooth, tension=0.75] coordinates { (0.667,0.577) (0.333,0.385) (0.333,0)};
\node at (0,-0.25) {\tiny $e$};
\node at (-0.667,1) {\tiny $f$};
\node at (0.667,1) {\tiny $g$};
\node at (0,0.8) {\tiny $i$};
\node at (0.2,0.4) {\tiny $j$};
\node at (-0.2,0.4) {\tiny $k$};
\end{tikzpicture}
\caption{Arcs on a triangle.\label{fig:trianglearcs}}
\end{figure}

\subsubsection{}

Recall that a \emph{ribbon graph} is a graph with a cyclic ordering of the edges incident to each vertex. We let $\Gamma$ be a ribbon graph with finitely many vertices and edges such that every vertex is either trivalent or univalent. We will write $\Gamma_0$ and $\Gamma_1$ for the set of vertices and set of edges, respectively, of~$\Gamma$. By a \emph{coloring} of~$\Gamma$, we will mean a function $c:\Gamma_1\rightarrow\mathbb{Z}_{\geq0}$. A coloring $c:\Gamma_1\rightarrow\mathbb{Z}_{\geq0}$ is said to  be \emph{admissible} if, for every triple of distinct edges $e$, $f$,~$g\in\Gamma_1$ incident to a common vertex, 
\begin{enumerate}
\item $c(e)+c(f)+c(g)$ is even.
\item $c(e)\leq c(f)+c(g)$.
\end{enumerate}

Suppose $c$ is an admissible coloring of~$\Gamma$. If $v$ is a trivalent vertex of~$\Gamma$, let $\Sigma_v$ be an oriented triangle with sides labeled by the edges incident to $v$ in such a way that the cyclic order defined by the ribbon graph agrees with the cyclic order induced by the orientation of~$\Sigma_v$. Let $\Pi_{v,c}\subset\partial\Sigma_v$ be a set consisting of $c(e)$ distinct points in the interior of the side with label~$e$, for every label~$e$. Since we assume $c$ is admissible, we get a diagram $\delta_{v,c}\in\mathcal{D}_A(\Sigma_v,\Pi_{v,c})$ by the construction of~\ref{sec:trianglecase}. On the other hand, for every edge~$e$ of~$\Gamma$, let $\Sigma_e=[0,1]\times[0,1]$ be an oriented rectangle. Let us label the sides $[0,1]\times\{0\}$ and $[0,1]\times\{1\}$ of~$\Sigma_e$ by the two vertices of~$\Gamma$ that are incident to $e$. Let $\Pi_{e,c}\subset\Sigma_e$ be a set consisting of $c(e)$ distinct points in the interior of each of these sides. Then the Jones-Wenzl idempotent defines an element $\delta_{e,c}\in\mathcal{D}_A(\Sigma_e,\Pi_{e,c})$ as in~\ref{sec:rectanglecase}.

Now if an edge~$e$ is incident to a trivalent vertex~$v$, we can glue the side of $\Sigma_e$ with label~$v$ to the side of~$\Sigma_v$ with label~$e$ by an orientation reversing homeomorphism that induces a bijection of the marked points on these sides. In this way, we obtain an oriented surface~$\Sigma_\Gamma$ well defined up to orientation preserving homeomorphism. The sets $\Pi_{e,c}$ give rise to a subset $\Pi_{\Gamma,c}\subset\partial\Sigma_\Gamma$. Using the gluing maps~\eqref{eqn:gluingmap}, we can glue the $\delta_v$ and $\delta_e$ to get an element $\delta_{\Gamma,c}\in\mathcal{D}_A(\Sigma_\Gamma,\Pi_{\Gamma,c})$.

\subsubsection{}

Let $\Gamma$ be a ribbon graph with trivalent and univalent vertices as above. We suppose that $\Gamma$ is embedded in a three-manifold~$M$ in such a way that $\Gamma\cap\partial M$ is the set of all univalent vertices of~$\Gamma$. Noting that $\Gamma$ naturally embeds in the surface~$\Sigma_\Gamma$ as a deformation retract, we can extend the embedding of~$\Gamma$ in~$M$ to an embedding of~$\Sigma_\Gamma$ in~$M$ such that $\Sigma_\Gamma\cap\partial M\subset\partial\Sigma_\Gamma$.

Given an admissible coloring~$c$ of~$\Gamma$, we can get a finite set $\Pi_{\Gamma,c}\subset\Sigma_\Gamma\cap\partial M$ of marked points and an element $\delta_{\Gamma,c}\in\mathcal{D}_A(\Sigma_\Gamma,\Pi_{\Gamma,c})$. If $\delta\in\mathcal{D}_A(\Sigma_\Gamma,\Pi_{\Gamma,c})$ is a diagram on~$(\Sigma_\Gamma,\Pi_{\Gamma,c})$, then we get an element of the extended skein module of~$M$ by regarding $\delta$ as a tangle in~$M$ and choosing a framing vector at each point of~$\delta$ normal to~$\Sigma_\Gamma$ in the direction determined by the orientation of this surface. By extending $\mathbb{C}(A)$-linearly, we get a map 
\[
\mathcal{D}_A(\Sigma_\Gamma,\Pi_{\Gamma,c})\longrightarrow\widehat{\Sk}_A(M).
\]
We will write $\varphi_c=\varphi_{\Gamma,c}$ for the image of~$\delta_{\Gamma,c}$ under this map.

\subsection{Fusion rules}

Let $\Gamma$ be a ribbon graph with trivalent and univalent vertices as above, and assume $\Gamma$ is embedded in a three-manifold~$M$ in such a way that~$\Gamma\cap\partial M$ is the set of all univalent vertices of~$\Gamma$. Then for any admissible coloring~$c$ of~$\Gamma$, we get an associated element~$\varphi_{\Gamma,c}\in\widehat{\Sk}_A(M)$ of the extended skein module by the construction of Section~\ref{sec:SkeinsFromRibbonGraphs}. In the following, we will often represent this element~$\varphi_{\Gamma,c}$ by drawing the graph~$\Gamma$ with each edge $e\in\Gamma_1$ labeled by the integer~$c(e)$.

Figures~\ref{fig:parallelstrands}, \ref{fig:halftwist}, \ref{fig:biangle}, and~\ref{fig:triangle} show a number of relations, called \emph{fusion rules}, between the elements $\varphi_{\Gamma,c}$ for different ribbon graphs~$\Gamma$. In each picture, the cyclic ordering of a vertex of the graph is assumed to be the one induced by the standard orientation of the plane in which it is drawn, and we write $[k]$ for the $k$th quantum integer defined by~\eqref{eqn:quantuminteger}. The fusion rules are proved in~\cite{MV94}.

\begin{figure}[ht]
\[
\begin{tikzpicture}[scale=0.9, baseline=(current bounding box.center)]
\draw[black, thick] (-0.25,-1) -- (-0.25,1) node[below left] {\tiny $c$};
\draw[black, thick] (0.25,-1) -- (0.25,1) node[below right] {\tiny $1$};
\end{tikzpicture}
=
\begin{tikzpicture}[scale=0.9, baseline=(current bounding box.center)]
\draw[black, thick] (0,-0.75) -- (-0.25,-1) node[above left] {\tiny $c$};
\draw[black, thick] (0,-0.75) -- (0.25,-1) node[above right] {\tiny $1$};
\draw[black, thick] (0,0.75) -- (0,-0.75);
\node[left] at (0,0) {\tiny $c+1$};
\draw[black, thick] (0,0.75) -- (-0.25,1) node[below left] {\tiny $c$};
\draw[black, thick] (0,0.75) -- (0.25,1) node[below right] {\tiny $1$};
\end{tikzpicture}
-\frac{[c]}{[c+1]}
\begin{tikzpicture}[scale=0.9, baseline=(current bounding box.center)]
\draw[black, thick] (0,-0.75) -- (-0.25,-1) node[above left] {\tiny $c$};
\draw[black, thick] (0,-0.75) -- (0.25,-1) node[above right] {\tiny $1$};
\draw[black, thick] (0,0.75) -- (0,-0.75);
\node[left] at (0,0) {\tiny $c-1$};
\draw[black, thick] (0,0.75) -- (-0.25,1) node[below left] {\tiny $c$};
\draw[black, thick] (0,0.75) -- (0.25,1) node[below right] {\tiny $1$};
\end{tikzpicture}
\]
\caption{Parallel strands rule.\label{fig:parallelstrands}}
\end{figure}

\begin{figure}[ht]
\begin{subfigure}{0.4\textwidth}
\[
\begin{tikzpicture}[baseline=(current bounding box.center)]
\draw[black, thick] (0,0) -- (0,-0.75) node[above left] {\tiny $c+1$};
\draw[black, thick] plot [smooth, tension=1] coordinates { (0,0) (0.25,0.5) (-0.5,1)} node[below left] {\tiny $1$};
\fill[white] (0,0.75) circle (0.1);
\draw[black, thick] plot [smooth, tension=1] coordinates { (0,0) (-0.25,0.5) (0.5,1)} node[below right] {\tiny $c$};
\end{tikzpicture}
= A^c
\begin{tikzpicture}[baseline=(current bounding box.center)]
\draw[black, thick] (0,0) -- (0,-0.75) node[above left] {\tiny $c+1$};
\draw[black, thick] (0,0) -- (-0.5,1) node[below left] {\tiny $1$};
\draw[black, thick] (0,0) -- (0.5,1) node[below right] {\tiny $c$};
\end{tikzpicture}
\]
\caption{\label{subfig:halftwist1}}
\end{subfigure}
\begin{subfigure}{0.4\textwidth}
\[
\begin{tikzpicture}[baseline=(current bounding box.center)]
\draw[black, thick] (0,0) -- (0,-0.75) node[above left] {\tiny $c-1$};
\draw[black, thick] plot [smooth, tension=1] coordinates { (0,0) (0.25,0.5) (-0.5,1)} node[below left] {\tiny $1$};
\fill[white] (0,0.75) circle (0.1);
\draw[black, thick] plot [smooth, tension=1] coordinates { (0,0) (-0.25,0.5) (0.5,1)} node[below right] {\tiny $c$};
\end{tikzpicture}
= -A^{-(c+2)}
\begin{tikzpicture}[baseline=(current bounding box.center)]
\draw[black, thick] (0,0) -- (0,-0.75) node[above left] {\tiny $c-1$};
\draw[black, thick] (0,0) -- (-0.5,1) node[below left] {\tiny $1$};
\draw[black, thick] (0,0) -- (0.5,1) node[below right] {\tiny $c$};
\end{tikzpicture}
\]
\caption{\label{subfig:halftwist2}}
\end{subfigure}
\caption{Half twist rules.\label{fig:halftwist}}
\end{figure}

\begin{figure}[ht]
\begin{subfigure}{0.4\textwidth}
\[
\begin{tikzpicture}[baseline=(current bounding box.center)]
\draw[black, thick] (0,-0.5) -- (0,-1) node[above left] {\tiny $c$};
\draw[black, thick] plot [smooth, tension=1] coordinates { (0,-0.5) (-0.25,0) (0,0.5)};
\node[left] at (-0.25,0) {\tiny $c+1$};
\draw[black, thick] plot [smooth, tension=1] coordinates { (0,-0.5) (0.25,0) (0,0.5)};
\node[right] at (0.25,0) {\tiny $1$};
\draw[black, thick] (0,0.5) -- (0,1) node[below left] {\tiny $c$};
\end{tikzpicture}
= -\frac{[c+2]}{[c+1]}
\begin{tikzpicture}[baseline=(current bounding box.center)]
\draw[black, thick] (0,-1) -- (0,1) node[below left] {\tiny $c$};
\end{tikzpicture}
\]
\caption{\label{subfig:biangle1}}
\end{subfigure}
\begin{subfigure}{0.4\textwidth}
\[
\begin{tikzpicture}[baseline=(current bounding box.center)]
\draw[black, thick] (0,-0.5) -- (0,-1) node[above left] {\tiny $c$};
\draw[black, thick] plot [smooth, tension=1] coordinates { (0,-0.5) (-0.25,0) (0,0.5)};
\node[left] at (-0.25,0) {\tiny $c-1$};
\draw[black, thick] plot [smooth, tension=1] coordinates { (0,-0.5) (0.25,0) (0,0.5)};
\node[right] at (0.25,0) {\tiny $1$};
\draw[black, thick] (0,0.5) -- (0,1) node[below left] {\tiny $c$};
\end{tikzpicture}
=
\begin{tikzpicture}[baseline=(current bounding box.center)]
\draw[black, thick] (0,-1) -- (0,1) node[below left] {\tiny $c$};
\end{tikzpicture}
\]
\caption{\label{subfig:biangle2}}
\end{subfigure}
\caption{Biangle reducing rules.\label{fig:biangle}}
\end{figure}

\begin{figure}[ht]
\begin{subfigure}{0.4\textwidth}
\[
\begin{tikzpicture}[scale=1.25, baseline=(current bounding box.center)]
\draw[black, thick] (0,0) -- (0,0.5) node[below left] {\tiny $a$};
\draw[black, thick] (0,0) -- (-0.25,-0.5) node[above left] {\tiny $b$};
\draw[black, thick] (0,0) -- (0.25,-0.5) node[above right] {\tiny $c$};
\draw[black, thick] (-0.25,-0.5) -- (-0.5,-0.75) node[below] {\tiny $b+1$};
\draw[black, thick] (0.25,-0.5) -- (0.5,-0.75) node[below] {\tiny $c+1$};
\draw[black, thick] (-0.25,-0.5) -- (0.25,-0.5) node[below left] {\tiny $1$};
\end{tikzpicture}
=
\begin{tikzpicture}[scale=1.25, baseline=(current bounding box.center)]
\draw[black, thick] (0,-0.25) -- (0,0.5) node[below left] {\tiny $a$};
\draw[black, thick] (0,-0.25) -- (-0.5,-0.75) node[below] {\tiny $b+1$};
\draw[black, thick] (0,-0.25) -- (0.5,-0.75) node[below] {\tiny $c+1$};
\end{tikzpicture}
\]
\caption{\label{subfig:triangle1}}
\end{subfigure}
\begin{subfigure}{0.4\textwidth}
\[
\begin{tikzpicture}[scale=1.25, baseline=(current bounding box.center)]
\draw[black, thick] (0,0) -- (0,0.5) node[below left] {\tiny $a$};
\draw[black, thick] (0,0) -- (-0.25,-0.5) node[above left] {\tiny $b$};
\draw[black, thick] (0,0) -- (0.25,-0.5) node[above right] {\tiny $c$};
\draw[black, thick] (-0.25,-0.5) -- (-0.5,-0.75) node[below] {\tiny $b+1$};
\draw[black, thick] (0.25,-0.5) -- (0.5,-0.75) node[below] {\tiny $c-1$};
\draw[black, thick] (-0.25,-0.5) -- (0.25,-0.5) node[below left] {\tiny $1$};
\end{tikzpicture}
=\frac{[\frac{a-b+c}{2}]}{[c]}
\begin{tikzpicture}[scale=1.25, baseline=(current bounding box.center)]
\draw[black, thick] (0,-0.25) -- (0,0.5) node[below left] {\tiny $a$};
\draw[black, thick] (0,-0.25) -- (-0.5,-0.75) node[below] {\tiny $b+1$};
\draw[black, thick] (0,-0.25) -- (0.5,-0.75) node[below] {\tiny $c-1$};
\end{tikzpicture}
\]
\caption{\label{subfig:triangle2}}
\end{subfigure}
\begin{subfigure}{0.5\textwidth}
\[
\begin{tikzpicture}[scale=1.25, baseline=(current bounding box.center)]
\draw[black, thick] (0,0) -- (0,0.5) node[below left] {\tiny $a$};
\draw[black, thick] (0,0) -- (-0.25,-0.5) node[above left] {\tiny $b$};
\draw[black, thick] (0,0) -- (0.25,-0.5) node[above right] {\tiny $c$};
\draw[black, thick] (-0.25,-0.5) -- (-0.5,-0.75) node[below] {\tiny $b-1$};
\draw[black, thick] (0.25,-0.5) -- (0.5,-0.75) node[below] {\tiny $c-1$};
\draw[black, thick] (-0.25,-0.5) -- (0.25,-0.5) node[below left] {\tiny $1$};
\end{tikzpicture}
=-\frac{[\frac{a+b+c}{2}+1][\frac{b+c-a}{2}]}{[b][c]}
\begin{tikzpicture}[scale=1.25, baseline=(current bounding box.center)]
\draw[black, thick] (0,-0.25) -- (0,0.5) node[below left] {\tiny $a$};
\draw[black, thick] (0,-0.25) -- (-0.5,-0.75) node[below] {\tiny $b-1$};
\draw[black, thick] (0,-0.25) -- (0.5,-0.75) node[below] {\tiny $c-1$};
\end{tikzpicture}
\]
\caption{\label{subfig:triangle3}}
\end{subfigure}
\caption{Triangle reducing rules.\label{fig:triangle}}
\end{figure}

\subsection{Definition of the quantum trace}
\label{sec:DefinitionOfTheQuantumTrace}

We now use the ideas of the previous sections to construct an algebra homomorphism from the skein algebra of a surface into a localized quantum torus. Let~$\Gamma$ be a finite ribbon graph in~$\mathbb{R}^3$ such that every vertex of $\Gamma$ is either trivalent or univalent. Then we can let~$M\subset\mathbb{R}^3$ be the handlebody defined as the closure of a small tubular neighborhood of~$\Gamma$ in~$\mathbb{R}^3$. The boundary of $M$ is a closed surface $\bar{S}$, and we define $S$ to be the surface with boundary obtained by removing a small open disk $D_v\subset\bar{S}$ around each of the univalent vertices~$v$ of~$\Gamma$.

If $c$ is any admissible coloring of~$\Gamma$, we get a finite set $\Pi_{\Gamma,c}\subset\bar{S}$ of marked points by the construction of Section~\ref{sec:SkeinsFromRibbonGraphs}, and we may assume that the $c(e)$ points of~$\Pi_{\Gamma,c}$ associated with an edge~$e\in\Gamma_1$ that ends at a univalent vertex~$v\in\Gamma_0$ are all contained in the disk~$D_v\subset\bar{S}$. We also get an element $\varphi_c\in\widehat{\Sk}_A(M)$, which is a $\mathbb{C}(A)$-linear combination of framed tangles ending at points of~$\Pi_{\Gamma,c}$. Suppose we are given a map $d:\Gamma_1^*\rightarrow\mathbb{Z}_{\geq0}$ where $\Gamma_1^*\subset\Gamma_1$ is the set of all $e\in\Gamma_1$ such that $e$ is incident to some univalent vertex. If $c,c':\Gamma_1\rightarrow\mathbb{Z}_{\geq0}$ are admissible colorings such that $c|_{\Gamma_1^*}=c'|_{\Gamma_1^*}=d$, then we can assume $\Pi_{\Gamma,c}=\Pi_{\Gamma,c'}$. We define 
\[
\Sk_A(M,d)\subset\widehat{\Sk}_A(M)
\]
to be the $\mathbb{C}(A)$-submodule spanned by $\varphi_c$ for all admissible colorings $c$ with $c|_{\Gamma_1^*}=d$.

There is a natural action of the skein algebra $\Sk_A(S)$ on~$\Sk_A(M,d)$. Indeed, suppose $L\in\Sk_A(S)$ is a framed link and $K\in\Sk_A(M,d)$ a framed tangle. We can glue $S\times[0,1]$ to~$M$ by identifying $S\times\{0\}\cong S$ with a subset of~$\partial M$. The union of $K$ and $L$ in the glued three-manifold is isotopic to a framed tangle~$L\cdot K$ in~$M$. Extending $\mathbb{C}(A)$-linearly gives the desired action.

To study the endomorphism of $\Sk_A(M,d)$ corresponding to an element $\gamma\in\Sk_A(S)$, we will consider the matrix elements of this endomorphism with respect to a basis. We work with elements $\widetilde{\varphi}_c=\kappa_c\,\varphi_c$ where $\kappa_c\in\mathbb{C}(A)$ is a chosen nonzero scalar for each admissible coloring~$c$. These elements form a basis by Lemma~2.1 of~\cite{DS25}. We will write $\Delta_\Gamma\subset\mathbb{Z}_{\geq0}^{\Gamma_1}$ for the set of admissible colorings of~$\Gamma$ and write $\Lambda_\Gamma\subset\mathbb{Z}_{\geq0}^{\Gamma_1}$ for the lattice of colorings $c:\Gamma_1\rightarrow\mathbb{Z}_{\geq0}$ such that $c(e)+c(f)+c(g)$ is even whenever $e$,~$f$,~$g\in\Gamma_1$ are the edges incident to a trivalent vertex. We say a set $U\subset\mathbb{Z}_{\geq0}^{\Gamma_1}$ is \emph{good} if there exist finitely many colorings $c_1,\dots,c_m\in\Lambda_\Gamma$ such that $U=\bigcap_{i=1}^m(\Delta_\Gamma+c_i)$. (We note that such a set is called ``large'' in the terminology of~\cite{DS25}.)

\begin{proposition}
\label{prop:skeinalgebraaction}
If $\gamma$ is any loop on~$S$, then there exists a good set $U_\gamma\subset\mathbb{Z}_{\geq0}^{\Gamma_1}$ and rational functions $F_k^\gamma\in\mathbb{C}(x_i:i\in\Gamma_1)$ such that, for any $c\in U_\gamma$, 
\[
\gamma\cdot\widetilde{\varphi}_c=\sum_{k\in\mathbb{Z}_{\geq0}^{\Gamma_1},k|_{\Gamma_1^*}=0}F_k^\gamma(A^{c(i)})\,\widetilde{\varphi}_{c+k}, 
\]
where we write $f(A^{c(i)})$ for the expression obtained from $f\in\mathbb{C}(x_i:i\in\Gamma_1)$ by making the substitution $x_i=A^{c(i)}$.
\end{proposition}

\begin{proof}
This statement appears in Proposition~2.3 of~\cite{DS25} in the special case where $\kappa_c=1$ for all~$c$. The general case is obtained by rescaling the coefficient in~\cite{DS25} by $\kappa_c/\kappa_{c+k}$.
\end{proof}

\begin{proposition}
\label{prop:uniquerationalfunction}
Let $\gamma$ be any loop on~$S$, and let $F_k^\gamma$ be the rational functions from Proposition~\ref{prop:skeinalgebraaction}. If $V_\gamma\subset\mathbb{Z}_{\geq0}^{\Gamma_1}$ is a good set and $G_k^\gamma\in\mathbb{C}(x_i:i\in\Gamma_1)$ are rational functions such that 
\[
\gamma\cdot\widetilde{\varphi}_c=\sum_{k\in\mathbb{Z}_{\geq0}^{\Gamma_1},k|_{\Gamma_1^*}=0}G_k^\gamma(A^{c(i)})\,\widetilde{\varphi}_{c+k}
\]
for any $c\in V_\gamma$, then we have $F_k^\gamma=G_k^\gamma$ for all~$k$.
\end{proposition}

\begin{proof}
This follows as in the proof of Lemma~2.4 of~\cite{DS25}.
\end{proof}

Let $\mathbf{C}=(C_l)_{l\in\Gamma_1^*}$ be a tuple of formal variables and write $\mathscr{Z}_{A,\mathbf{C}}^0(\Gamma)$ be the $\mathbb{C}(A,C_l:l\in\Gamma_1^*)$-algebra generated by variables $Q_i^{\pm1}$, $E_i^{\pm1}$ for $i\in\Gamma_1\setminus\Gamma_1^*$ subject to the relations $Q_iE_j=A^{\delta_{ij}}E_jQ_i$ and $[Q_i,Q_j]=[E_i,E_j]=0$ for $i$,~$j\in\Gamma_1\setminus\Gamma_1^*$. Let $\mathscr{Z}_{A,\mathbf{C}}(\Gamma)$ be the localization of $\mathscr{Z}_{A,\mathbf{C}}^0(\Gamma)$ at the multiplicative subset generated by the expressions $A^mQ_i^2-A^{-m}Q_i^{-2}$ for $m\in\mathbb{Z}$. By Proposition~\ref{prop:uniquerationalfunction}, a loop $\gamma$ on~$S$ determines a well defined collection of rational functions~$F_k^\gamma$.

\begin{proposition}
\label{prop:quantumtraceexists}
There is an injective $\mathbb{C}[A^{\pm1}]$-algebra homomorphism 
\begin{equation}
\label{eqn:quantumtrace}
\Upsilon:\Sk_A(S)\longrightarrow\mathscr{Z}_{A,\mathbf{C}}(\Gamma)
\end{equation}
such that for any loop $\gamma$ on~$S$, we have 
\[
\Upsilon(\gamma)=\sum_{k\in\mathbb{Z}_{\geq0}^{\Gamma_1},k|_{\Gamma_1^*}=0}E^kF_k^\gamma(Q_i,C_l)
\]
where $E^k\coloneqq\prod_{i\in\Gamma_1}E_i^{k(i)}$ and $F_k^\gamma(Q_i,C_l)$ is obtained from the rational function $F_k^\gamma\in\mathbb{C}(x_i:i\in\Gamma)$ of Proposition~\ref{prop:skeinalgebraaction} by making the substitutions $x_i=Q_i$ for $i\in\Gamma_1\setminus\Gamma_1^*$ and the substitutions $x_l=C_l$ for $l\in\Gamma_1^*$.
\end{proposition}

\begin{proof}
This follows as in Lemmas~2.6 and~2.8 of~\cite{DS25}.
\end{proof}

The map~\eqref{eqn:quantumtrace} was constructed by Detcherry and Santharoubane~\cite{DS25} in the special case where $\kappa_c=1$ for all~$c$. It can be considered as an analog of the well known quantum trace map of Bonahon and Wong~\cite{BW11}. We will thus refer to~\eqref{eqn:quantumtrace} also as a \emph{quantum trace} map.

\subsection{Computations of the quantum trace}

Let us now compute the quantum trace in some important special cases. We suppose $\Gamma$ is the ribbon graph illustrated in Figure~\ref{fig:ribbongraph} where the cyclic ordering of the edges incident to a vertex is induced by the orientation of the plane in which it is drawn. We call the edges $e_1,\dots,e_{n-1}$ and $f_0,\dots,f_{n+1}$ as in the figure.

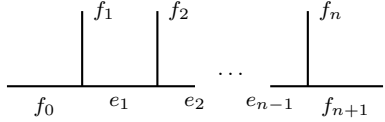
\begin{figure}[ht]
\begin{tikzpicture}
\draw[black, thick] (-2.5,0) -- (0,0);
\draw[black, thick] (-1.5,0) -- (-1.5,1);
\draw[black, thick] (-0.5,0) -- (-0.5,1);
\draw[black, thick] (1.5,0) -- (1.5,1);
\node[below] at (-2,0) {\tiny $f_0$};
\node[below] at (-1,0) {\tiny $e_1$};
\node[below] at (0,0) {\tiny $e_2$};
\node[above] at (0.5,0) {\tiny $\dots$};
\node[below] at (1,0) {\tiny $e_{n-1}$};
\node[below] at (2,0) {\tiny $f_{n+1}$};
\node[right] at (-1.5,1) {\tiny $f_1$};
\node[right] at (-0.5,1) {\tiny $f_2$};
\node[right] at (1.5,1) {\tiny $f_n$};
\draw[black, thick] (1,0) -- (2.5,0);
\end{tikzpicture}
\caption{A ribbon graph.\label{fig:ribbongraph}}
\end{figure}

Let us assume that~$\Gamma$ is embedded in~$\mathbb{R}^3$. As in Section~\ref{sec:DefinitionOfTheQuantumTrace}, we let $M\subset\mathbb{R}^3$ be the handlebody defined as the closure of a small tubular neighborhood of~$\Gamma$, and we let~$S$ be the surface obtained from $\bar{S}=\partial M$ by removing a small disk~$D_v\subset\bar{S}$ around each univalent vertex~$v$ of~$\Gamma$. This surface~$S$ is homeomorphic to~$S_{0,n+2}$, and we can identify it with the surface in Figure~\ref{subfig:puncturedsphere} in such a way that if $v$ is the univalent vertex incident to the edge~$f_i$ of~$\Gamma$, then $\partial\bar{D}_v$ is identified with the $i$th boundary component in Figure~\ref{subfig:puncturedsphere}.

Let us fix a map $d:\Gamma_1^*\rightarrow\mathbb{Z}_{\geq0}$ and write $d_i\coloneqq d(f_i)$. Given any admissible coloring $c:\Gamma_1\rightarrow\mathbb{Z}_{\geq0}$, we write $c_i\coloneqq c(e_i)$. If this coloring satisfies $c|_{\Gamma_1^*}=d$, then we get the basis element $\varphi_c\in\Sk_A(M,d)$ as in Section~\ref{sec:DefinitionOfTheQuantumTrace}. Recall that in the definition of the quantum trace map, we consider the action of the skein algebra $\Sk_A(S)$ on the normalized basis elements $\widetilde{\varphi}_c=\kappa_c\,\varphi_c$. Here we will take the normalizing factor to be 
\[
\kappa_c\coloneqq\frac{[c_1]![c_2]!\dots[c_{n-1}]!}{[\frac{d_0+c_1-d_1}{2}]![\frac{c_1+c_2-d_2}{2}]!\dots[\frac{c_{n-1}+d_{n+1}-d_n}{2}]!}
\]
where $[k]!\coloneqq\prod_{j=1}^k[j]$ and $[0]!\coloneqq1$ are the \emph{quantum factorials}. In the following two propositions, we compute the values of the quantum trace map on generators of the skein algebra using this normalization. To simplify expressions in the following, we will write $C_j\coloneqq C_{f_j}$, $Q_j\coloneqq Q_{e_j}$, and $E_j\coloneqq E_{e_j}$.

\begin{proposition}
\label{prop:tracegammadelta}
Let $\gamma_i$ and $\delta_j$ be the curves defined in Figure~\ref{fig:generators}. Then the quantum trace map~$\Upsilon$ takes the values 
\begin{align*}
\Upsilon(\gamma_i) &= -A^2Q_i^2-A^{-2}Q_i^{-2}, \\
\Upsilon(\delta_j) &= -A^2C_j^2-A^{-2}C_j^{-2}.
\end{align*}
\end{proposition}

\begin{proof}
Let $c:\Gamma_1\rightarrow\mathbb{Z}_{\geq0}$ be an admissible coloring. Using the fusion rules, we can compute the action of $\gamma_i$ on~$\varphi_c$ as in Figure~\ref{fig:fusioncalculation}. (By abuse of notation, we are writing $c_i=c(e_i)$ simply as $c$ in this figure.) We find that $\gamma_i\cdot\varphi_c=-(A^{2c_i+2}+A^{-2c_i-2})\varphi_c$, or equivalently, 
\[
\gamma_i\cdot\widetilde{\varphi}_c=-(A^{2c_i+2}+A^{-2c_i-2})\widetilde{\varphi}_c.
\]
Together with Proposition~\ref{prop:quantumtraceexists}, this implies the formula for $\Upsilon(\gamma_i)$. The formula for $\Upsilon(\delta_j)$ is proved in the same way, replacing~$e_i$ by~$f_j$ and replacing~$\gamma_i$ by~$\delta_j$.
\end{proof}

\begin{figure}[ht]
\begin{align*}
\begin{tikzpicture}[scale=1.5, baseline=(current bounding box.center)]
\clip(-1.25,-0.75) rectangle (1.25,0.55);
\draw[black, thick] (0.4,0) -- (-1,0) node[above right] {\tiny $c$};
\draw [black,thick,domain=-165:165] plot ({0.5*cos(\x)}, {0.5*sin(\x)});
\draw[black, thick] (0.6,0) -- (1,0);
\end{tikzpicture}
&=
\begin{tikzpicture}[scale=1.5, baseline=(current bounding box.center)]
\clip(-1.25,-0.75) rectangle (1.25,0.55);
\draw[black, thick] (0.4,0) -- (-1,0) node[above right] {\tiny $c$};
\draw [black,thick,domain=-170:0] plot ({0.5*cos(\x)}, {0.5*sin(\x)});
\node[below] at (0,-0.5) {\tiny $1$};
\draw [black,thick,domain=0:180] plot ({0.375+0.125*cos(\x)}, {0.125*sin(\x)});
\draw [black,thick,domain=30:180] plot ({-0.375-0.125*cos(\x)}, {0.125*sin(\x)});
\draw[black, thick] (0.6,0) -- (1,0) node[above left] {\tiny $c$};
\node[above] at (0,0) {\tiny $c+1$};
\end{tikzpicture}
-\frac{[c]}{[c+1]}
\begin{tikzpicture}[scale=1.5, baseline=(current bounding box.center)]
\clip(-1.25,-0.75) rectangle (1.25,0.55);
\draw[black, thick] (0.4,0) -- (-1,0) node[above right] {\tiny $c$};
\draw [black,thick,domain=-170:0] plot ({0.5*cos(\x)}, {0.5*sin(\x)});
\node[below] at (0,-0.5) {\tiny $1$};
\draw [black,thick,domain=0:180] plot ({0.375+0.125*cos(\x)}, {0.125*sin(\x)});
\draw [black,thick,domain=30:180] plot ({-0.375-0.125*cos(\x)}, {0.125*sin(\x)});
\draw[black, thick] (0.6,0) -- (1,0) node[above left] {\tiny $c$};
\node[above] at (0,0) {\tiny $c-1$};
\end{tikzpicture} \\
&=A^c
\begin{tikzpicture}[scale=1.5, baseline=(current bounding box.center)]
\clip(-1.25,-0.625) rectangle (1.25,0.625);
\node[below] at (0.125,-0.375) {\tiny $1$};
\draw[black, thick] (0.4,0) -- (-1,0) node[above right] {\tiny $c$};
\draw [black,thick,domain=-180:0] plot ({0.125+0.375*cos(\x)}, {0.375*sin(\x)});
\draw [black,thick,domain=0:180] plot ({0.375+0.125*cos(\x)}, {0.125*sin(\x)});
\draw[black, thick] (0.6,0) -- (1,0) node[above left] {\tiny $c$};
\node[above] at (0,0) {\tiny $c+1$};
\end{tikzpicture}
+\frac{[c]}{[c+1]}A^{-(c+2)}
\begin{tikzpicture}[scale=1.5, baseline=(current bounding box.center)]
\clip(-1.25,-0.625) rectangle (1.25,0.625);
\node[below] at (0.125,-0.375) {\tiny $1$};
\draw[black, thick] (0.4,0) -- (-1,0) node[above right] {\tiny $c$};
\draw [black,thick,domain=-180:0] plot ({0.125+0.375*cos(\x)}, {0.375*sin(\x)});
\draw [black,thick,domain=0:180] plot ({0.375+0.125*cos(\x)}, {0.125*sin(\x)});
\draw[black, thick] (0.6,0) -- (1,0) node[above left] {\tiny $c$};
\node[above] at (0,0) {\tiny $c-1$};
\end{tikzpicture} \\
&=A^{2c}
\begin{tikzpicture}[scale=1.5, baseline=(current bounding box.center)]
\clip(-1.25,-0.5) rectangle (1.25,0.5);
\node[below] at (0,-0.25) {\tiny $1$};
\draw[black, thick] (0,0) -- (-1,0) node[above right] {\tiny $c$};
\draw [black,thick,domain=-180:0] plot ({0.25*cos(\x)}, {0.25*sin(\x)});
\draw[black, thick] (0,0) -- (1,0) node[above left] {\tiny $c$};
\node[above] at (0,0) {\tiny $c+1$};
\end{tikzpicture}
-\frac{[c]}{[c+1]}A^{-2(c+2)}
\begin{tikzpicture}[scale=1.5, baseline=(current bounding box.center)]
\clip(-1.25,-0.5) rectangle (1.25,0.5);
\node[below] at (0,-0.25) {\tiny $1$};
\draw[black, thick] (0,0) -- (-1,0) node[above right] {\tiny $c$};
\draw [black,thick,domain=-180:0] plot ({0.25*cos(\x)}, {0.25*sin(\x)});
\draw[black, thick] (0,0) -- (1,0) node[above left] {\tiny $c$};
\node[above] at (0,0) {\tiny $c-1$};
\end{tikzpicture} \\
&=\left(-\frac{[c+2]}{[c+1]}A^{2c}-\frac{[c]}{[c+1]}A^{-2(c+2)}\right)
\begin{tikzpicture}[scale=1.5, baseline=(current bounding box.center)]
\clip(-1.25,-0.5) rectangle (1.25,0.5);
\draw[black, thick] (-1,0) -- (1,0);
\node[above] at (0,0) {\tiny $c$};
\end{tikzpicture} \\
&=-(A^{2c+2}+A^{-2c-2})
\begin{tikzpicture}[scale=1.5, baseline=(current bounding box.center)]
\clip(-1.25,-0.5) rectangle (1.25,0.5);
\draw[black, thick] (-1,0) -- (1,0);
\node[above] at (0,0) {\tiny $c$};
\end{tikzpicture}
\end{align*}
\caption{A calculation with fusion rules.\label{fig:fusioncalculation}}
\end{figure}

It follows from Proposition~\ref{prop:tracegammadelta} that the quantum trace map induces a well defined injective $\mathbb{C}[A^{\pm1}]$-algebra homomorphism $\Sk_{A,\boldsymbol{\lambda}}(S)\rightarrow\mathscr{Z}_{A,\boldsymbol{C}}(\Gamma)$ mapping $\lambda_j\mapsto A^{-2}C_j^{-2}$. Abusing notation, we denote this homomorphism also by~$\Upsilon$.

\begin{proposition}
\label{prop:tracesigma}
Let $\sigma_{i,i+1}$ be defined as in Figure~\ref{fig:generators}. Then the quantum trace map~$\Upsilon$ takes values 
\[
\Upsilon(\sigma_{i,i+1})=E_i^2H_2+H_0+E_i^{-2}H_{-2}.
\]
If $1<i<n-1$ then the coefficients are given by 
\begin{align*}
H_2 &= -\frac{u(A^2C_i^{-1}Q_{i-1}Q_i)\,u(A^2C_{i+1}^{-1}Q_{i+1}Q_i)\,u(C_iQ_{i-1}Q_i^{-1})\,u(C_{i+1}Q_{i+1}Q_i^{-1})}{u(A^2Q_i^2)\,u(A^4Q_i^2)}, \\
H_{-2} &= -\frac{u(A^2C_iQ_{i-1}Q_i)\,u(C_iQ_{i-1}^{-1}Q_i)\,u(A^2C_{i+1}Q_{i+1}Q_i)\,u(C_{i+1}Q_{i+1}^{-1}Q_i)}{u(Q_i^2)\,u(A^2Q_i^2)}, \\
H_0 &= -H_2-H_{-2}-A^2C_i^2C_{i+1}^2-A^{-2}C_i^{-2}C_{i+1}^{-2}
\end{align*}
where $u(x)=x-x^{-1}$. In the cases $i=1$ and $i=n-1$, the coefficients are given by the same expressions with $Q_0$ replaced by $C_0$ and with $Q_n$ replaced by $C_{n+1}$, respectively.
\end{proposition}

\begin{proof}
We will prove the proposition assuming $1<i<n-1$. The cases $i=1$ and $i=n-1$ follow by the same argument, replacing the $Q$ variables by $C$ variables where appropriate. To simplify notation in the following, we will also write $\sigma=\sigma_{i,i+1}$ for the curve defined by Figure~\ref{fig:generators}.

Using the fusion rules, we can find rational functions $F_k^\sigma$ such that there is an expansion 
\begin{equation}
\label{eqn:sigmaexpansion}
\sigma\cdot\varphi_c=\sum_kF_k^\sigma(A^{c(i)})\varphi_{c+k}
\end{equation}
for $c$ in a good set as in Proposition~\ref{prop:skeinalgebraaction}. Note that here we consider the action on the \emph{unnormalized} basis elements $\varphi_c$. By Proposition~\ref{prop:quantumtraceexists}, we get an algebra homomorphism $\Upsilon_{\mathrm{DS}}:\Sk_A(S)\rightarrow\mathscr{Z}_{A,\mathbf{C}}(\Gamma)$ given by 
\[
\Upsilon_{\mathrm{DS}}(\sigma)=\sum_kE^kF_k^\sigma(Q_i,C_l).
\]
Here the subscript $\mathrm{DS}$ emphasizes that this unnormalized case was considered by Detcherry and Santharoubane in~\cite{DS25}. By equation~(5) in Section~3 of~\cite{DS25}, one has 
\[
\Upsilon_{\mathrm{DS}}(\sigma)=E_i^2G_2+G_0+E_i^{-2}G_{-2}
\]
where 
\begin{align*}
G_2 &= -u(C_iQ_{i-1}Q_i^{-1})\,u(C_{i+1}Q_{i+1}Q_i^{-1}), \\
G_{-2} &= -u(A^2C_iQ_{i-1}Q_i)\,u(C_iQ_{i-1}^{-1}Q_i)\,u(C_i^{-1}Q_{i-1}Q_i) \\
& \qquad\cdot u(A^2C_{i+1}Q_{i+1}Q_i)\,u(C_{i+1}Q_{i+1}^{-1}Q_i)\,u(C_{i+1}^{-1}Q_{i+1}Q_i) \\
& \qquad\cdot u(A^{-2}Q_i^2)^{-1}\,u(Q_i^2)^{-2}\,u(A^2Q_i^2)^{-1},
\end{align*}
and $G_0$ is a rational expression in the~$C_j$ and~$Q_j$.

Now to compute the quantum trace $\Upsilon(\sigma)$, we must replace the basis elements~$\varphi_c$ and~$\varphi_{c+k}$ in~\eqref{eqn:sigmaexpansion} by their normalized versions~$\widetilde{\varphi}_c=\kappa_c\,\varphi_c$ and~$\widetilde{\varphi}_{c+k}=\kappa_{c+k}\,\varphi_{c+k}$, respectively. We must therefore rescale the coefficient $F_k^\sigma(A^{c(i)})$ in~\eqref{eqn:sigmaexpansion} by the factor $\kappa_c/\kappa_{c+k}$. It then follows from Proposition~\ref{prop:quantumtraceexists} that the quantum trace in our normalization is 
\[
\Upsilon(\sigma)=E_i^2H_2+H_0+E_i^{-2}H_{-2}
\]
where 
\begin{align*}
H_2 &= \frac{u(A^2C_i^{-1}Q_{i-1}Q_i)\,u(A^2C_{i+1}^{-1}Q_{i+1}Q_i)}{u(A^2Q_i^2)\,u(A^4Q_i^2)}G_2, \\
H_{-2} &= \frac{u(Q_i^2)\,u(A^{-2}Q_i^2)}{u(C_i^{-1}Q_{i-1}Q_i)\,u(C_{i+1}^{-1}Q_{i+1}Q_i)}G_{-2}.
\end{align*}
These agree with the formulas in the proposition. For the constant term $H_0$, we refer to Proposition~3.4 of~\cite{MP15}. There the authors compute the quantum trace using different values for the normalizing factor~$\kappa_c$, but this is immaterial for computing the constant term. They also apply a shift to the coloring. If we translate the result of~\cite{MP15} into our notation, we get exactly the formula for~$H_0$ in the proposition.
\end{proof}

\subsection{The polynomial representation}
\label{sec:ThePolynomialRepresentation}

We now describe a change of variables that we will use to relate the relative skein algebra of~$S_{0,n+2}$ to a quantized Coulomb branch. Let $\mathbf{t}=(t_0^{\frac{1}{2}},\dots,t_{n+1}^{\frac{1}{2}})$ be a tuple of formal variables and write $\mathscr{X}_{q,\mathbf{t}}^0$ for the $\mathbb{C}(q^{\frac{1}{2}},t_0^{\frac{1}{2}},\dots,t_{n+1}^{\frac{1}{2}})$-algebra generated by variables $X_i^{\pm\frac{1}{2}}$, $\varpi_i^{\pm1}$ for $i=1,\dots,n-1$ subject to the relations $\varpi_iX_j^{\frac{1}{2}}=q^{\delta_{ij}/2}X_j^{\frac{1}{2}}\varpi_i$ and $[X_i^{\frac{1}{2}},X_j^{\frac{1}{2}}]=[\varpi_i,\varpi_j]=0$ for all~$i$,~$j$. We let $\mathscr{X}_{q,\mathbf{t}}$ be the localization of~$\mathscr{X}_{q,\mathbf{t}}^0$ at the multiplicative subset generated by the expressions $q^{m/2}X_i-q^{-m/2}X_i^{-1}$ for $m\in\mathbb{Z}$.

\begin{lemma}
\label{lem:antihomoms}
\leavevmode
\begin{enumerate}
\item There is a $\mathbb{C}$-algebra anti-isomorphism $*:\mathscr{Z}_{A,\mathbf{C}}(\Gamma)\rightarrow\mathscr{X}_{q,\mathbf{t}}$ mapping 
\[
A\mapsto q^{\frac{1}{2}}, \quad C_i\mapsto q^{-\frac{1}{2}}t_i^{-\frac{1}{2}}, \quad Q_i\mapsto q^{-\frac{1}{2}}X_i^{\frac{1}{2}}, \quad E_i\mapsto q^{-\frac{1}{2}}X_i^{-1}\varpi_i.
\]
\item There is a $\mathbb{C}[\lambda_0^{\pm1},\dots,\lambda_{n+1}^{\pm1}]$-algebra anti-automorphism $\dagger:\Sk_{A,\boldsymbol{\lambda}}(S)\rightarrow\Sk_{A,\boldsymbol{\lambda}}(S)$ that fixes the isotopy class of any framed link in~$S\times[0,1]$ whose projection to~$S$ has no self-intersections and maps $A\mapsto A^{-1}$.
\end{enumerate}
\end{lemma}

\begin{proof}
The statement in part~(1) follows easily from the definitions of these algebras. It is well known (see Section~3.4 of~\cite{M16}) that there exists an anti-automorphism $\Sk_A(S)\rightarrow\Sk_A(S)$ that fixes the class of a framed link in $S\times[0,1]$ whose projection to~$S$ has no self-intersections and maps $A\mapsto A^{-1}$. This induces a map $\Sk_{A,\boldsymbol{\lambda}}(S)\rightarrow\Sk_{A,\boldsymbol{\lambda}}(S)$ as in part~(2).
\end{proof}

By Lemma~\ref{lem:antihomoms}, there is an injective $\mathbb{C}$-algebra homomorphism $\Sk_{A,\boldsymbol{\lambda}}(S_{0,n+2})\rightarrow\mathscr{X}_{q,\mathbf{t}}$ given by 
\[
\Phi\coloneqq*\circ\Upsilon\circ\dagger.
\]
We call this homomorphism the \emph{polynomial representation} of~$\Sk_{A,\boldsymbol{\lambda}}(S_{0,n+2})$ since it is analogous to the polynomial representation of the spherical double affine Hecke algebra of type~$(C_1^\vee,C_1)$ (see Section~2.4 of~\cite{AS24a} for details).

\begin{proposition}
\label{prop:skeinoperatorvalues}
The polynomial representation~$\Phi$ maps $A\mapsto q^{-\frac{1}{2}}$ and $\lambda_i\mapsto t_i$ for each~$i$. Let $\gamma_i$ and~$\sigma_{i,i+1}$ be the curves defined in Figure~\ref{fig:generators}. Then for $1<i<n-1$, the polynomial representation maps 
\begin{align*}
\gamma_i &\mapsto -X_i-X_i^{-1}, \\
\sigma_{i,i+1} &\mapsto T_i(X_i)(\tau_i-1)+T_i(X_i^{-1})(\tau_i^{-1}-1)-qt_it_{i+1}-q^{-1}t_i^{-1}t_{i+1}^{-1}
\end{align*}
where 
\[
T_i(x)=-q^{-1}t_i^{-1}t_{i+1}^{-1}\frac{(1-qt_iX_{i-1}x)(1-qt_iX_{i-1}^{-1}x)(1-qt_{i+1}X_{i+1}x)(1-qt_{i+1}X_{i+1}^{-1}x)}{(1-x^2)(1-q^2x^2)}
\]
and $\tau_i=q^{-2}X_i^{-2}\varpi_i^2$. In the cases $i=1$ and $i=n-1$, the image of~$\sigma_{i,i+1}$ is given by the same expression with $X_0$ replaced by $t_0$ and with $X_n$ replaced by $t_{n+1}$, respectively.
\end{proposition}

\begin{proof}
This follows from Propositions~\ref{prop:tracegammadelta} and~\ref{prop:tracesigma}, the definitions of~$*$ and~$\dagger$, and a straightforward calculation.
\end{proof}

\section{Quantized Coulomb branches and categorification}
\label{sec:QuantizedCoulombBranchesAndCategorification}

In this section, we review some general facts about the quantized $K$-theoretic Coulomb branches of Braverman, Finkelberg, and Nakajima. We explain how such a quantized Coulomb branch arises as the Grothendieck ring of a monoidal category.

\subsection{The variety of triples}
\label{sec:TheVarietyOfTriples}

In the following, we consider a complex, connected, reductive algebraic group~$G$, which we refer to as the \emph{gauge group} of our theory. We write~$G_{\mathcal{K}}$ and~$G_{\mathcal{O}}$ for the $\mathcal{K}$- and $\mathcal{O}$-valued points of~$G$, respectively, where $\mathcal{K}=\mathbb{C}(\!( z)\!)$ and $\mathcal{O}=\mathbb{C}\llbracket z\rrbracket$. Then the \emph{affine Grassmannian} of~$G$ is defined as 
\[
\Gr_G\coloneqq G_{\mathcal{K}}/G_{\mathcal{O}}.
\]
This object can be understood as an ind-scheme representing an explicit functor of points (see Section~2 of~\cite{BFN18}), and in the following, we will take the reduced ind-scheme structure.

Suppose now that we are given a complex representation~$N$ of~$G$. We will write $N_{\mathcal{K}}=N\otimes_{\mathbb{C}}\mathcal{K}$ and $N_{\mathcal{O}}=N\otimes_{\mathbb{C}}\mathcal{O}$. Following Braverman, Finkelberg, and Nakajima~\cite{BFN18}, we consider the space 
\[
\mathcal{T}_{G,N}\coloneqq G_{\mathcal{K}}\times_{G_{\mathcal{O}}}N_{\mathcal{O}}
\]
where the right hand side denotes the quotient of $G_{\mathcal{K}}\times N_{\mathcal{O}}$ by the equivalence relation~$\sim$ defined by $(g,n)\sim(gb^{-1},bn)$ for every $b\in G_{\mathcal{O}}$. We write $[g,n]$ for the equivalence class of a pair $(g,n)\in G_{\mathcal{K}}\times N_{\mathcal{O}}$ in this space~$\mathcal{T}_{G,N}$. Given such a pair, one in general has $gn\in N_{\mathcal{K}}$, and we define the \emph{variety of triples} to be the subspace of $\mathcal{T}_{G,N}$ given by 
\[
\mathcal{R}_{G,N}\coloneqq\{[g,n]\in\mathcal{T}_{G,N}:gn\in N_{\mathcal{O}}\}.
\]
The spaces $\mathcal{T}_{G,N}$ and $\mathcal{R}_{G,N}$ can be understood as ind-schemes with reduced ind-scheme structure. Using the functor of points description in Section~2 of~\cite{BFN18}, one can view $\mathcal{R}_{G,N}$ as a moduli space parametrizing triples of gauge theoretic data on a formal disk, but we will not need this fact here. When there is no possibility of confusion, we will simply write $\mathcal{T}=\mathcal{T}_{G,N}$ and $\mathcal{R}=\mathcal{R}_{G,N}$.

There is an action of $G_{\mathcal{K}}$ on~$\mathcal{T}_{G,N}$ by left multiplication, and the subgroup $G_{\mathcal{O}}$ preserves the variety of triples~$\mathcal{R}_{G,N}$. In our examples, the action of $G$ on~$N$ will extend to an action of~$\widetilde{G}=G\times F$ where $F$ is an algebraic group called the \emph{flavor symmetry group}. We can then extend the action of~$G_{\mathcal{O}}$ on~$\mathcal{R}_{G,N}$ to an action of~$\widetilde{G}_{\mathcal{O}}$. We will also consider the natural action of $\mathbb{C}^*$ on~$\mathcal{R}_{G,N}$ where a complex number $t^\frac{1}{2}\in\mathbb{C}^*$ maps a point $[g(z),n(z)]\in\mathcal{R}_{G,N}$ to $[g(tz),t^{\frac{1}{2}}n(tz)]$. The $t^{\frac{1}{2}}$ factor in this formula comes from the convention in Section~2(i) of~\cite{BFN18}.

Note that there is a natural action of $G_{\mathcal{O}}$ on the affine Grassmannian~$\Gr_G$. It is well known that $\Gr_G$ is a union 
\[
\Gr_G=\bigsqcup_{\lambda\in\Lambda^+}\Gr_G^{\lambda}
\]
of $G_{\mathcal{O}}$-orbits $\Gr_G^\lambda$ indexed by the set $\Lambda^+$ of dominant coweights of~$G$. The closure of a $G_{\mathcal{O}}$-orbit is a union $\overline{\Gr}_G^\lambda=\bigsqcup_{\mu\leq\lambda}\Gr_G^{\mu}$ where $\leq$ is the standard partial order on~$\Lambda^+$. Let us write $\mathcal{R}_{\leq\lambda}\coloneqq\pi^{-1}(\overline{\Gr}_G^\lambda)$ and $\mathcal{R}_\lambda\coloneqq\pi^{-1}(\Gr_G^\lambda)$ where $\pi:\mathcal{R}\rightarrow\Gr_G$ is the projection $[g,n]\mapsto gG_{\mathcal{O}}$. Then $\mathcal{R}_\lambda$ is open in~$\mathcal{R}_{\leq\lambda}$ and we write $\mathcal{R}_{<\lambda}\coloneqq\mathcal{R}_{\leq\lambda}\setminus\mathcal{R}_\lambda$ for the complement. The spaces $\mathcal{R}_{\leq\lambda}$, $\mathcal{R}_{\lambda}$, and $\mathcal{R}_{<\lambda}$ are $\widetilde{G}_{\mathcal{O}}\rtimes\mathbb{C}^*$-invariant, and we recover the variety of triples as the direct limit $\mathcal{R}=\displaystyle\lim_{\longrightarrow}\mathcal{R}_{\leq\lambda}$.

\subsection{The equivariant derived category}
\label{sec:TheEquivariantDerivedCategory}

Let $(I,\leq)$ be directed set. Suppose $\mathcal{X}=\displaystyle\lim_{\longrightarrow}\mathcal{X}_\lambda$ is an ind-scheme written as a direct limit of schemes $\mathcal{X}_\lambda$ for~$\lambda\in I$ with compatible closed embeddings $i_{\lambda\mu}:\mathcal{X}_\lambda\hookrightarrow\mathcal{X}_\mu$ for $\lambda\leq\mu$. Suppose $\Gamma$ is a pro-linear algebraic group \cite[Definition~1.4.2]{VV10} acting on each~$\mathcal{X}_\lambda$ so that the maps $i_{\lambda\mu}$ are $\Gamma$-equivariant. We will assume that $\mathcal{X}$ is ind-coherent in the sense of~\cite[Definition~1.3.4(b)]{VV10} and an admissible ind-$\Gamma$-scheme in the sense of~\cite[Definition~1.4.5(c)]{VV10}. The ind-coherence condition implies that the embedding $i_{\lambda\mu}$ induces an exact functor 
\begin{equation}
\label{eqn:inducedfunctor}
(i_{\lambda\mu})_*:\Db\Coh^{\Gamma}(\mathcal{X}_{\lambda})\rightarrow \Db\Coh^{\Gamma}(\mathcal{X}_{\mu})
\end{equation}
where we write $\Db\Coh^{\Gamma}(\mathcal{Z})$ for the bounded derived category of $\Gamma$-equivariant coherent sheaves on a $\Gamma$-scheme~$\mathcal{Z}$.

\begin{definition}[\cite{VV10}, Section~1.1.2]
\label{def:2colimit}
Let $(I,\leq)$ be a directed set. Suppose we have a category $\mathcal{C}_\lambda$ for each $\lambda\in I$ and a functor $j_{\lambda\mu}:\mathcal{C}_\lambda\rightarrow\mathcal{C}_\mu$ whenever $\lambda\leq\mu$. Then the \emph{2-colimit} of this family of categories is the category $\mathcal{C}=\twocolim_\lambda\mathcal{C}_\lambda$ where 
\begin{enumerate}
\item The objects of $\mathcal{C}$ are pairs $(\lambda,\mathcal{F}_\lambda)$ where $\lambda\in I$ and $\mathcal{F}_\lambda$ is an object of~$\mathcal{C}_\lambda$.
\item The set of morphisms from $(\lambda,\mathcal{F}_\lambda)$ to~$(\mu,\mathcal{F}_\mu)$ in~$\mathcal{C}$ is the colimit 
\[
\Hom_\mathcal{C}\left((\lambda,\mathcal{F}_\lambda),(\mu,\mathcal{F}_\mu)\right)=\displaystyle\lim_{\substack{\longrightarrow \\ \nu\geq\lambda,\mu}}\Hom_{\mathcal{C}_\nu}\left(j_{\lambda\nu}(\mathcal{F}_\lambda),j_{\mu\nu}(\mathcal{F}_\mu)\right)
\]
in the category of sets.
\end{enumerate}
\end{definition}

In particular, we can apply the construction of Definition~\ref{def:2colimit} to the categories $\Db\Coh^{\Gamma}(\mathcal{X}_\lambda)$ related by the functors~\eqref{eqn:inducedfunctor}. We thus define the bounded derived category of $\Gamma$-equivariant coherent sheaves on~$\mathcal{X}$ as the 2-colimit 
\[
\Db\Coh^{\Gamma}(\mathcal{X})\coloneqq\twocolim_\lambda \Db\Coh^{\Gamma}(\mathcal{X}_{\lambda}).
\]

\begin{proposition}[\cite{VV10}, Proposition~1.1.4]
\label{prop:abeliancategory}
Let $(I,\leq)$ be a directed set. If $\mathcal{C}_{\lambda}$ is a triangulated category for each $\lambda\in I$ and $j_{\lambda\mu}:\mathcal{C}_\lambda\rightarrow\mathcal{C}_\mu$ is an exact functor whenever $\lambda\leq\mu$, then the category $\mathcal{C}=\twocolim_\lambda\mathcal{C}_\lambda$ in Definition~\ref{def:2colimit} is triangulated with Grothendieck group 
\[
K_0(\mathcal{C})=\lim_{\displaystyle\longrightarrow} K_0(\mathcal{C}_\lambda).
\]
\end{proposition}

In this paper, we complexify so that all equivariant $K$-groups are complex vector spaces. From Proposition~\ref{prop:abeliancategory}, we see that the category $\Db\Coh^{\Gamma}(\mathcal{X})$ is triangulated. The $\Gamma$-equivariant $K$-theory of~$\mathcal{X}$ can be defined as the Grothendieck group 
\[
K^{\Gamma}(\mathcal{X})=K_0\left(\Db\Coh^{\Gamma}(\mathcal{X})\right)
\]
of this category. In particular, taking $I=\Lambda^+$ to be the set of dominant coweights with its standard partial order~$\leq$, we can define the derived category $\Db\Coh^{\widetilde{G}_{\mathcal{O}}\rtimes\mathbb{C}^*}(\mathcal{R}_{G,N})$ and equivariant $K$-theory $K^{\widetilde{G}_{\mathcal{O}}\rtimes\mathbb{C}^*}(\mathcal{R}_{G,N})$ of the variety of triples.

\subsection{Monoidal structure}

Following Braverman, Finkelberg, and Nakajima~\cite{BFN18}, we will describe a monoidal product on $\Db\Coh^{\widetilde{G}_{\mathcal{O}}\rtimes\mathbb{C}^*}(\mathcal{R}_{G,N})$ called the convolution product. Let $\widetilde{G}_{\mathcal{K}}^{\mathcal{O}}=G_{\mathcal{K}}\times F_{\mathcal{O}}$. Then we have the identifications $\mathcal{T}=\widetilde{G}_{\mathcal{K}}^{\mathcal{O}}\times_{\widetilde{G}_{\mathcal{O}}}N_{\mathcal{O}}$ and $\mathcal{R}=\{[g,n]\in\widetilde{G}_{\mathcal{K}}^{\mathcal{O}}\times_{\widetilde{G}_{\mathcal{O}}}N_{\mathcal{O}}:gn\in N_{\mathcal{O}}\}$, and we have the following commutative diagram: 
\[
\xymatrix{
\mathcal{R}\times\mathcal{R} \ar[d] & p^{-1}(\mathcal{R}\times\mathcal{R}) \ar[l] \ar[d] \ar[r] & q(p^{-1}(\mathcal{R}\times\mathcal{R})) \ar[d] \ar[r] & \mathcal{R} \ar[d] \\
\mathcal{T}\times\mathcal{R} & \widetilde{G}_{\mathcal{K}}^{\mathcal{O}}\times\mathcal{R} \ar[l]_-{p} \ar[r]^-{q} & \widetilde{G}_{\mathcal{K}}^{\mathcal{O}}\times_{\widetilde{G}_{\mathcal{O}}}\mathcal{R} \ar[r]^-{m} & \mathcal{T}.
}
\]
In this diagram, the vertical arrows are inclusions of closed subspaces, while the horizontal arrows on the bottom row are given by 
\[
\left([g_1,g_2n],[g_2,n]\right)\stackrel{p}{\longmapsfrom}\left(g_1,[g_2,n]\right)\stackrel{q}{\longmapsto}\left[g_1,[g_2,n]\right]\stackrel{m}{\longmapsto} [g_1g_2,n].
\]
We will consider the following group actions on the spaces appearing in the bottom row of this diagram:
\begin{enumerate}
\item $\widetilde{G}_{\mathcal{O}}\times\widetilde{G}_{\mathcal{O}}$ acts on $\mathcal{T}\times\mathcal{R}$ by 
\[
(g,h)\cdot([g_1,n_1],[g_2,n_2])=([gg_1,n_1],[hg_2,n_2]).
\]
\item $\widetilde{G}_{\mathcal{O}}\times\widetilde{G}_{\mathcal{O}}$ acts on $\widetilde{G}_{\mathcal{K}}^{\mathcal{O}}\times\mathcal{R}$ by 
\[
(g,h)\cdot(g_1,[g_2,n])=(gg_1h^{-1},[hg_2,n]).
\]
\item $\widetilde{G}_{\mathcal{O}}$ acts on $\widetilde{G}_{\mathcal{K}}^{\mathcal{O}}\times_{\widetilde{G}_{\mathcal{O}}}\mathcal{R}$ by 
\[
g\cdot[g_1,[g_2,n]]=[gg_1,[g_2,n]].
\]
\item $\widetilde{G}_{\mathcal{O}}$ acts on $\mathcal{T}$ by 
\[
g\cdot[g_0,n]=[gg_0,n].
\]
\end{enumerate}
If we view $\widetilde{G}_{\mathcal{O}}$ in~(3),~(4) as the first factor of $\widetilde{G}_{\mathcal{O}}\times\widetilde{G}_{\mathcal{O}}$, then $p$, $q$,~$m$ are $\widetilde{G}_{\mathcal{O}}\times\widetilde{G}_{\mathcal{O}}$-equivariant.

The spaces appearing in the above diagram are ind-schemes, and we can define their equivariant derived categories using the methods of Section~\ref{sec:TheEquivariantDerivedCategory}. We can also define induced functors between these categories. By the construction in Remark~3.9(3) of~\cite{BFN18}, the map $p$ in the top row of the diagram induces a functor 
\[
p^*:\Db\Coh^{G_{\mathcal{O}}\rtimes\mathbb{C}^*}(\mathcal{R})\times \Db\Coh^{G_{\mathcal{O}}\rtimes\mathbb{C}^*}(\mathcal{R})\longrightarrow \Db\Coh^{G_{\mathcal{O}}\times G_{\mathcal{O}}\rtimes\mathbb{C}^*}(p^{-1}(\mathcal{R}\times\mathcal{R})).
\]
Since we have $\widetilde{G}_{\mathcal{K}}^{\mathcal{O}}\times_{\widetilde{G}_{\mathcal{O}}}\mathcal{R}=(\widetilde{G}_{\mathcal{K}}^{\mathcal{O}}\times\mathcal{R})/\widetilde{G}_{\mathcal{O}}$, the pullback by the quotient map~$q$ is an equivalence $q^*:\Db\Coh^{\widetilde{G}_{\mathcal{O}}\rtimes\mathbb{C}^*}(\widetilde{G}_{\mathcal{K}}^{\mathcal{O}}\times_{\widetilde{G}_{\mathcal{O}}}\mathcal{R})\stackrel{\sim}{\rightarrow}\Db\Coh^{\widetilde{G}_{\mathcal{O}}\times\widetilde{G}_{\mathcal{O}}\rtimes\mathbb{C}^*}(\widetilde{G}_{\mathcal{K}}^{\mathcal{O}}\times\mathcal{R})$, which provides a functor 
\[
(q^*)^{-1}:\Db\Coh^{\widetilde{G}_{\mathcal{O}}\times\widetilde{G}_{\mathcal{O}}\rtimes\mathbb{C}^*}(p^{-1}(\mathcal{R}\times\mathcal{R}))\longrightarrow\Db\Coh^{\widetilde{G}_{\mathcal{O}}\rtimes\mathbb{C}^*}(q(p^{-1}(\mathcal{R}\times\mathcal{R}))).
\]
Finally, the multiplication map $m$ is ind-proper~\cite{BFN18}, and the pushforward along~$m$ is a functor 
\[
m_*:\Db\Coh^{\widetilde{G}_{\mathcal{O}}\rtimes\mathbb{C}^*}(q(p^{-1}(\mathcal{R}\times\mathcal{R})))\longrightarrow\Db\Coh^{\widetilde{G}_{\mathcal{O}}\rtimes\mathbb{C}^*}(\mathcal{R}).
\]
The \emph{convolution product} is defined as the composition 
\[
*\coloneqq m_*\circ(q^*)^{-1}\circ p^*:\Db\Coh^{\widetilde{G}_{\mathcal{O}}\rtimes\mathbb{C}^*}(\mathcal{R})\times \Db\Coh^{\widetilde{G}_{\mathcal{O}}\rtimes\mathbb{C}^*}(\mathcal{R})\longrightarrow\Db\Coh^{\widetilde{G}_{\mathcal{O}}\rtimes\mathbb{C}^*}(\mathcal{R})
\]
of the above functors.

\subsection{The quantized Coulomb branch}

We can now give the definition of the quantized Coulomb~branch. As discussed in Section~3 of~\cite{BFN18}, the convolution product on the equivariant derived category induces an associative multiplication on the vector space $K^{\widetilde{G}_{\mathcal{O}}\rtimes\mathbb{C}^*}(\mathcal{R}_{G,N})$, also denoted~$*$, making it into a noncommutative algebra over the ring $K^{\mathbb{C}^*}(\mathrm{pt})\cong\mathbb{C}[q^{\pm1}]$ by Section~5.2.1 of~\cite{CG97}. In the present paper, we will replace the loop rotation group $\mathbb{C}^*$ by its standard double cover and thereby extend scalars to the ring~$\mathbb{C}[q^{\pm\frac{1}{2}}]$.

\begin{definition}[\cite{BFN18}]
The \emph{quantized ($K$-theoretic) Coulomb branch} associated to the groups~$G$,~$F$, and representation~$N$ is the algebra $(K^{\widetilde{G}_{\mathcal{O}}\rtimes\mathbb{C}^*}(\mathcal{R}_{G,N}),*)$ with underlying vector space $K^{\widetilde{G}_{\mathcal{O}}\rtimes\mathbb{C}^*}(\mathcal{R}_{G,N})$ and product~$*$.
\end{definition}

As its name suggests, the quantized Coulomb branch is a deformation quantization of a Poisson~algebra with deformation parameter~$q$. This commutative algebra is given by the $\widetilde{G}_{\mathcal{O}}$-equivariant $K$-theory $K^{\widetilde{G}_{\mathcal{O}}}(\mathcal{R}_{G,N})$ with a convolution product~$*$ defined as above. Its spectrum is called the \emph{($K$-theoretic) Coulomb branch} (see Remark~3.14 in~\cite{BFN18}).

\subsection{The associated graded algebra}
\label{sec:TheAssociatedGradedAlgebra}

We now define a filtration analogous to the one in Section~\ref{sec:TheAssociatedGradedAlgebraSkein}. We consider the quantized Coulomb branch $\mathcal{A}=K^{\widetilde{G}_{\mathcal{O}}\rtimes\mathbb{C}^*}(\mathcal{R}_{G,N})$ associated to a gauge group~$G$, flavor symmetry group~$F$, and representation~$N$. For each dominant coweight $\lambda\in\Lambda^+$ of~$G$, we have subspaces $\mathcal{A}_{\leq\lambda}\coloneqq K^{\widetilde{G}_{\mathcal{O}}\rtimes\mathbb{C}^*}(\mathcal{R}_{\leq\lambda})$ and $\mathcal{A}_{<\lambda}\coloneqq K^{\widetilde{G}_{\mathcal{O}}\rtimes\mathbb{C}^*}(\mathcal{R}_{<\lambda})$ of this quantized Coulomb~branch. As in Section~6(i) of~\cite{BFN18}, the convolution product on~$\mathcal{A}$ maps 
\[
*:\mathcal{A}_{\leq\lambda}\times\mathcal{A}_{\leq\mu}\rightarrow\mathcal{A}_{\leq\lambda+\mu}
\]
so that $\mathcal{A}$ becomes a filtered algebra. We write 
\[
\gr^\bullet\mathcal{A}\coloneqq\bigoplus_{\lambda\in\Lambda^+}\mathcal{A}_{\leq\lambda}/\mathcal{A}_{<\lambda}
\]
for the associated graded algebra.

For any $\lambda\in\Lambda^+$, let us write $i:\mathcal{R}_{<\lambda}\hookrightarrow\mathcal{R}_{\leq\lambda}$ and $j:\mathcal{R}_\lambda\hookrightarrow\mathcal{R}_{\leq\lambda}$ for the inclusion maps. The fact that $\mathcal{R}_\lambda$ is a vector bundle over~$\Gr_G^\lambda$ implies that $K_1^{G_{\mathcal{O}}\rtimes\mathbb{C}^*}(\mathcal{R}_\lambda)=0$. From the long exact sequence in equivariant $K$-theory, we therefore obtain the short exact sequence 
\[
\xymatrix{ 
0 \ar[r] & K^{\widetilde{G}_{\mathcal{O}}\rtimes\mathbb{C}^*}(\mathcal{R}_{<\lambda}) \ar[r]^{i_*} & K^{\widetilde{G}_{\mathcal{O}}\rtimes\mathbb{C}^*}(\mathcal{R}_{\leq\lambda}) \ar[r]^{j^*} & K^{\widetilde{G}_{\mathcal{O}}\rtimes\mathbb{C}^*}(\mathcal{R}_{\lambda}) \ar[r] & 0,
}
\]
and hence $\mathcal{A}_{\leq\lambda}/\mathcal{A}_{<\lambda}\cong K^{\widetilde{G}_{\mathcal{O}}\rtimes\mathbb{C}^*}(\mathcal{R}_{\lambda})$. Let $T\subset G$ and $T_F\subset F$ be maximal tori, let $W$ be the Weyl~group of~$G$, let $W_\lambda\subset W$ be the stabilizer of~$\lambda$, and let us write $r_\lambda=[\mathcal{O}_{\mathcal{R}_\lambda}]$ for the class in $K^{\widetilde{G}_{\mathcal{O}}\rtimes\mathbb{C}^*}(\mathcal{R}_{\lambda})$ of the structure sheaf of~$\mathcal{R}_\lambda$.

\begin{lemma}
\label{lem:KtheoryRlambda}
There is an isomorphism $K^{\widetilde{G}_{\mathcal{O}}\rtimes\mathbb{C}^*}(\mathcal{R}_\lambda)\cong\mathbb{C}[T\times T_F\times\mathbb{C}^*]^{W_\lambda}\cdot r_\lambda$.
\end{lemma}

\begin{proof}
The space $\mathcal{R}_\lambda$ is a vector bundle over $\Gr_G^\lambda$, so by the Thom~isomorphism~theorem \cite[Theorem~5.4.17]{CG97}, there is an isomorphism $K^{\widetilde{G}_{\mathcal{O}}\rtimes\mathbb{C}^*}(\Gr_G^\lambda)\stackrel{\sim}{\rightarrow}K^{\widetilde{G}_{\mathcal{O}}\rtimes\mathbb{C}^*}(\mathcal{R}_\lambda)$ mapping $[\mathcal{O}_{\Gr_G^\lambda}]$ to~$r_\lambda$. We~have $\Gr_G^\lambda=G_{\mathcal{O}}\cdot z^\lambda$ where $z^\lambda$ is the image of~$z$ under the cocharacter~$\lambda$, and we write $P_\lambda\coloneqq G_{\mathcal{O}}\cap z^\lambda G_{\mathcal{O}}z^{-\lambda}\subset G_{\mathcal{O}}$ for the stabilizer of $z^\lambda$ in~$\Gr_G^\lambda$. By the induction formula~\cite[Section~5.2.16]{CG97}, there is an isomorphism $K^{\widetilde{G}_{\mathcal{O}}\rtimes\mathbb{C}^*}(\Gr_G^\lambda)\stackrel{\sim}{\rightarrow}K^{P_\lambda\times F\rtimes\mathbb{C}^*}(\{z^\lambda\})$ mapping $[\mathcal{O}_{\Gr_G^\lambda}]$ to~1. The group $P_\lambda$ has a decomposition $P_\lambda=L_\lambda\ltimes U_\lambda$ where $U_\lambda$ is pro-unipotent and $L_\lambda\coloneqq G\cap z^\lambda Gz^{-\lambda}$. The Weyl group of~$L_\lambda$ is~$W_\lambda$. By Section~5.2.18 and Theorem~6.1.4 of~\cite{CG97}, we therefore have 
\begin{align*}
K^{P_\lambda\times F\rtimes\mathbb{C}^*}(\mathrm{pt}) &\cong K^{L_\lambda\times F\rtimes\mathbb{C}^*}(\mathrm{pt}) \\
&\cong K^{F\rtimes\mathbb{C}^*}(\mathrm{pt})\otimes K^{L_\lambda}(\mathrm{pt}) \\
&\cong \mathbb{C}[T_F\times\mathbb{C}^*]\otimes\mathbb{C}[T]^{W_\lambda}.
\end{align*}
The lemma follows by composing the above isomorphisms.
\end{proof}

\begin{proposition}
\label{prop:associatedgradedbasis}
$\gr^\bullet\mathcal{A}$ has a basis consisting of elements of the form $fr_\lambda$ for $\lambda\in\Lambda^+$ and $f\in\mathbb{C}[T\times T_F\times \mathbb{C}^*]^{W_\lambda}$.
\end{proposition}

\begin{proof}
This follows from Lemma~\ref{lem:KtheoryRlambda} and the preceding discussion.
\end{proof}

We can also give a formula for multiplication in the algebra $\gr^\bullet\mathcal{A}$ following Section~4(ii) of~\cite{BFN18}. Write $N=\bigoplus_iN_{(\epsilon_i,\zeta_i)}$ where $N_{(\epsilon_i,\zeta_i)}$ is a weight space for $T\times T_F\times\mathbb{C}^*$ of $T$-weight~$\epsilon_i$ and $T_F$-weight~$\zeta_i$.  For any $\lambda$,~$\mu\in\Lambda^+$ and any index~$i$, define $A_{\lambda,\mu}^i\in\mathbb{C}[T\times T_F\times\mathbb{C}^*]^{W_{\lambda+\mu}}$ by the formula 
\[
A_{\lambda,\mu}^i=
\begin{cases}
\displaystyle\prod_{j=1}^{d(\epsilon_i(\lambda),\epsilon_i(\mu))}\left(1-e^{-\epsilon_i-\zeta_i}q^{-2(\epsilon_i(\lambda)-j+\frac{1}{2})}\right) &\text{if $\epsilon_i(\lambda)\geq0\geq\epsilon_i(\mu)$} \\
\displaystyle\prod_{j=1}^{d(\epsilon_i(\lambda),\epsilon_i(\mu))}\left(1-e^{-\epsilon_i-\zeta_i}q^{-2(\epsilon_i(\lambda)+j-\frac{1}{2})}\right) &\text{if $\epsilon_i(\lambda)\leq0\leq\epsilon_i(\mu)$} \\
1 &\text{otherwise}
\end{cases}
\]
where 
\[
d(a,b)=
\begin{cases}
0 &\text{if $a$,~$b$ have the same sign} \\
\min(|a|,|b|) &\text{if $a$,~$b$ have different signs}
\end{cases}
\]
for $a,b\in\mathbb{Z}$ and $e^\chi\in\mathbb{C}[T\times T_F\times\mathbb{C}^*]$ is the regular function corresponding to a character $\chi\in X^*(T\times T_F\times\mathbb{C}^*)$.

\begin{proposition}
\label{prop:AssociatedGradedMultiplication}
For any $\lambda$,~$\mu\in\Lambda^+$ and $f\in\mathbb{C}[T\times T_F\times\mathbb{C}^*]^{W_\lambda}$, $g\in\mathbb{C}[T\times T_F\times\mathbb{C}^*]^{W_\mu}$, we have 
\[
fr_\lambda*gr_\mu=\left(\prod_iA_{\lambda,\mu}^i\right)fg\,r_{\lambda+\mu}
\]
in the associated graded $\gr^\bullet\mathcal{A}$.
\end{proposition}

\begin{proof}
This is a $K$-theoretic version of a result of Braverman, Finkelberg, and Nakajima. In~\cite{BFN18} they consider the equivariant Borel--Moore homology $\mathcal{H}=H_\bullet^{\widetilde{G}\rtimes\mathbb{C}^*}(\mathcal{R})$ equipped with a natural convolution product~$*$. The algebra obtained in this way has a filtration by the subspaces $H_\bullet^{\widetilde{G}\rtimes\mathbb{C}^*}(\mathcal{R}_{\leq\lambda})$. In \cite[Proposition~6.2]{BFN18}, Braverman, Finkelberg, and Nakajima consider the homology classes $h_\lambda=[\mathcal{R}_{\leq\lambda}]$ as elements of the associated graded $\gr^\bullet\mathcal{H}$ and prove a multiplication formula $fh_\lambda*gh_\mu=(\prod_iC_{\lambda,\mu}^i)fgh_{\lambda+\mu}$ for coefficients $f$ and~$g$ where $C_{\lambda,\mu}^i$ is given explicitly. This factor $C_{\lambda,\mu}^i$ is the Euler class of a vector bundle defined in the proof of \cite[Theorem~4.1]{BFN18}. Since we consider $K$-theory rather than homology, we replace this Euler class by the $K$-theoretic Euler class, which can be shown by a case-by-case analysis to equal $A_{\lambda,\mu}^i$.
\end{proof}

\begin{corollary}
\label{cor:easymultiplication}
For all integers $k$,~$l\geq0$ and all~$\lambda\in\Lambda^+$, we have 
\[
r_{k\lambda}*r_{l\lambda}=r_{(k+l)\lambda}.
\]
\end{corollary}

\begin{proof}
If $\epsilon\in X^*(T\times T_F\times\mathbb{C}^*)$ is any character, then $\epsilon(k\lambda)$ and $\epsilon(l\lambda)$ have the same sign. In this case, the formula above implies $A_{k\lambda,l\lambda}=1$.
\end{proof}

\section{Genus zero Coulomb branches and their generators}
\label{sec:GenusZeroCoulombBranchesAndTheirGenerators}

In this section, we describe the quantized $K$-theoretic Coulomb branch that we associate to a compact oriented surface of genus zero with boundary, and we prove a result about generation of this algebra.

\subsection{The gauge and flavor groups}
\label{sec:TheGaugeAndFlavorGroups}

Let us consider the gauge group $G=G_1\times\dots\times G_{n-1}$ where $G_i=\mathrm{SL}_2$ for $i=1,\dots,n-1$ and the torus $T=T_1\times\dots\times T_{n-1}\subset G$ where $T_i\subset G_i$ is the diagonal subgroup. We write 
\[
\Lambda\coloneqq X_*(T)=\bigoplus_{i=1}^{n-1}\mathbb{Z}\alpha_i
\]
for the cocharacter lattice of~$T\subset G$ where $X_*(T_i)=\mathbb{Z}\alpha_i$. Thus any cocharacter $\lambda\in\Lambda$ can be written $\lambda=\sum_i\lambda_i\alpha_i$ for some~$\lambda_i\in\mathbb{Z}$, and we define the \emph{support} of $\lambda$ to be the set $\supp(\lambda)=\{i:\lambda_i\neq0\}$. If we write $X^*(T_i)=\mathbb{Z}\omega_i$ for the character lattice of~$T_i$, then we have 
\[
\mathbb{C}[T]\cong\mathbb{C}[X_1^{\pm1},\dots,X_{n-1}^{\pm1}]
\]
where $X_i=e^{\omega_i}$ is the function corresponding to $\omega_i$. Similarly, we take our flavor symmetry group to be a product $F=F_0\times\dots\times F_{n+1}$ where $F_i=\mathbb{C}^*$ for $i=0,\dots,{n+1}$. If we write $X^*(F_i)=\mathbb{Z}\eta_i$, then we have 
\[
\mathbb{C}[F]\cong\mathbb{C}[t_0^{\pm1},\dots,t_{n+1}^{\pm1}]
\]
where $t_i=e^{\eta_i}$ is the function corresponding to~$\eta_i$.

Let us write $N_j=\mathbb{C}^2$ for $j=0,\dots,n$. Then there is an action of $G_i$ on $N_j\otimes N_{j+1}$ where an element of $G_i$ acts by matrix multiplication on a tensor factor whose index is equal to~$i$ and trivially on a tensor factor whose index is not equal to~$i$. By taking the direct sum, we get an action of $G_i$ on the vector space 
\[
N=\bigoplus_{j=0}^{n-1}N_j\otimes N_{j+1}.
\]
Thus we get an action of~$G$ on this vector space. Similarly, for $i=1,\dots,n$, we let $F_i$ act on $N_j\otimes N_{j+1}$ by scalar multiplication if $j=i-1$ and trivially if $j\neq i-1$. We let $F_0\cong\left\{\left(\begin{smallmatrix}t & 0 \\ 0 & t^{-1}\end{smallmatrix}\right):t\in\mathbb{C}^*\right\}$ act on $N_j\otimes N_{j+1}$ by matrix multiplication on the first tensor factor if $j=0$ and trivially if $j\neq0$, and we let $F_{n+1}\cong\left\{\left(\begin{smallmatrix}t & 0 \\ 0 & t^{-1}\end{smallmatrix}\right):t\in\mathbb{C}^*\right\}$ act on $N_j\otimes N_{j+1}$ by matrix multiplication on the second tensor factor if $j=n-1$ and trivially if $j\neq n-1$. By taking the direct sum, we get an action of $F_i$ on~$N$ and hence an action of~$F$ on~$N$. It follows from these definitions that the set of $T$-weights is 
\[
\wt(N)\coloneqq\{\pm\omega_i\pm\omega_{i+1}:1\leq i\leq n-2\}\cup\{\pm\omega_1,\pm\omega_{n-1}\}
\]
where all signs are independent. We will use this information together with Proposition~\ref{prop:AssociatedGradedMultiplication} to compute products in the associated graded algebra $\gr^\bullet\mathcal{A}$ from Section~\ref{sec:TheAssociatedGradedAlgebra}.

\subsection{Generating the associated graded}
\label{sec:GeneratingTheAssociatedGraded}

Let us continue with the setup of the previous subsection and consider the algebra 
\[
\mathscr{G}\coloneqq\gr^\bullet\mathcal{A}\otimes_{\mathbb{C}[q^{\pm\frac{1}{2}},t_0^{\pm1},\dots,t_{n+1}^{\pm1}]}\mathbb{C}(q^{\frac{1}{2}},t_0,\dots,t_{n+1})
\]
obtained from $\gr^\bullet\mathcal{A}$ by extending scalars to rational functions in $q^{\frac{1}{2}},t_0,\dots,t_{n+1}$. It is the associated graded of $\mathscr{A}\coloneqq\mathcal{A}\otimes_{\mathbb{C}[q^{\pm\frac{1}{2}},t_0^{\pm1},\dots,t_{n+1}^{\pm1}]}\mathbb{C}(q^{\frac{1}{2}},t_0,\dots,t_{n+1})$. We also consider the $\mathbb{C}(q^{\frac{1}{2}},t_0,\dots,t_{n+1})$-subalgebra $\mathscr{B}\subset\mathscr{G}$ generated by the elements $X_i+X_i^{-1}$ and $r_{\alpha_i}=[\mathcal{O}_{\mathcal{R}_{\alpha_i}}]$ for $i=1,\dots,n-1$.

\begin{lemma}
\label{lem:rlambdacommutation}
If $\lambda=\sum_i\lambda_i\alpha_i\in\Lambda^+$ and $j\in\supp(\lambda)$, then 
\[
r_\lambda(X_j+X_j^{-1})=(q^{2\lambda_j}X_j+q^{-2\lambda_j}X_j^{-1})r_\lambda.
\]
\end{lemma}

Lemma~\ref{lem:rlambdacommutation} can be proved geometrically using the definition of the convolution product, but we will give a short proof in Section~\ref{sec:manylemmas} below.

\begin{lemma}
\label{lem:conditional}
If $r_\lambda\in\mathscr{B}$ then $fr_\lambda\in\mathscr{B}$ for all $f\in\mathbb{C}[T\times T_F\times\mathbb{C}^*]^{W_\lambda}$.
\end{lemma}

\begin{proof}
Any $\lambda\in\Lambda^+$ can be written $\lambda=\sum_i\lambda_i\alpha_i$ for $\lambda_i\geq0$. Then we have $\mathbb{C}[T\times T_F\times\mathbb{C}^*]^{W_\lambda}=\mathbb{C}[q^{\pm\frac{1}{2}},t_0^{\pm1}\dots,t_{n+1}^{\pm1},X_i+X_i^{-1},X_j^{\pm1}:i\not\in\supp(\lambda),j\in\supp(\lambda)]$. Since $X_i+X_i^{-1}\in\mathscr{B}$ for all~$i$ by definition, it is enough to show that if $fr_\lambda\in\mathscr{B}$ for some $f\in\mathbb{C}[T\times T_F\times\mathbb{C}^*]^{W_\lambda}$ then $X_j^{\pm}fr_\lambda\in\mathscr{B}$ for any $j\in\supp(\lambda)$. By Lemma~\ref{lem:rlambdacommutation}, we have 
\[
fr_\lambda(X_j+X_j^{-1})-q^{-2\lambda_j}(X_j+X_j^{-1})fr_\lambda=(q^{2\lambda_j}-q^{-2\lambda_j})X_jfr_\lambda.
\]
Since $\lambda_j\neq0$ we have 
\[
X_jfr_\lambda=\frac{1}{q^{2\lambda_j}-q^{-2\lambda_j}}\left(fr_\lambda(X_j+X_j^{-1})-q^{-2\lambda_j}(X_j+X_j^{-1})fr_\lambda\right)\in\mathscr{B}
\]
and 
\[
X_j^{-1}fr_\lambda=(X_j+X_j^{-1})fr_\lambda-X_jfr_\lambda\in\mathscr{B}
\]
as desired.
\end{proof}

By Proposition~\ref{prop:AssociatedGradedMultiplication}, we have $r_\lambda*r_\mu=r_{\lambda+\mu}$ if $\epsilon(\lambda)$ and $\epsilon(\mu)$ have the same sign for all $\epsilon\in\wt(N)$. Note that for any cocharacter $\lambda=\sum_i\lambda_i\alpha_i\in\Lambda$ we have 
\[
(\pm\omega_i\pm\omega_{i+1})(\lambda)=\pm\lambda_i\pm\lambda_{i+1}, \quad \pm\omega_1(\lambda)=\pm\lambda_1, \quad \pm\omega_{n-1}(\lambda)=\pm\lambda_{n-1}.
\]
If $\lambda$,~$\mu\in\Lambda^+$ are dominant cocharacters, then it follows that $\epsilon(\lambda)$ and $\epsilon(\mu)$ can have different signs for $\epsilon\in\wt(N)$ only when $\epsilon=\pm(\omega_i-\omega_{i+1})$. In the following, we will use the notation $\xi_i\coloneqq\omega_i-\omega_{i+1}$ and also write $\alpha_{i,j}=\alpha_i+\alpha_{i+1}+\dots+\alpha_j$ for $1\leq i<j\leq n-1$ and $\alpha_{i,i}=\alpha_i$.

\begin{lemma}
We have $r_{\alpha_{i,j}}\in\mathscr{B}$ for $1\leq i\leq j\leq n-1$.
\end{lemma}

\begin{proof}
We argue by induction on $k=j-i$. By definition of~$\mathscr{B}$, we have $r_{\alpha_{i,i}}\in\mathscr{B}$, so the statement in the lemma is true when $k=0$. Let us assume the statement is true when $k=K-1$, and consider indices $i$,~$j$ such that $j-i=K$. Then we have $r_{\alpha_{i,j-1}}\in\mathscr{B}$. One computes $\xi_{j-1}(\alpha_{i,j-1})=1$, $\xi_{j-1}(\alpha_j)=-1$, and 
\[
\xi_l(\alpha_{i,j-1})\xi_l(\alpha_j)=0 \quad\text{for $l\neq j-1$}.
\]
It then follows from Proposition~\ref{prop:AssociatedGradedMultiplication} that 
\begin{align*}
r_{\alpha_{i,j-1}}*r_{\alpha_j} &= \left(1-e^{-\xi_{j-1}-\eta_j}q^{-1}\right)\left(1-e^{\xi_{j-1}-\eta_j}q\right)r_{\alpha_{i,j}} \\
&= \left(1-q^{-1}t_j^{-1}X_{j-1}^{-1}X_j\right)\left(1-qt_j^{-1}X_{j-1}X_j^{-1}\right)r_{\alpha_{i,j}}
\end{align*}
and 
\[
(X_{j-1}r_{\alpha_{i,j-1}})*(X_j^{-1}r_{\alpha_j}) = \left(X_{j-1}X_j^{-1}-q^{-1}t_j^{-1}\right)\left(1-qt_j^{-1}X_{j-1}X_j^{-1}\right)r_{\alpha_{i,j}}.
\]
Similarly, 
\[
r_{\alpha_j}*r_{\alpha_{i,j-1}} = \left(1-q^{-1}t_j^{-1}X_{j-1}X_j^{-1}\right)\left(1-qt_j^{-1}X_{j-1}^{-1}X_j\right)r_{\alpha_{i,j}}
\]
and 
\[
(X_j^{-1}r_{\alpha_j})*(X_{j-1}r_{\alpha_{i,j-1}}) = \left(1-q^{-1}t_j^{-1}X_{j-1}X_j^{-1}\right)\left(X_{j-1}X_j^{-1}-qt_j^{-1}\right)r_{\alpha_{i,j}}.
\]
We can combine these into a single matrix equation 
\begin{equation}
\label{eqn:Coulombmatrix}
\begin{pmatrix}
r_{\alpha_{i,j-1}}*r_{\alpha_j} \\
r_{\alpha_j}*r_{\alpha_{i,j-1}} \\
(X_{j-1}r_{\alpha_{i,j-1}})*(X_j^{-1}r_{\alpha_j}) \\
(X_j^{-1}r_{\alpha_j})*(X_{j-1}r_{\alpha_{i,j-1}})
\end{pmatrix}
=
-t^{-1}
\underbrace{
\begin{pmatrix}
q & q^{-1} & \delta & 0 \\
q^{-1} & q & \delta & 0 \\
\delta & 0 & q^{-1} & q \\
\delta & 0 & q & q^{-1}
\end{pmatrix}}_P
\begin{pmatrix}
X_{j-1}X_j^{-1}r_{\alpha_{i,j}} \\
X_{j-1}^{-1}X_jr_{\alpha_{i,j}} \\
r_{\alpha_{i,j}} \\
X_{j-1}^2X_j^{-2}r_{\alpha_{i,j}}
\end{pmatrix}
\end{equation}
where $t\coloneqq t_j$ and $\delta\coloneqq-t-t^{-1}$. As in the proof of Proposition~\ref{prop:generators}, the determinant of~$P$ is not identically zero, so there is an inverse~$P^{-1}$ with coefficients in $\mathbb{C}(q,t_0,\dots,t_{n+1})$. In particular, we can write $r_{\alpha_{i,j}}$ as a polynomial in $r_{\alpha_{i,j-1}}*r_{\alpha_j}$, \ $r_{\alpha_j}*r_{\alpha_{i,j-1}}$, \ $(X_{j-1}r_{\alpha_{i,j-1}})*(X_j^{-1}r_{\alpha_j})$, and $(X_j^{-1}r_{\alpha_j})*(X_{j-1}r_{\alpha_{i,j-1}})$. From the definition of~$\mathscr{B}$ together with Lemma~\ref{lem:conditional} and our inductive hypothesis, we know that latter are all contained in~$\mathscr{B}$. Hence $r_{\alpha_{i,j}}\in\mathscr{B}$ and the lemma follows by induction.
\end{proof}

\begin{lemma}
\label{lem:generateassociatedgraded}
We have $\mathscr{B}=\mathscr{G}$. That is, $\mathscr{G}$ is generated as a $\mathbb{C}(q^{\frac{1}{2}},t_0,\dots,t_{n+1})$-algebra by the elements $X_i+X_i^{-1}$ and $r_{\alpha_i}$ for $i=1,\dots,n-1$.
\end{lemma}

\begin{proof}
By Proposition~\ref{prop:associatedgradedbasis}, we know that $\mathscr{G}$ is spanned by elements of the form $fr_\lambda$ for $\lambda\in\Lambda^+$ and $f\in\mathbb{C}[T\times T_F\times\mathbb{C}^*]^{W_\lambda}$. Thus it suffices to show that each of these elements is contained in $\mathscr{B}\subset\mathscr{G}$. In fact, by Lemma~\ref{lem:conditional} it suffices to show that $r_\lambda\in\mathscr{B}$ for every $\lambda\in\Lambda^+$. We will prove this by induction on $k=|\supp(\lambda)|$. We first note that if $k=1$ then $\lambda=m\alpha_i$ for some $m\in\mathbb{Z}_{\geq0}$ and some $i=1,\dots,n-1$, and therefore $r_\lambda=r_{\alpha_i}^m\in\mathscr{B}$ by Corollary~\ref{cor:easymultiplication}.

Let us assume that the desired statement is true for $k<K$, and suppose $\lambda\in\Lambda^+$ is a dominant coweight satisfying $|\supp(\lambda)|=K$. If $\supp(\lambda)$ is not an interval, that is, if there exists $i$ such that $\supp(\lambda)\subset[1,i-1]\cup[i+1,n-1]$ where $\supp(\lambda)\cap[1,i-1]\neq\emptyset$ and $\supp(\lambda)\cap[i+1,n-1]\neq\emptyset$, then we can write 
\[
\lambda=\lambda'+\lambda''
\]
where 
\[
\supp(\lambda')\subset[1,i-1], \quad \supp(\lambda'')\subset[i+1,n-1].
\]
Then for every $\epsilon\in\wt(N)$ we must have $\epsilon(\lambda')=0$ or $\epsilon(\lambda'')=0$, and Proposition~\ref{prop:AssociatedGradedMultiplication} implies $r_\lambda=r_{\lambda'}*r_{\lambda''}$. Since $|\supp(\lambda')|<K$ and $|\supp(\lambda'')|<K$, we have $r_{\lambda'}$,~$r_{\lambda''}\in\mathscr{B}$ by our assumption, and hence $r_\lambda\in\mathscr{B}$.

It remains to consider the case where the support $\supp(\lambda)$ is an interval so that we have $\supp(\lambda)=\{i,i+1,\dots,i+K-1\}$ for some~$i$. In this case, take $j\in\supp(\lambda)$ such that $\lambda_j\leq\lambda_i$ for all~$i\in\supp(\lambda)$. Then we can write 
\[
\lambda=\lambda_j\alpha_{i,i+K-1}+\lambda'
\]
where 
\[
\lambda'=\sum_{l=i}^{i+K-1}(\lambda_l-\lambda_j)\alpha_l.
\]
By definition of $\xi_l$, we have 
\[
\xi_l(\lambda_j\alpha_{i,i+K-1})=
\begin{cases}
0 &\text{if $l\neq i-1, i+K-1$} \\
-\lambda_j &\text{if $l=i-1$} \\
\lambda_j &\text{if $l=i+K-1$}
\end{cases}
\]
and also $\xi_{i-1}(\lambda')=-(\lambda_i-\lambda_j)$ and $\xi_{i+K-1}(\lambda')=\lambda_{i+K-1}-\lambda_j$. It follows from our assumption on~$j$ that $\xi_l(\lambda_j\alpha_{i,i+K-1})$ and $\xi_l(\lambda')$ have the same sign for all~$l$. We have already observed that if $\epsilon\in\wt(N)$, then $\epsilon(\lambda)$ and $\epsilon(\lambda')$ can have different signs only when $\epsilon=\xi_l$ for some~$l$. Hence Proposition~\ref{prop:AssociatedGradedMultiplication} implies $r_\lambda=r_{\lambda_j\alpha_{i,i+K-1}}*r_{\lambda'}$. We know $r_{\lambda_j\alpha_{i,i+K-1}}=r_{\alpha_{i,i+K-1}}^{\lambda_j}\in\mathscr{B}$ by Corollary~\ref{cor:easymultiplication}. We also have $r_{\lambda'}\in\mathscr{B}$ by our inductive hypothesis since $|\supp(\lambda')|<K$. Hence $r_\lambda\in\mathscr{B}$ and the lemma follows by induction.
\end{proof}

\subsection{Generating the Coulomb branch}
\label{sec:GeneratingTheCoulombBranch}

Let $\mathcal{A}=K^{\widetilde{G}_{\mathcal{O}}\rtimes\mathbb{C}^*}(\mathcal{R}_{G,N})$ be the quantized Coulomb branch of the previous subsection. By the basic properties of equivariant $K$-theory discussed in Sections~5.2.1 and~6.1 of~\cite{CG97}, this quantized Coulomb branch contains the subalgebra 
\begin{align*}
K^{\widetilde{G}\times\mathbb{C}^*}(\mathrm{pt}) &\cong K^{T\times F\times\mathbb{C}^*}(\mathrm{pt})^W \\
&\cong \mathbb{C}[q^{\pm\frac{1}{2}},t_0^{\pm1},\dots,t_{n+1}^{\pm1},X_1^{\pm1},\dots,X_{n-1}^{\pm1}]^W \\
&\cong \mathbb{C}[q^{\pm\frac{1}{2}},t_0^{\pm1},\dots,t_{n+1}^{\pm1},X_i+X_i^{-1}:i=1,\dots,n-1]
\end{align*}
where $\mathrm{pt}$ denotes the variety consisting of a single point and $W$ is the Weyl group of~$G$ acting naturally on~$\mathbb{C}[T]\cong\mathbb{C}[X_1^{\pm1},\dots,X_{n-1}^{\pm1}]$. Abusing notation, we write $X_i+X_i^{-1}$ and $r_\lambda=[\mathcal{O}_{\mathcal{R}_\lambda}]$ for elements of~$\mathcal{A}$ and also for the corresponding elements of the associated graded algebra $\gr^\bullet\mathcal{A}$.

The elements $q^{\frac{1}{2}},t_0,\dots,t_{n+1}$ are central in~$\mathcal{A}$, and so $\mathcal{A}$ can be considered as an algebra over $\mathbb{C}[q^{\pm\frac{1}{2}},t_0^{\pm1},\dots,t_{n+1}^{\pm1}]$. In the following, we will write 
\[
\mathscr{A}=\mathcal{A}\otimes_{\mathbb{C}[q^{\pm\frac{1}{2}},t_0^{\pm1},\dots,t_{n+1}^{\pm1}]}\mathbb{C}(q^{\frac{1}{2}},t_0,\dots,t_{n+1})
\]
for the algebra obtained from~$\mathcal{A}$ by extending scalars to rational functions in $q^{\frac{1}{2}},t_0,\dots,t_{n+1}$.

\begin{proposition}
\label{prop:generateCoulomb}
$\mathscr{A}$ is generated as a $\mathbb{C}(q^{\frac{1}{2}},t_0,\dots,t_{n+1})$-algebra by the elements $X_i+X_i^{-1}$ and~$r_{\alpha_i}$ for $i=1,\dots,n-1$.
\end{proposition}

\begin{proof}
The algebra $\mathscr{A}$ is a filtered algebra with filtration induced by the one on~$\mathcal{A}$. By Lemma~\ref{lem:generateassociatedgraded}, the associated graded algebra $\gr^\bullet\mathscr{A}=\gr^\bullet\mathcal{A}\otimes\mathbb{C}(q^{\frac{1}{2}},t_0,\dots,t_{n+1})$ is generated by $X_i+X_i^{-1}$ and~$r_{\alpha_i}$ for $i=1,\dots,n-1$. This implies the desired statement.
\end{proof}

\section{A representation of the quantized Coulomb branch}
\label{sec:ARepresentationOfTheQuantizedCoulombBranch}

In this section, we describe a class of quantized Coulomb branches associated to quivers, and we explain how to obtain a faithful representation of such a quantized Coulomb branch.

\subsection{Quiver gauge theories}
\label{sec:QuiverGaugeTheories}

Let $Q=(Q_0,Q_1,s,t)$ be a finite quiver where $Q_0$ is the set of vertices, $Q_1$ is the set of arrows, and $s:Q_1\rightarrow Q_0$ and $t:Q_1\rightarrow Q_0$ are the maps sending an arrow to its source and target, respectively. In the following, we consider the vector space 
\begin{equation}
\label{eqn:quiverrep}
N=\bigoplus_{a\in Q_1}\Hom(\mathbb{C}^2,\mathbb{C}^2).
\end{equation}
We also fix a decomposition $Q_0=\mathcal{G}\sqcup\mathcal{F}$ of the set of vertices into a set~$\mathcal{G}$ of \emph{gauge nodes} and a set~$\mathcal{F}$ of \emph{framing nodes}. We define the gauge group 
\[
\mathbf{G}=\prod_{i\in\mathcal{G}}\mathrm{GL}_2,
\]
which acts naturally on the vector space~$N$. We write $\mathbf{F}_1\subset\prod_{i\in\mathcal{F}}\mathrm{GL}_2$ for the diagonal torus, which also acts naturally on~$N$. The quiver $Q$ determines a second torus~$\mathbf{F}_2=(\mathbb{C}^*)^{Q_1}$, which acts on~$N$ in such a way that the $\mathbb{C}^*$-factor corresponding to an arrow~$a\in Q_1$ acts by rescaling the corresponding hom-set in the direct sum~\eqref{eqn:quiverrep}. We define the flavor symmetry group to be the product 
\[
\mathbf{F}=\mathbf{F}_1\times\mathbf{F}_2.
\]
The actions of~$\mathbf{F}$ and~$\mathbf{G}$ on the vector space~$N$ commute, and therefore we get an action of the product $\mathbf{G}\times\mathbf{F}$ on this vector space.

In the following, we will consider the case where $Q_0=\{i_0,\dots,i_n\}$ and where $Q_1=\{a_1,\dots,a_n\}$ consists of an arrow $a_j:i_{j-1}\rightarrow i_j$ for $j$~odd and an arrow $a_j:i_j\rightarrow i_{j-1}$ for $j$~even. The underlying graph for the resulting quiver $Q$ is illustrated in Figure~\ref{fig:typeAquiver}. In such a picture, we represent gauge nodes by circles and represent framing nodes by squares. From this quiver, we get groups $\mathbf{G}$ and~$\mathbf{F}$ and a representation~$N$ of $\mathbf{G}\times\mathbf{F}$ as above. We note that the groups $G$ and~$F$ considered in Section~\ref{sec:TheGaugeAndFlavorGroups} are naturally subgroups of~$\mathbf{G}$ and~$\mathbf{F}$, respectively, and the restriction of the representation~$N$ to~$G\times F$ is the same as the representation in Section~\ref{sec:TheGaugeAndFlavorGroups} because $\Hom(\mathbb{C}^2,\mathbb{C}^2)\cong\mathbb{C}^2\otimes\mathbb{C}^2$ as $\mathrm{SL}_2$-representations.

\begin{figure}[ht]
\begin{center}
\begin{tikzpicture}
\clip(-1,-0.5) rectangle (11,1);
\node[label={\tiny $i_0$}] at (0,0) [minimum size=0.5cm, draw] (0) {};
\node[label={\tiny $i_1$}] at (2,0) [circle, minimum size=0.5cm, draw] (1) {};
\node[label={\tiny $i_2$}] at (4,0) [circle, minimum size=0.5cm, draw] (2) {};
\node at (6,0) [circle, minimum size=0.5cm] (3) {$\dots$};
\node[label={\tiny $i_{n-1}$}] at (8,0) [circle, minimum size=0.5cm, draw] (4) {};
\node[label={\tiny $i_n$}] at (10,0) [minimum size=0.5cm, draw] (5) {};
\draw (0)--(1) node [midway, above, fill=white] {\tiny $a_1$};
\draw (1)--(2) node [midway, above, fill=white] {\tiny $a_2$};
\draw (2)--(3) node [midway, above, fill=white] {\tiny $a_3$};
\draw (3)--(4) node [midway, above, fill=white] {\tiny $a_{n-1}$};
\draw (4)--(5) node [midway, above, fill=white] {\tiny $a_n$};
\end{tikzpicture}
\end{center}
\caption{Form of the quiver~$Q$.\label{fig:typeAquiver}}
\end{figure}
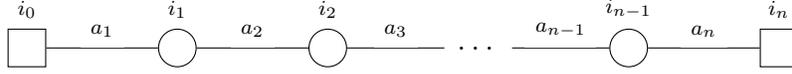

\subsection{Localization}
\label{sec:Localization}

We can construct a faithful polynomial representation of the quantized Coulomb~branch associated to a quiver $Q$ as above. Write $\mathbf{z}$ for a tuple consisting of a formal variable $z_{k,l}$ for each framing node $k\in\mathcal{F}$ and each $l\in\mathbb{Z}/2\mathbb{Z}$ and a formal variable~$z_a$ for each arrow $a\in Q_1$. Let $\mathcal{D}_{q,\mathbf{z}}^0$ be the $\mathbb{C}[q^{\pm\frac{1}{2}},z_{k,l}^{\pm1},z_a^{\pm1}]$-algebra generated by variables $D_{i,r}$, $w_{i,r}$ and their inverses for $i\in\mathcal{G}$ and $r\in\mathbb{Z}/2\mathbb{Z}$ subject to the commutation relations 
\[
[D_{i,r},D_{j,s}]=[w_{i,r},w_{j,s}]=0, \quad D_{i,r}w_{j,s}=q^{2\delta_{ij}\delta_{rs}}w_{j,s}D_{i,r}.
\]
Finally, let $\mathcal{D}_{q,\mathbf{z}}$ be the localization of~$\mathcal{D}_{q,\mathbf{z}}^0$ at the multiplicative set generated by the expressions $w_{i,r}-q^\mu w_{i,s}$ for $\mu\in\mathbb{Z}$ and $r\neq s$ and the expressions $1-q^\mu$ for $\mu\in\mathbb{Z}\setminus\{0\}$.

As explained in Section~3.3 of~\cite{AS24a}, the algebra $\mathcal{D}_{q,\mathbf{z}}^0$ is isomorphic to the quantized Coulomb~branch $K^{(\mathbf{T}\times\mathbf{F})_{\mathcal{O}}\rtimes\mathbb{C}^*}(\mathcal{R}_{\mathbf{T},0})$ associated to a maximal torus $\mathbf{T}\subset\mathbf{G}$ acting on the zero vector space, and the localization theorem in equivariant $K$-theory provides a $\mathbb{C}$-algebra embedding 
\begin{equation}
\label{eqn:localization}
K^{(\mathbf{G}\times\mathbf{F})_{\mathcal{O}}\rtimes\mathbb{C}^*}(\mathcal{R}_{\mathbf{G},N})\hookrightarrow\mathcal{D}_{q,\mathbf{z}}
\end{equation}
into the localization of this algebra. In the following, we will use this embedding to define a faithful representation of the quantized Coulomb branch from the previous section and study the operators corresponding to the generators of Section~\ref{sec:GeneratingTheCoulombBranch} in this representation.

\subsection{Monopole operators}
\label{sec:MonopoleOperators}

Suppose that $\mathbf{G}$ and $N$ are the gauge group and representation associated to a quiver~$Q$ as in Section~\ref{sec:QuiverGaugeTheories}. Then we have $\mathbf{G}=\prod_{i\in\mathcal{G}}\mathrm{GL}_2$. Let $\varpi_{i,1}$ be the first fundamental coweight of the $i$th $\mathrm{GL}_2$ factor so that $\varpi_{i,1}$ corresponds to the cocharacter $t\mapsto\left(\begin{smallmatrix}t & 0 \\ 0 & 1\end{smallmatrix}\right)$. It is known~\cite{FT19} that $\Gr_{\mathrm{GL}_2}^{\varpi_{i,1}}$ is closed and naturally identified with the variety $\mathbb{P}^1$ of one-dimensional quotients of~$\mathbb{C}^2$. The latter has a tautological line bundle, and we write $\mathcal{E}_i$ for the pullback of this line bundle along the projection~$\mathcal{R}_{\varpi_{i,1}}\rightarrow\Gr_{\mathrm{GL}_2}^{\varpi_{i,1}}$. Similarly, let $\varpi_{i,1}^*$ be the cocharacter $t\mapsto\left(\begin{smallmatrix}1 & 0 \\ 0 & t^{-1}\end{smallmatrix}\right)$ of the $i$th $\mathrm{GL}_2$ factor. Then $\Gr_{\mathrm{GL}_2}^{\varpi_{i,1}^*}$ is closed and naturally identified with the variety $\mathbb{P}^1$ of one-dimensional subspaces of~$\mathbb{C}^2$. We write $\mathcal{F}_i$ for the pullback of the tautological line bundle along~$\mathcal{R}_{\varpi_{i,1}^*}\rightarrow\Gr_{\mathrm{GL}_2}^{\varpi_{i,1}^*}$.

For any Laurent polynomial~$f\in\mathbb{C}[x^{\pm1}]$, we get a vector bundle $f(\mathcal{E}_i)$ on $\mathcal{R}_{\varpi_{i,1}}$ in a natural way. Similarly, we get a vector bundle $f(\mathcal{F}_i)$ on~$\mathcal{R}_{\varpi_{i,1}^*}$. Then we get vector bundles 
\begin{equation}
\label{eqn:vectorbundles}
f(\mathcal{E}_i)\otimes\mathcal{O}_{\mathcal{R}_{\varpi_{i,1}}}, \quad f(\mathcal{F}_i)\otimes\mathcal{O}_{\mathcal{R}_{\varpi_{i,1}^*}},
\end{equation}
which we can view as coherent sheaves on~$\mathcal{R}_{\mathbf{G},N}$. Their classes in $K^{(\mathbf{G}\times\mathbf{F})_{\mathcal{O}}\rtimes\mathbb{C}^*}(\mathcal{R}_{\mathbf{G},N})$ are called \emph{dressed minuscule monopole operators}, and the Laurent polynomial $f$ is called a \emph{dressing}. Here we will recall the formulas for these elements from~\cite{FT19} in the case where the Coulomb branch is defined by the quiver illustrated in Figure~\ref{fig:typeAquiver}. To simplify notation, we denote the generators of the algebra $\mathcal{D}_{q,\mathbf{z}}$ by $D_{j,\varepsilon}\coloneqq D_{i_j,\varepsilon}$, $w_{j,\varepsilon}\coloneqq w_{i_j,\varepsilon}$, $z_j\coloneqq z_{a_j}$, $z_{0,\varepsilon}\coloneqq z_{i_0,\varepsilon}$, and~$z_{n+1,\varepsilon}\coloneqq z_{i_n,\varepsilon}$. We also abbreviate $+\coloneqq +1$ and $-\coloneqq -1$. For any $m\in\mathbb{Z}$, we consider elements $E_{i,1}[x^m]$ and~$F_{i,1}[x^m]$ of~$\mathcal{D}_{q,\mathbf{z}}$ defined as follows: 
\begin{enumerate}
\item For any even integer $1<i<n-1$, we let 
\[
E_{i,1}[x^m] = \sum_{\varepsilon\in\{\pm1\}}w_{i,\varepsilon}^m\mathcal{P}_{i,\varepsilon}D_{i,\varepsilon}, \qquad
F_{i,1}[x^m] = \sum_{\varepsilon\in\{\pm1\}}q^{-2m}w_{i,\varepsilon}^m\mathcal{Q}_{i,\varepsilon}D_{i,\varepsilon}^{-1}
\]
where 
\begin{align*}
\mathcal{P}_{i,\varepsilon} &= \frac{(1-qz_iw_{i,\varepsilon}w_{i-1,+}^{-1})(1-qz_iw_{i,\varepsilon}w_{i-1,-}^{-1})(1-qz_{i+1}w_{i,\varepsilon}w_{i+1,+}^{-1})(1-qz_{i+1}w_{i,\varepsilon}w_{i+1,-}^{-1})}{1-w_{i,-\varepsilon}w_{i,\varepsilon}^{-1}}, \\
\mathcal{Q}_{i,\varepsilon} &= \frac{1}{1-w_{i,\varepsilon}w_{i,-\varepsilon}^{-1}}.
\end{align*}
\item For any odd integer $1<i<n-1$, we let 
\[
E_{i,1}[x^m] = \sum_{\varepsilon\in\{\pm1\}}w_{i,\varepsilon}^m\mathcal{P}_{i,\varepsilon}D_{i,\varepsilon}, \qquad
F_{i,1}[x^m] = \sum_{\varepsilon\in\{\pm1\}}q^{-2m}w_{i,\varepsilon}^m\mathcal{Q}_{i,\varepsilon}D_{i,\varepsilon}^{-1}
\]
where 
\begin{align*}
\mathcal{P}_{i,\varepsilon} &= \frac{1}{1-w_{i,-\varepsilon}w_{i,\varepsilon}^{-1}}, \\
\mathcal{Q}_{i,\varepsilon} &= \frac{(1-qz_iw_{i-1,+}w_{i,\varepsilon}^{-1})(1-qz_iw_{i-1,-}w_{i,\varepsilon}^{-1})(1-qz_{i+1}w_{i+1,+}w_{i,\varepsilon}^{-1})(1-qz_{i+1}w_{i+1,-}w_{i,\varepsilon}^{-1})}{1-w_{i,\varepsilon}w_{i,-\varepsilon}^{-1}}.
\end{align*}
\item For $i=1$ we let $E_{i,1}[x^m]$ and $F_{i,1}[x^m]$ be given by the expressions from~(2) with $w_{0,+}$ and~$w_{0,-}$ replaced by~$z_{0,+}$ and~$z_{0,-}$, respectively. For $i=n-1$ these operators are given by the expressions from~(1) if $i$~is even or by the expressions from~(2) if $i$~is odd with $w_{n,+}$ and~$w_{n,-}$ replaced by~$z_{n+1,+}$ and~$z_{n+1,-}$, respectively.
\end{enumerate}
It is known~\cite{FT19} that the embedding~\eqref{eqn:localization} maps the class of $\mathcal{E}_i^m\otimes\mathcal{O}_{\mathcal{R}_{\varpi_{i,1}}}$ to~$E_{i,1}[x^m]$ and maps the class of $\mathcal{F}_i^m\otimes\mathcal{O}_{\mathcal{R}_{\varpi_{i,1}^*}}$ to $F_{i,1}[x^m]$.

\subsection{Generators from monopole operators}
\label{sec:GeneratorsFromMonopoleOperators}

We now write $\mathbf{G}$, $\mathbf{F}$, and $N$ for the gauge and flavor symmetry groups and representation constructed in Section~\ref{sec:QuiverGaugeTheories} using the quiver in Figure~\ref{fig:typeAquiver}. We write $G$ and $F$ for the gauge and flavor symmetry groups from Section~\ref{sec:TheGaugeAndFlavorGroups}, which act on the same vector space~$N$. We write 
\[
\mathbb{A}=K^{(\mathbf{G}\times\mathbf{F})_{\mathcal{O}}\rtimes\mathbb{C}^*}(\mathcal{R}_{\mathbf{G},N}), \quad \mathcal{A}=K^{(G\times F)_{\mathcal{O}}\rtimes\mathbb{C}^*}(\mathcal{R}_{G,N})
\]
for the associated quantized Coulomb branches.

\begin{lemma}
\label{lem:subquotient}
If we regard $\mathbb{A}$ as a subalgebra of $\mathcal{D}_{q,\mathbf{z}}$ using the embedding~\eqref{eqn:localization}, then there is an isomorphism 
\[
\mathcal{A}\cong\mathbb{A}^{(\mathbb{C}^*)^{n-1}}/(w_{i,-}w_{i,+}-1,z_{k,-}z_{k,+}-1)
\]
where the $i$th factor of $(\mathbb{C}^*)^{n-1}$ acts by simultaneously rescaling the variables $D_{i,-}$ and $D_{i,+}$. For~$i=1,\dots,n-1$, this isomorphism maps $X_i+X_i^{-1}$ to $w_{i,+}+w_{i,-}$ and maps $r_{\alpha_i}$ to $E_{i,1}[1]F_{i,1}[1]+C_i$ where $C_i$ is a polynomial in the variables $q^{\pm1}$, $z_{k,\pm}^{\pm1}$, $z_a^{\pm1}$, $w_{j,\pm}$, symmetric under $w_{j,\pm}\mapsto w_{j,\mp}$ for every~$j$.
\end{lemma}

\begin{proof}
The isomorphism is proved in Lemma~3.3 of~\cite{AS24a}, and the proof given there shows that $X_i+X_i^{-1}$ maps to $w_{i,+}+w_{i,-}$ for $i=1,\dots,n-1$. To prove the remaining assertion, let $\mathbf{T}=\prod_i\mathbf{T}_i\subset\mathbf{G}$ where $\mathbf{T}_i\subset\mathrm{GL}_2$ is the diagonal torus in the $i$th $\mathrm{GL}_2$ factor of~$\mathbf{G}$. Let us write $X^*(\mathbf{T}_i)=\mathbb{Z}\xi_{i,1}\oplus\mathbb{Z}\xi_{i,2}$ where $\xi_{i,j}$ corresponds to the character $\left(\begin{smallmatrix}t_1 & 0 \\ 0 & t_2\end{smallmatrix}\right)\mapsto t_j$ of~$\mathbf{T}_i$. Since $\varpi_{i,1}$ and $\varpi_{i,1}^*$ correspond to the cocharacters $t\mapsto\left(\begin{smallmatrix}t & 0 \\ 0 & 1\end{smallmatrix}\right)$ and $t\mapsto\left(\begin{smallmatrix}1 & 0 \\ 0 & t^{-1}\end{smallmatrix}\right)$, respectively, we have 
\begin{gather*}
d(\xi_{i,1}(\varpi_{i,1}),\xi_{i,1}(\varpi_{i,1}^*)) = d(1,0) = 0, \\
d(\xi_{i,2}(\varpi_{i,1}),\xi_{i,2}(\varpi_{i,1}^*)) = d(0,-1) = 0,
\end{gather*}
where $d(a,b)$ is the function defined in Section~\ref{sec:TheAssociatedGradedAlgebra}. Then Proposition~\ref{prop:AssociatedGradedMultiplication} implies 
\[
[\mathcal{O}_{\mathcal{R}_{\varpi_{i,1}+\varpi_{i,1}^*}}]=[\mathcal{O}_{\mathcal{R}_{\varpi_{i,1}}}]*[\mathcal{O}_{\mathcal{R}_{\varpi_{i,1}^*}}]
\]
in the associated graded algebra $\gr^\bullet\mathbb{A}$. It follows that 
\[
[\mathcal{O}_{\mathcal{R}_{\varpi_{i,1}+\varpi_{i,1}^*}}]=[\mathcal{O}_{\mathcal{R}_{\varpi_{i,1}}}]*[\mathcal{O}_{\mathcal{R}_{\varpi_{i,1}^*}}]+k_i[\mathcal{O}_{\mathcal{R}_{\varpi_{i,1}}}]+l_i[\mathcal{O}_{\mathcal{R}_{\varpi_{i,1}^*}}]+c_i
\]
in~$\mathbb{A}$ for some coefficients $k_i$,~$l_i$,~$c_i\in\mathbb{C}[\mathbf{T}\times\mathbf{F}\times\mathbb{C}^*]$ with $c_i$ invariant under the Weyl group action. It follows from the last sentence of Section~\ref{sec:MonopoleOperators} that the embedding~\eqref{eqn:localization} maps the right hand side to an expression of the form 
\[
E_{i,1}[1]F_{i,1}[1]+K_iE_{i,1}[1]+L_iF_{i,1}[1]+C_i
\]
for some coefficients $K_i$,~$L_i$,~$C_i\in\mathbb{C}[q^{\pm\frac{1}{2}},z_{k,\pm}^{\pm1},z_a^{\pm1},w_{j,\pm}^{\pm1}]$ where $C_i$ is invariant under~$w_{j,\pm}\mapsto w_{j,\mp}$ for every~$j$. The proof of Lemma~3.3 in~\cite{AS24a} shows that the isomorphism maps $r_{\alpha_i}\in\mathcal{A}$ for the $\alpha_i$ defined in Section~\ref{sec:TheGaugeAndFlavorGroups} to the class of~$\mathcal{O}_{\mathcal{R}_{\varpi_{i,1}+\varpi_{i,1}^*}}$. Since the image consists of $(\mathbb{C}^*)^{n-1}$-invariant elements, we see from the formulas for $E_{i,1}[1]$ and $F_{i,1}[1]$ that $K_i=L_i=0$. This completes the proof.
\end{proof}

\section{Proof of the main result}
\label{sec:ProofOfTheMainResult}

In this section, we prove our main result, Theorem~\ref{thm:intromain}, relating the relative skein algebra of~$S_{0,n+2}$ to a quantized $K$-theoretic Coulomb branch.

\subsection{Two algebra embeddings}
\label{sec:TwoAlgebraEmbeddings}

Recall that in Section~\ref{sec:ThePolynomialRepresentation} we defined a $\mathbb{C}(q^{\frac{1}{2}},t_0^{\frac{1}{2}},\dots,t_{n+1}^{\frac{1}{2}})$-algebra $\mathscr{X}_{q,\mathbf{t}}$ with generators $X_i^{\pm\frac{1}{2}}$, $\varpi_i^{\pm1}$ for $i=1,\dots,n-1$. On the other hand, in Section~\ref{sec:Localization} we defined a $\mathbb{C}[q^{\pm\frac{1}{2}},z_1^{\pm1},\dots,z_n^{\pm1},z_{0,\pm}^{\pm1},z_{n+1,\pm}^{\pm1}]$-algebra $\mathcal{D}_{q,\mathbf{z}}$ with generators $D_{i,\pm}$, $w_{i,\pm}$ for $i=1,\dots,n-1$. There is an action of the torus $(\mathbb{C}^*)^{n-1}$ on the latter algebra where the $i$th $\mathbb{C}^*$-factor acts by simultaneously rescaling the variables $D_{i,+}$ and~$D_{i,-}$.

\begin{lemma}
\label{lem:relateambientrings}
The assignments 
\[
q^{\frac{1}{2}}\mapsto q^{\frac{1}{2}}, \quad z_i\mapsto t_i, \quad z_{0,\pm}\mapsto t_0^{\pm1}, \quad z_{n+1,\pm}\mapsto t_{n+1}^{\pm1} \quad w_{i,\pm}\mapsto X_i^{\pm1}, \quad D_{i,+}D_{i,-}^{-1}\mapsto q^{-4}X_i^{-4}\varpi_i^2
\]
determine a well defined injective $\mathbb{C}$-algebra homomorphism 
\[
\mathcal{D}_{q,\mathbf{z}}^{(\mathbb{C}^*)^{n-1}}/(z_{i,+}z_{i,-}-1,w_{i,+}w_{i,-}-1)\hookrightarrow\mathscr{X}_{q,\mathbf{t}}.
\]
\end{lemma}

\begin{proof}
It is easy to check that these assignments map the commutation relations in $(\mathcal{D}_{q,\mathbf{z}}^0)^{(\mathbb{C}^*)^{n-1}}$ to relations in~$\mathscr{X}_{q,\mathbf{t}}$ and therefore provide a well defined $\mathbb{C}$-algebra homomorphism $(\mathcal{D}_{q,\mathbf{z}}^0)^{(\mathbb{C}^*)^{n-1}}\rightarrow\mathscr{X}_{q,\mathbf{t}}$. Moreover, the elements $z_{i,+}z_{i,-}-1$ and $w_{i,+}w_{i,-}-1$ map to zero, so this induces a well defined $\mathbb{C}$-algebra homomorphism 
\begin{equation}
\label{eqn:relateambientrings}
(\mathcal{D}_{q,\mathbf{z}}^0)^{(\mathbb{C}^*)^{n-1}}/(z_{i,+}z_{i,-}-1,w_{i,+}w_{i,-}-1)\rightarrow\mathscr{X}_{q,\mathbf{t}}.
\end{equation}
Suppose $\bar{P}$ is an element of the kernel of this map. Then $\bar{P}$ is the class of some $P\in(\mathcal{D}_{q,\mathbf{z}}^0)^{(\mathbb{C}^*)^{n-1}}$, which we can take to be a Laurent polynomial in the variables $q^{\frac{1}{2}}$, $z_i$, $z_{0,+}$, $z_{n+1,+}$, $w_{i,+}$, and~$D_{i,+}D_{i,-}^{-1}$. The images of these elements under the map~\eqref{eqn:relateambientrings} satisfy no relations, so we necessarily have $P=0$ in~$(\mathcal{D}_{q,\mathbf{z}}^0)^{(\mathbb{C}^*)^n}$. Hence the map~\eqref{eqn:relateambientrings} is injective. It induces a $\mathbb{C}$-algebra homomorphism on the localization $\mathcal{D}_{q,\mathbf{z}}^{(\mathbb{C}^*)^{n-1}}/(z_{i,+}z_{i,-}-1,w_{i,+}w_{i,-}-1)$, which is injective since a fraction maps to zero if and only if its numerator maps to zero.
\end{proof}

In Section~\ref{sec:ThePolynomialRepresentation}, we constructed an injective $\mathbb{C}$-algebra homomorphism 
\[
\Phi:\Sk_{A,\boldsymbol{\lambda}}(S_{0,n+2})\hookrightarrow\mathscr{X}_{q,\mathbf{t}},
\]
which we called the polynomial representation. On the other hand, let us write $\mathbb{M}_{q,\mathbf{z}}$ for the $\mathbb{C}[q^{\pm\frac{1}{2}},z_1^{\pm1},\dots,z_n^{\pm1},z_{0,\pm}^{\pm1},z_{n+1,\pm}^{\pm1}]$-subalgebra of $\mathcal{D}_{q,\mathbf{z}}$ generated by all dressed minuscule monopole operators and symmetric polynomials in the $w_{i,\pm}$ for each index~$i$. Write $\mathcal{M}_{q,\mathbf{z}}\coloneqq\mathbb{M}_{q,\mathbf{z}}^{(\mathbb{C}^*)^{n-1}}/(z_{i,+}z_{i,-}-1,w_{i,+}w_{i,-}-1)$. Then we get an injective $\mathbb{C}$-algebra homomorphism 
\[
\Psi:\mathcal{M}_{q,\mathbf{z}}\hookrightarrow\mathscr{X}_{q,\mathbf{t}}
\]
by restricting the map in Lemma~\ref{lem:relateambientrings}.

\subsection{Comparing the generators}

Recall that in Section~\ref{sec:GeneratorsForSkeinAlgebras} we introduced the skein algebra generators $\gamma_i$ and $\sigma_{i,i+1}$ illustrated in Figure~\ref{fig:generators}. These curves lie in an embedded four-holed sphere $\Sigma_i\subset S_{0,n+2}$. For $1<i<n-1$ this subsurface $\Sigma_i$ is cut out by the curves~$\gamma_{i-1}$ and~$\gamma_{i+1}$. For $i=1$ it is cut out by the curves~$\delta_0$ and~$\gamma_2$, and for $i=n-1$ it is cut out by curves~$\gamma_{n-2}$ and~$\delta_{n+1}$. For each index $i=1,\dots,n-1$, we consider a family of additional generators $\theta_{i,m}$ for $m\in\mathbb{Z}$ represented by curves in~$\Sigma_i$. These curves are illustrated in Figure~\ref{fig:familycurves} where $\theta_{i,0}=\sigma_{i,i+1}$. Note that $\theta_{i,m}$ is obtained from $\theta_{i,0}$ by applying the $m$th power of the Dehn twist around~$\gamma_i$.

\begin{figure}[ht]
\begin{tikzpicture}[scale=1.25]
\clip(-1.75,-0.75) rectangle (1.2,1.2);
\draw[black, thin] plot [smooth, tension=1] coordinates { (0,0.25) (0.2,0.3) (0.25,0.5)};
\draw[black, thin] plot [smooth, tension=1] coordinates { (1,0.25) (0.8,0.3) (0.75,0.5)};
\draw[black, thin] plot [smooth, tension=1] coordinates { (0,0.25) (-0.2,0.3) (-0.25,0.5)};
\draw[black, thin] plot [smooth, tension=1] coordinates { (-1,0.25) (-0.8,0.3) (-0.75,0.5)};
\draw (0,-0.25) -- (-1.1,-0.25);
\draw (0,-0.25) -- (1.1,-0.25);
\draw (-1,0.25) -- (-1.1,0.25);
\draw (1,0.25) -- (1.1,0.25);
\draw[black, thin] plot [smooth, tension=1] coordinates { (-0.25,0.5) (-0.375,0.45) (-0.625,0.45) (-0.75,0.5)};
\draw[black, thin] plot [smooth, tension=1] coordinates { (-0.25,0.5) (-0.375,0.55) (-0.625,0.55) (-0.75,0.5)};
\draw[black, thin] plot [smooth, tension=1] coordinates { (0.25,0.5) (0.375,0.45) (0.625,0.45) (0.75,0.5)};
\draw[black, thin] plot [smooth, tension=1] coordinates { (0.25,0.5) (0.375,0.55) (0.625,0.55) (0.75,0.5)};
\draw[black, thin, dotted] plot [smooth, tension=1] coordinates { (-1,-0.25) (-1.05,-0.125) (-1.05,0.125) (-1,0.25)};
\draw[black, thin] plot [smooth, tension=1] coordinates { (-1,-0.25) (-0.95,-0.125) (-0.95,0.125) (-1,0.25)};
\draw[black, thick] plot [smooth, tension=1] coordinates { (-0.9,0.25) (-0.5,0.15) (-0.1,0.25)};
\draw[black, thick, dotted] plot [smooth, tension=1] coordinates { (0.9,0.25) (0.5,0.15) (0.1,0.25)};
\draw[black, thick, dotted] plot [smooth, tension=1] coordinates { (-0.25,-0.25) (-0.65,-0.05) (-0.9,0.25)};
\draw[black, thick] plot [smooth, tension=1] coordinates { (-0.25,-0.25) (0,-0.05) (0.1,0.25)};
\draw[black, thick] plot [smooth, tension=1] coordinates { (0.25,-0.25) (0.65,-0.05) (0.9,0.25)};
\draw[black, thick, dotted] plot [smooth, tension=1] coordinates { (0.25,-0.25) (0,-0.05) (-0.1,0.25)};
\draw[black, thin, dotted] plot [smooth, tension=1] coordinates { (1,-0.25) (0.95,-0.125) (0.95,0.125) (1,0.25)};
\draw[black, thin] plot [smooth, tension=1] coordinates { (1,-0.25) (1.05,-0.125) (1.05,0.125) (1,0.25)};
\node[below] at (0,-0.25) {\tiny $\theta_{i,-2}$};
\node at (-1.5,0) {\tiny $\dots$};
\end{tikzpicture}
\begin{tikzpicture}[scale=1.25]
\clip(-1.2,-0.75) rectangle (1.2,1.2);
\draw[black, thin] plot [smooth, tension=1] coordinates { (0,0.25) (0.2,0.3) (0.25,0.5)};
\draw[black, thin] plot [smooth, tension=1] coordinates { (1,0.25) (0.8,0.3) (0.75,0.5)};
\draw[black, thin] plot [smooth, tension=1] coordinates { (0,0.25) (-0.2,0.3) (-0.25,0.5)};
\draw[black, thin] plot [smooth, tension=1] coordinates { (-1,0.25) (-0.8,0.3) (-0.75,0.5)};
\draw (0,-0.25) -- (-1.1,-0.25);
\draw (0,-0.25) -- (1.1,-0.25);
\draw (-1,0.25) -- (-1.1,0.25);
\draw (1,0.25) -- (1.1,0.25);
\draw[black, thin] plot [smooth, tension=1] coordinates { (-0.25,0.5) (-0.375,0.45) (-0.625,0.45) (-0.75,0.5)};
\draw[black, thin] plot [smooth, tension=1] coordinates { (-0.25,0.5) (-0.375,0.55) (-0.625,0.55) (-0.75,0.5)};
\draw[black, thin] plot [smooth, tension=1] coordinates { (0.25,0.5) (0.375,0.45) (0.625,0.45) (0.75,0.5)};
\draw[black, thin] plot [smooth, tension=1] coordinates { (0.25,0.5) (0.375,0.55) (0.625,0.55) (0.75,0.5)};
\draw[black, thin, dotted] plot [smooth, tension=1] coordinates { (-1,-0.25) (-1.05,-0.125) (-1.05,0.125) (-1,0.25)};
\draw[black, thin] plot [smooth, tension=1] coordinates { (-1,-0.25) (-0.95,-0.125) (-0.95,0.125) (-1,0.25)};
\draw[black, thick] plot [smooth, tension=1] coordinates { (0,-0.25) (0.65,-0.1) (0.9,0.25)};
\draw[black, thick, dotted] plot [smooth, tension=1] coordinates { (0,-0.25) (-0.65,-0.1) (-0.9,0.25)};
\draw[black, thick] plot [smooth, tension=1] coordinates { (-0.9,0.25) (-0.45,0.15) (0,0.25)};
\draw[black, thick, dotted] plot [smooth, tension=1] coordinates { (0.9,0.25) (0.45,0.15) (0,0.25)};
\draw[black, thin, dotted] plot [smooth, tension=1] coordinates { (1,-0.25) (0.95,-0.125) (0.95,0.125) (1,0.25)};
\draw[black, thin] plot [smooth, tension=1] coordinates { (1,-0.25) (1.05,-0.125) (1.05,0.125) (1,0.25)};
\node[below] at (0,-0.25) {\tiny $\theta_{i,-1}$};
\end{tikzpicture}
\begin{tikzpicture}[scale=1.25]
\clip(-1.2,-0.75) rectangle (1.2,1.2);
\draw[black, thin] plot [smooth, tension=1] coordinates { (0,0.25) (0.2,0.3) (0.25,0.5)};
\draw[black, thin] plot [smooth, tension=1] coordinates { (1,0.25) (0.8,0.3) (0.75,0.5)};
\draw[black, thin] plot [smooth, tension=1] coordinates { (0,0.25) (-0.2,0.3) (-0.25,0.5)};
\draw[black, thin] plot [smooth, tension=1] coordinates { (-1,0.25) (-0.8,0.3) (-0.75,0.5)};
\draw (0,-0.25) -- (-1.1,-0.25);
\draw (0,-0.25) -- (1.1,-0.25);
\draw (-1,0.25) -- (-1.1,0.25);
\draw (1,0.25) -- (1.1,0.25);
\draw[black, thin] plot [smooth, tension=1] coordinates { (-0.25,0.5) (-0.375,0.45) (-0.625,0.45) (-0.75,0.5)};
\draw[black, thin] plot [smooth, tension=1] coordinates { (-0.25,0.5) (-0.375,0.55) (-0.625,0.55) (-0.75,0.5)};
\draw[black, thin] plot [smooth, tension=1] coordinates { (0.25,0.5) (0.375,0.45) (0.625,0.45) (0.75,0.5)};
\draw[black, thin] plot [smooth, tension=1] coordinates { (0.25,0.5) (0.375,0.55) (0.625,0.55) (0.75,0.5)};
\draw[black, thin, dotted] plot [smooth, tension=1] coordinates { (-1,-0.25) (-1.05,-0.125) (-1.05,0.125) (-1,0.25)};
\draw[black, thin] plot [smooth, tension=1] coordinates { (-1,-0.25) (-0.95,-0.125) (-0.95,0.125) (-1,0.25)};
\draw[black, thick] plot [smooth, tension=1] coordinates { (0.9,0.25) (0,-0.1) (-0.9,0.25)};
\draw[black, thick, dotted] plot [smooth, tension=1] coordinates { (-0.9,0.25) (0,0.05) (0.9,0.25)};
\draw[black, thin, dotted] plot [smooth, tension=1] coordinates { (1,-0.25) (0.95,-0.125) (0.95,0.125) (1,0.25)};
\draw[black, thin] plot [smooth, tension=1] coordinates { (1,-0.25) (1.05,-0.125) (1.05,0.125) (1,0.25)};
\node[below] at (0,-0.25) {\tiny $\theta_{i,0}$};
\end{tikzpicture}
\begin{tikzpicture}[scale=1.25]
\clip(-1.2,-0.75) rectangle (1.2,1.2);
\draw[black, thin] plot [smooth, tension=1] coordinates { (0,0.25) (0.2,0.3) (0.25,0.5)};
\draw[black, thin] plot [smooth, tension=1] coordinates { (1,0.25) (0.8,0.3) (0.75,0.5)};
\draw[black, thin] plot [smooth, tension=1] coordinates { (0,0.25) (-0.2,0.3) (-0.25,0.5)};
\draw[black, thin] plot [smooth, tension=1] coordinates { (-1,0.25) (-0.8,0.3) (-0.75,0.5)};
\draw (0,-0.25) -- (-1.1,-0.25);
\draw (0,-0.25) -- (1.1,-0.25);
\draw (-1,0.25) -- (-1.1,0.25);
\draw (1,0.25) -- (1.1,0.25);
\draw[black, thin] plot [smooth, tension=1] coordinates { (-0.25,0.5) (-0.375,0.45) (-0.625,0.45) (-0.75,0.5)};
\draw[black, thin] plot [smooth, tension=1] coordinates { (-0.25,0.5) (-0.375,0.55) (-0.625,0.55) (-0.75,0.5)};
\draw[black, thin] plot [smooth, tension=1] coordinates { (0.25,0.5) (0.375,0.45) (0.625,0.45) (0.75,0.5)};
\draw[black, thin] plot [smooth, tension=1] coordinates { (0.25,0.5) (0.375,0.55) (0.625,0.55) (0.75,0.5)};
\draw[black, thin, dotted] plot [smooth, tension=1] coordinates { (-1,-0.25) (-1.05,-0.125) (-1.05,0.125) (-1,0.25)};
\draw[black, thin] plot [smooth, tension=1] coordinates { (-1,-0.25) (-0.95,-0.125) (-0.95,0.125) (-1,0.25)};
\draw[black, thick, dotted] plot [smooth, tension=1] coordinates { (0,-0.25) (0.65,-0.1) (0.9,0.25)};
\draw[black, thick] plot [smooth, tension=1] coordinates { (0,-0.25) (-0.65,-0.1) (-0.9,0.25)};
\draw[black, thick, dotted] plot [smooth, tension=1] coordinates { (-0.9,0.25) (-0.45,0.15) (0,0.25)};
\draw[black, thick] plot [smooth, tension=1] coordinates { (0.9,0.25) (0.45,0.15) (0,0.25)};
\draw[black, thin, dotted] plot [smooth, tension=1] coordinates { (1,-0.25) (0.95,-0.125) (0.95,0.125) (1,0.25)};
\draw[black, thin] plot [smooth, tension=1] coordinates { (1,-0.25) (1.05,-0.125) (1.05,0.125) (1,0.25)};
\node[below] at (0,-0.25) {\tiny $\theta_{i,1}$};
\end{tikzpicture}
\begin{tikzpicture}[scale=1.25]
\clip(-1.2,-0.75) rectangle (1.75,1.2);
\draw[black, thin] plot [smooth, tension=1] coordinates { (0,0.25) (0.2,0.3) (0.25,0.5)};
\draw[black, thin] plot [smooth, tension=1] coordinates { (1,0.25) (0.8,0.3) (0.75,0.5)};
\draw[black, thin] plot [smooth, tension=1] coordinates { (0,0.25) (-0.2,0.3) (-0.25,0.5)};
\draw[black, thin] plot [smooth, tension=1] coordinates { (-1,0.25) (-0.8,0.3) (-0.75,0.5)};
\draw (0,-0.25) -- (-1.1,-0.25);
\draw (0,-0.25) -- (1.1,-0.25);
\draw (-1,0.25) -- (-1.1,0.25);
\draw (1,0.25) -- (1.1,0.25);
\draw[black, thin] plot [smooth, tension=1] coordinates { (-0.25,0.5) (-0.375,0.45) (-0.625,0.45) (-0.75,0.5)};
\draw[black, thin] plot [smooth, tension=1] coordinates { (-0.25,0.5) (-0.375,0.55) (-0.625,0.55) (-0.75,0.5)};
\draw[black, thin] plot [smooth, tension=1] coordinates { (0.25,0.5) (0.375,0.45) (0.625,0.45) (0.75,0.5)};
\draw[black, thin] plot [smooth, tension=1] coordinates { (0.25,0.5) (0.375,0.55) (0.625,0.55) (0.75,0.5)};
\draw[black, thin, dotted] plot [smooth, tension=1] coordinates { (-1,-0.25) (-1.05,-0.125) (-1.05,0.125) (-1,0.25)};
\draw[black, thin] plot [smooth, tension=1] coordinates { (-1,-0.25) (-0.95,-0.125) (-0.95,0.125) (-1,0.25)};
\draw[black, thick, dotted] plot [smooth, tension=1] coordinates { (-0.9,0.25) (-0.5,0.15) (-0.1,0.25)};
\draw[black, thick] plot [smooth, tension=1] coordinates { (0.9,0.25) (0.5,0.15) (0.1,0.25)};
\draw[black, thick] plot [smooth, tension=1] coordinates { (-0.25,-0.25) (-0.65,-0.05) (-0.9,0.25)};
\draw[black, thick, dotted] plot [smooth, tension=1] coordinates { (-0.25,-0.25) (0,-0.05) (0.1,0.25)};
\draw[black, thick, dotted] plot [smooth, tension=1] coordinates { (0.25,-0.25) (0.65,-0.05) (0.9,0.25)};
\draw[black, thick] plot [smooth, tension=1] coordinates { (0.25,-0.25) (0,-0.05) (-0.1,0.25)};
\draw[black, thin, dotted] plot [smooth, tension=1] coordinates { (1,-0.25) (0.95,-0.125) (0.95,0.125) (1,0.25)};
\draw[black, thin] plot [smooth, tension=1] coordinates { (1,-0.25) (1.05,-0.125) (1.05,0.125) (1,0.25)};
\node[below] at (0,-0.25) {\tiny $\theta_{i,2}$};
\node at (1.5,0) {\tiny $\dots$};
\end{tikzpicture}
\caption{A family of curves on~$S_{0,n+2}$.\label{fig:familycurves}}
\end{figure}
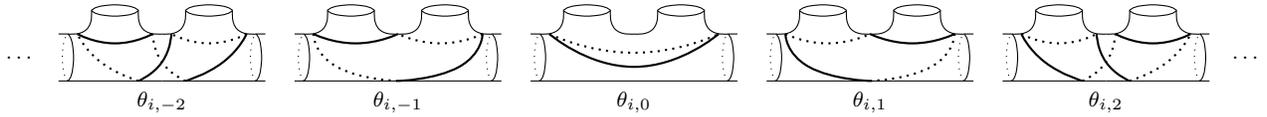

\begin{lemma}
\label{lem:imagefamily}
The polynomial representation $\Phi$ maps 
\[
\theta_{i,m}\mapsto q^mX_i^mT_i(X_i)(\tau_i-1)+q^mX_i^{-m}T_i(X_i^{-1})(\tau_i^{-1}-1)+f_m
\]
where $T_i$ and $\tau_i$ are the operators defined in Proposition~\ref{prop:skeinoperatorvalues} and $f_m\in\mathbb{C}[q^{\pm\frac{1}{2}},t_k^{\pm1},X_j^{\pm1}]$ is invariant under $X_j\mapsto X_j^{-1}$.
\end{lemma}

\begin{proof}
It follows from the skein relation that we have an identity of the form 
\[
A^2\gamma_i\theta_{i,0}-A^{-2}\theta_{i,0}\gamma_i=(A^4-A^{-4})\theta_{i,1}-(A^2-A^{-2})L_i
\]
in $\Sk_A(S_{0,n+2})$. For $1<i<n-1$, the factor $L_i$ is given by $L_i=\gamma_{i-1}\delta_{i+1}+\gamma_{i+1}\delta_i$. For $i=1$ it is given by the same expression with $\gamma_{i-1}$ replaced by $\delta_0$, and for $i=n-1$ it is given by the same expression with $\gamma_{i+1}$ replaced by $\delta_{n+1}$. The polynomial representation extends to an injective $\mathbb{C}$-algebra homomorphism $\Phi:\mathscr{S}_{A,\boldsymbol{\lambda}}(S_{0,n+2})\rightarrow\mathscr{X}_{q,\mathbf{t}}$. Let us write $\widehat{g}\coloneqq\Phi(g)$ for any $g\in\mathscr{S}_{A,\boldsymbol{\lambda}}(S_{0,n+2})$. Then the above identity implies 
\begin{equation}
\label{eqn:m=1operator}
\widehat{\theta}_{i,1}=\frac{1}{q^{-2}-q^2}\left(q^{-1}\widehat{\gamma}_i\widehat{\theta}_{i,0}-q\widehat{\theta}_{i,0}\widehat{\gamma}_i+(q^{-1}-q)\widehat{L}_i\right).
\end{equation}
By applying Proposition~\ref{prop:skeinoperatorvalues} to the right hand side of this expression, we see that $\widehat{\theta}_{i,1}$ lies in the $\mathbb{C}(q^{\frac{1}{2}},t_0^{\frac{1}{2}},\dots,t_{n+1}^{\frac{1}{2}},X_1,\dots,X_{n-1})$-vector space spanned by 1,~$\tau_i$ and~$\tau_i^{-1}$. Thus we can write 
\[
\widehat{\theta}_{i,1}=C_1(\tau_i-1)+C_2(\tau_i^{-1}-1)+C_3
\]
for unique coefficients $C_1$,~$C_2$,~$C_3\in\mathbb{C}(q^{\frac{1}{2}},t_0^{\frac{1}{2}},\dots,t_{n+1}^{\frac{1}{2}},X_1,\dots,X_{n-1})$. From the expressions for $\widehat{\gamma}_i$ and $\widehat{\theta}_{i,0}=\widehat{\sigma}_{i,i+1}$ given in Proposition~\ref{prop:skeinoperatorvalues}, one sees that $C_3=\widehat{\theta}_{i,1}|_{\tau_i=1}$ is a Laurent polynomial in the~$X_j$, invariant under $X_j\mapsto X_j^{-1}$. Using Proposition~\ref{prop:skeinoperatorvalues}, one can compute the coefficients of~$\tau_i$ and~$\tau_i^{-1}$ on the right hand side of~\eqref{eqn:m=1operator}, and one finds that $C_1=qX_iT_i(X_i)$ and $C_2=qX_i^{-1}T_i(X_i^{-1})$. Thus the statement in the lemma is true for $m=1$. It is true for $m=0$ by Proposition~\ref{prop:skeinoperatorvalues}.

Fix an integer $M\geq1$ and assume the statement is true for $0\leq m\leq M$. Using the skein relations, one can show that $\gamma_i\theta_{i,M}=A^2\theta_{i,M+1}+A^{-2}\theta_{i,M-1}+B_{i,M}$ where $B_{i,M}$ is a $\mathbb{C}[\lambda_0^{\pm1},\dots,\lambda_{n+1}^{\pm1}]$-linear combination of the~$\gamma_j$. It follows from our assumption that $\Phi$ maps 
\begin{align*}
\theta_{i,M+1} &= A^{-2}\gamma_i\theta_{i,M}-A^{-4}\theta_{i,M-1}-A^{-2}B_{i,M} \\
&\mapsto q^{M+1}X_i^{M+1}T_i(X_i)(\tau_i-1)+q^{M+1}X_i^{-M-1}T_i(X_i^{-1})(\tau_i^{-1}-1)+f_{M+1}
\end{align*}
where $f_{M+1}=q(X_i+X_i^{-1})f_M-q^2f_{M-1}-q\widehat{B}_{i,M}$. By induction, the statement in the lemma is true for all integers $m\geq0$. A similar inductive argument shows that the statement is true for all integers $m\leq0$.
\end{proof}

We now compare the images of the skein algebra generators under the embedding~$\Phi$ with the images of elements of the Coulomb branch under the embedding~$\Psi$.

\begin{lemma}
\label{lem:compareembeddings}
The maps $\Phi$ and $\Psi$ satisfy
\begin{gather*}
\Phi(A)=\Psi(q^{-\frac{1}{2}}), \quad \Phi(\lambda_0)=\Psi(z_{0,+}), \quad \Phi(\lambda_{n+1})=\Psi(z_{n+1,+}), \\
\Phi(\lambda_i)=\Psi(z_i) \quad\text{for $i=1,\dots,n$}.
\end{gather*}
If $\gamma_i$ is defined as in Figure~\ref{fig:generators} and $\theta_{i,m}$ is defined as in Figure~\ref{fig:familycurves}, then we have 
\begin{gather*}
\Phi(\gamma_i)=\Psi(-w_{i,+}-w_{i,-}) \quad\text{for $i=1,\dots,n-1$}, \\
\Phi(\theta_{i,m})=\Psi(q^mB_iE_{i,1}[x^{m-2}]F_{i,1}[1]+C_{i,m}) \quad\text{for $1\leq i\leq n-1$ even}, \\
\Phi(\theta_{i,m})=\Psi(q^{4-m}B_iF_{i,1}[x^{2-m}]E_{i,1}[1]+C_{i,m}) \quad\text{for $1\leq i\leq n-1$ odd},
\end{gather*}
where $B_i=-q^{-1}z_i^{-1}z_{i+1}^{-1}$ and $C_{i,m}\in\mathbb{C}[q^{\pm\frac{1}{2}},z_k^{\pm1},z_{0,\pm}^{\pm1},z_{n+1,\pm}^{\pm1},w
_{j,\pm}^{\pm1}]$ is invariant under $w_{j,\pm}\mapsto w_{j,\mp}$.
\end{lemma}

\begin{proof}
Suppose $i$ is an even integer satisfying $1<i<n-1$. Then by the commutation relations in~$\mathcal{D}_{q,\mathbf{z}}$ and the formulas given in Section~\ref{sec:MonopoleOperators} for the dressed minuscule monopole operators, we have 
\[
E_{i,1}[x^{m-2}]F_{i,1}[1]=
\sum_{\varepsilon_1,\varepsilon_2\in\{\pm1\}}w_{i,\varepsilon_1}^{m-2}\mathcal{L}_{i,\varepsilon_1,\varepsilon_2}D_{i,\varepsilon_1}D_{i,\varepsilon_2}^{-1}
\]
where we use the notation 
\[
\mathcal{L}_{i,\varepsilon_1,\varepsilon_2}=\frac{(1-qz_iw_{i,\varepsilon_1}w_{i-1,+}^{-1})(1-qz_iw_{i,\varepsilon_1}w_{i-1,-}^{-1})(1-qz_{i+1}w_{i,\varepsilon_1}w_{i+1,+}^{-1})(1-qz_{i+1}w_{i,\varepsilon_1}w_{i+1,-}^{-1})}{(1-w_{i,-\varepsilon_1}w_{i,\varepsilon_1}^{-1})(1-q^{2\varepsilon_1\varepsilon_2}w_{i,\varepsilon_2}w_{i,-\varepsilon_2}^{-1})}.
\]
If $\varepsilon_1=\varepsilon_2$ then a straightforward calculation shows that 
\[
\Psi(q^mw_{i,\varepsilon_1}^{m-2}B_i\mathcal{L}_{i,\varepsilon_1,\varepsilon_2}D_{i,\varepsilon_1}D_{i,\varepsilon_2}^{-1}) = -q^mX_i^{\varepsilon_1m}T_i(X_i^{\varepsilon_1})
\]
where $T_i$ is defined as in Proposition~\ref{prop:skeinoperatorvalues}. On the other hand, if $\varepsilon_1\neq\varepsilon_2$ then a similar calculation shows that we have 
\[
\Psi(q^mw_{i,\varepsilon_1}^{m-2}B_i\mathcal{L}_{i,\varepsilon_1,\varepsilon_2}D_{i,\varepsilon_1}D_{i,\varepsilon_2}^{-1}) = q^mX_i^{\varepsilon_1m}T_i(X_i^{\varepsilon_1})\tau_i^{\varepsilon_1}
\]
where $\tau_i=q^{-2}X_i^{-2}\varpi_i^2$. Hence 
\[
\Psi(q^mB_iE_{i,1}[x^{m-2}]F_{i,1}[1])=\sum_{\varepsilon\in\{\pm1\}}q^mX_i^{\varepsilon m}T_i(X_i^{\varepsilon})(\tau_i^{\varepsilon}-1).
\]
The same equation also holds for $i=1$ and $i=n-1$ even after an appropriate specialization of variables. If we now define $C_{i,m}$ to be the unique expression such that $\Psi(C_{i,m})$ equals the Laurent~polynomial~$f_m$ defined above, then we obtain the second-to-last equality by Lemma~\ref{lem:imagefamily}. The last equality is proved similarly, and the remaining ones follow immediately from Proposition~\ref{prop:skeinoperatorvalues} and the definition of the map~$\Psi$ given in Lemma~\ref{lem:relateambientrings} .
\end{proof}

\subsection{An isomorphism of the localized algebras}

Let $G$, $F$, and~$N$ be the groups and representation defined in Section~\ref{sec:TheGaugeAndFlavorGroups}, and write $\mathcal{A}(G,F,N)=K^{(G\times F)_{\mathcal{O}}\rtimes\mathbb{C}^*}(\mathcal{R}_{G,N})$ for the quantized Coulomb branch from Section~\ref{sec:GeneratorsFromMonopoleOperators}. As noted in Section~\ref{sec:GeneratingTheCoulombBranch}, it is an algebra over $K^{F\times\mathbb{C}^*}(\mathrm{pt})\cong\mathbb{C}[q^{\pm\frac{1}{2}},t_0^{\pm1},\dots,t_{n+1}^{\pm1}]$, and we write 
\[
\mathscr{A}(G,F,N)=\mathcal{A}(G,F,N)\otimes_{\mathbb{C}[q^{\pm\frac{1}{2}},t_0^{\pm1},\dots,t_{n+1}^{\pm1}]}\mathbb{C}(q^{\frac{1}{2}},t_0,\dots,t_{n+1})
\]
for the algebra obtained by extending scalars to rational functions in $q^{\frac{1}{2}},t_0,\dots,t_{n+1}$.

\begin{proposition}
\label{prop:mainlocalized}
There is a $\mathbb{C}$-algebra isomorphism 
\[
\mathscr{S}_{A,\boldsymbol{\lambda}}(S_{0,n+2})\cong\mathscr{A}(G,F,N).
\]
\end{proposition}

\begin{proof}
The embeddings $\Phi$ and $\Psi$ of Section~\ref{sec:TwoAlgebraEmbeddings} extend uniquely to algebra embeddings $\Phi:\mathscr{S}_{A,\boldsymbol{\lambda}}(S_{0,n+2})\hookrightarrow\mathscr{X}_{q,\mathbf{t}}$ and $\Psi:\mathscr{M}_{q,\mathbf{z}}\hookrightarrow\mathscr{X}_{q,\mathbf{t}}$ where $\mathscr{M}_{q,\mathbf{z}}$ is the algebra obtained from $\mathcal{M}_{q,\mathbf{z}}\coloneqq\mathbb{M}_{q,\mathbf{z}}^{(\mathbb{C}^*)^{n-1}}/(z_{i,+}z_{i,-}-1,w_{i,+}w_{i,-}-1)$ by extending scalars to rational functions in the variables $q^{\frac{1}{2}},z_{0,+},z_{n+1,+},z_1,\dots,z_n$, and $\mathbb{M}_{q,\mathbf{z}}$ is defined in Section~\ref{sec:TwoAlgebraEmbeddings}. By Proposition~\ref{prop:generators}, we know that the elements $\gamma_i$ and $\sigma_{i,i+1}=\theta_{i,0}$ generate the relative skein algebra $\mathscr{S}_{A,\boldsymbol{\lambda}}(S_{0,n+2})$. By Lemma~\ref{lem:compareembeddings}, these correspond to elements of $\mathcal{M}_{q,\mathbf{z}}$, and so we get an embedding $\mathscr{S}_{A,\boldsymbol{\lambda}}(S_{0,n+2})\hookrightarrow\mathscr{M}_{q,\mathbf{z}}$. By the discussion in Section~\ref{sec:MonopoleOperators}, we know $\mathbb{M}_{q,\mathbf{z}}$ is contained in the image of the algebra~$\mathbb{A}=K^{(\mathbf{G}\times\mathbf{F})_{\mathcal{O}}\rtimes\mathbb{C}^*}(\mathcal{R}_{\mathbf{G},N})$ under the map~\eqref{eqn:localization}. We have $\mathcal{M}_{q,\mathbf{z}}\subset\mathcal{A}(G,F,N)$ by Lemma~\ref{lem:subquotient}, so $\mathscr{M}_{q,\mathbf{z}}\subset\mathscr{A}(G,F,N)$, and therefore we get an embedding $\mathscr{S}_{A,\boldsymbol{\lambda}}(S_{0,n+2})\hookrightarrow\mathscr{A}(G,F,N)$.

We claim that this last embedding is surjective. To see this, it suffices, by Proposition~\ref{prop:generateCoulomb} and Lemma~\ref{lem:subquotient}, to show that $w_{i,+}+w_{i,-}$ and $E_{i,1}[1]F_{i,1}[1]$ are contained in the image. The former is contained in the image by Lemma~\ref{lem:compareembeddings}. Note that the term $C_{i,m}$ in Lemma~\ref{lem:compareembeddings} also lies in the image of the embedding. If $i$ is even, then by taking $m=2$ in the second-to-last equality in Lemma~\ref{lem:compareembeddings}, we see that $E_{i,1}[1]F_{i,1}[1]$ is in the image of the embedding. If $i$ is odd, then by taking $m=2$ in the last equality in Lemma~\ref{lem:compareembeddings}, we see that $F_{i,1}[1]E_{i,1}[1]$ is in the image. As noted in equation~(12) of~\cite{AS24a}, the monopole operators obey a commutation relation of the form $[E_{i,1}[1],F_{i,1}[1]]=(q-q^{-1})h$ where $h$ is some symmetric polynomial in the~$w_{j,\pm}$. Thus $E_{i,1}[1]F_{i,1}[1]$ is in the image of our embedding when $i$ is odd. This completes the proof.
\end{proof}

\subsection{Integral forms}

Finally, we will prove an isomorphism of the integral forms of the algebras appearing in Proposition~\ref{prop:mainlocalized}.

\subsubsection{}

Let us write $\mathscr{S}=\mathscr{S}_{A,\boldsymbol{\lambda}}(S_{0,n+2})$ and $\mathscr{A}=\mathscr{A}(G,F,N)$ for the algebras in Proposition~\ref{prop:mainlocalized} and write $\mathcal{S}=\Sk_{A,\boldsymbol{\lambda}}(S_{0,n+2})$ and $\mathcal{A}=\mathcal{A}(G,F,N)$ for their integral forms. We will write~$\mathscr{I}:\mathscr{S}\rightarrow\mathscr{A}$ for the isomorphism of Proposition~\ref{prop:mainlocalized}. In Section~\ref{sec:TheAssociatedGradedAlgebraSkein} we defined the degree of a multicurve, which makes~$\mathcal{S}$ and~$\mathscr{S}$ into filtered algebras with $\gr^\bullet\mathscr{S}=\gr^\bullet\mathcal{S}\otimes\mathbb{C}(A,\lambda_0,\dots,\lambda_{n+1})$. Similarly, $\mathcal{A}$~and~$\mathscr{A}$ are filtered algebras with $\gr^\bullet\mathscr{A}=\gr^\bullet\mathcal{A}\otimes\mathbb{C}(q^{\frac{1}{2}},t_0,\dots,t_{n+1})$. The submodules in the filtration of~$\mathcal{S}$ are indexed by $2\mathbb{Z}_{\geq0}^{n-1}$, which is naturally identified with the set $\Lambda^+$ of dominant coweights of~$G$ by $2\mathbf{e}_i\mapsto\alpha_i$ where $\mathbf{e}_i\in\mathbb{Z}_{\geq0}^{n-1}$ is the $i$th standard basis vector. If $\lambda\in\Lambda^+$ is a dominant coweight and $\mathbb{U}$ is any of the algebras $\mathcal{S}$, $\mathscr{S}$, $\mathcal{A}$, or~$\mathscr{A}$, then we will write~$\mathcal{F}^\lambda\mathbb{U}$ for the submodule in the filtration of~$\mathbb{U}$ corresponding to~$\lambda$. We will write $\mathcal{F}^{<\lambda}\mathbb{U}$ for the submodule of~$\mathbb{U}$ spanned by $\mathcal{F}^\mu\mathbb{U}$ for~$\mu<\lambda$.

\begin{lemma}
$\mathscr{I}:\mathscr{S}\stackrel{\sim}{\rightarrow}\mathscr{A}$ induces an isomorphism of $\Lambda^+$-graded algebras $\bar{\mathscr{I}}:\gr^\bullet\mathscr{S}\stackrel{\sim}{\rightarrow}\gr^\bullet\mathscr{A}$.
\end{lemma}

\begin{proof}
By applying the formulas of Lemma~\ref{lem:compareembeddings} to generators, one sees that $\mathscr{I}(\mathcal{F}^\lambda\mathscr{S})\subset\mathcal{F}^\lambda\mathscr{A}$ and $\mathscr{I}^{-1}(\mathcal{F}^\lambda\mathscr{A})\subset\mathcal{F}^\lambda\mathscr{S}$ for every $\lambda\in\Lambda^+$. Thus $\mathscr{I}$ restricts to an isomorphism $\mathcal{F}^\lambda\mathscr{S}\stackrel{\sim}{\rightarrow}\mathcal{F}^\lambda\mathscr{A}$ for every $\lambda\in\Lambda^+$. This implies the lemma.
\end{proof}

In the following, if $x\in\mathcal{S}$ and there exists a unique $\lambda\in\Lambda^+$ such that $x\in\mathcal{F}^\lambda\mathcal{S}\setminus\mathcal{F}^{<\lambda}\mathcal{S}$, then we will write $\bar{x}$ for the image of~$x$ in~$\gr^\lambda\mathcal{S}$. Similarly, if $x\in\mathcal{A}$ and~$x\in\mathcal{F}^\lambda\mathcal{A}\setminus\mathcal{F}^{<\lambda}\mathcal{A}$ for unique $\lambda\in\Lambda^+$, then we will write $\bar{x}\in\gr^\lambda\mathcal{A}$ for the image of~$x$ in~$\gr^\lambda\mathcal{A}$.

\begin{lemma}
\label{lem:conditionforiso}
$\mathscr{I}$ restricts to an isomorphism $\mathcal{S}\stackrel{\sim}{\rightarrow}\mathcal{A}$ if and only if $\bar{\mathscr{I}}$ restricts to a surjective algebra homomorphism $\gr^\bullet\mathcal{S}\rightarrow\gr^\bullet\mathcal{A}$.
\end{lemma}

\begin{proof}
Suppose that $\bar{\mathscr{I}}$ restricts to a surjective algebra homomorphism $\gr^\bullet\mathcal{S}\rightarrow\gr^\bullet\mathcal{A}$. In this case, we claim that $\mathscr{I}$ restricts to a surjective map $\mathcal{F}^\lambda\mathcal{S}\rightarrow\mathcal{F}^\lambda\mathcal{A}$ for all $\lambda\in\Lambda^+$. This is true for $\lambda=0$ because $\gr^0\mathcal{S}=\mathcal{F}^0\mathcal{S}$ and $\gr^0\mathcal{A}=\mathcal{F}^0\mathcal{A}$ are the ground rings. Assume inductively that the claim holds for all $\lambda<\mu$. For any multicurve $x\in\mathcal{F}^\mu\mathcal{S}$, we get an element $\bar{\mathscr{I}}(\bar{x})\in\gr^\bullet\mathcal{A}$, and there exists $y\in\mathcal{F}^\mu\mathcal{A}$ such that $\bar{\mathscr{I}}(\bar{x})=\bar{y}$. This means that $\mathscr{I}(x)-y\in\mathcal{F}^{<\mu}\mathcal{A}$ and therefore $\mathscr{I}(x)\in\mathcal{F}^\mu\mathcal{A}$. Thus we see that $\mathscr{I}$ restricts to a map $\mathcal{F}^\mu\mathcal{S}\rightarrow\mathcal{F}^\mu\mathcal{A}$. This map fits into a commutative diagram 
\[
\xymatrix{
0 \ar[r] & \mathcal{F}^{<\mu}\mathcal{S} \ar[r] \ar[d] & \mathcal{F}^\mu\mathcal{S} \ar[d] \ar[r] & \gr^\mu\mathcal{S} \ar[d] \ar[r] & 0 \\
0 \ar[r] & \mathcal{F}^{<\mu}\mathcal{A} \ar[r] & \mathcal{F}^\mu\mathcal{A} \ar[r] &  \gr^\mu\mathcal{A} \ar[r] & 0
}
\]
with exact rows. The right vertical map is the restriction of~$\bar{\mathscr{I}}$, which is surjective by assumption. The left vertical map is the restriction of~$\mathscr{I}$, which is surjective by the inductive hypothesis. Hence, by the five~lemma, $\mathscr{I}$ restricts to a surjective map $\mathcal{F}^\mu\mathcal{S}\rightarrow\mathcal{F}^\mu\mathcal{A}$. By induction, $\mathscr{I}$ restricts to a surjection $\mathcal{F}^\lambda\mathcal{S}\rightarrow\mathcal{F}^\lambda\mathcal{A}$ for each~$\lambda\in\Lambda^+$ as claimed. In fact, this restriction must be an isomorphism because $\mathscr{I}$ is injective. Since $\mathcal{S}=\bigcup_\lambda\mathcal{F}^\lambda\mathcal{S}$ and $\mathcal{A}=\bigcup_\lambda\mathcal{F}^\lambda\mathcal{A}$, it follows that $\mathscr{I}$ restricts to an isomorphism $\mathcal{S}\stackrel{\sim}{\rightarrow}\mathcal{A}$. Conversely, if $\mathscr{I}$ restricts to an isomorphism $\mathcal{S}\stackrel{\sim}{\rightarrow}\mathcal{A}$, then clearly $\bar{\mathscr{I}}$ restricts to a surjective homomorphism $\gr^\bullet\mathcal{S}\rightarrow\gr^\bullet\mathcal{A}$. This completes the proof.
\end{proof}

\subsubsection{}
\label{sec:manylemmas}

We now consider a multicurve $\sigma$ on~$S=S_{0,n+2}$. Let $m$ be the geometric intersection number of $\sigma$ and the curve~$\gamma_k$ illustrated in Figure~\ref{fig:generators}. By the skein relations, we have an expansion 
\begin{equation}
\label{eqn:highestlowestterm}
\gamma_k\sigma=A^m\sigma_m+A^{-m}\sigma_{-m}+\dots
\end{equation}
for some multicurves $\sigma_{\pm m}$ where ``$\dots$'' denotes a linear combination of multicurves whose degree in~$\mathcal{S}$ is strictly less than the degree of~$\sigma$. We will write $\tw_{\gamma_k}^+(\sigma)=\sigma_m$ and $\tw_{\gamma_k}^-(\sigma)=\sigma_{-m}$. Recall that the set $\mathcal{P}=\{\gamma_1,\dots,\gamma_{n-1}\}$ provides a pants decomposition of~$S$ and so we get an embedding~$I$ sending a multicurve on~$S$ to its Dehn--Thurston coordinates as in~\eqref{eqn:DehnThurston}. The reason for our notation is that if $I(\sigma)=(l_{\gamma_i},t_{\gamma_i})$ then $I(\tw_{\gamma_k}^\pm(\sigma))=(l_{\gamma_i},t_{\gamma_i}')$ has the twist parameter $t_{\gamma_i}'=t_{\gamma_i}\pm1$ for~$i=k$ and $t_{\gamma_i}'=t_{\gamma_i}$ for~$i\neq k$.

\begin{lemma}
\label{lem:twisting}
Let $\sigma$ be a multicurve on~$S$ such that $\bar{\mathscr{I}}(\bar{\sigma})=fr_\lambda$ for some $\lambda=\sum_i\lambda_i\alpha_i\in\Lambda^+$ and some $f\in\mathbb{C}[T\times T_F\times\mathbb{C}^*]$. If $\lambda_k\neq0$ then 
\[
\bar{\mathscr{I}}(\overline{\tw_{\gamma_k}^{\pm}(\sigma)})=q^{\pm\lambda_k}X_k^{\pm1}\bar{\mathscr{I}}(\bar{\sigma})
\]
in $\gr^\bullet\mathcal{A}$.
\end{lemma}

\begin{proof}
Lemma~\ref{lem:rlambdacommutation} says 
\[
fr_\lambda(X_k+X_k^{-1})=(q^{2\lambda_k}X_k+q^{-2\lambda_k}X_k^{-1})fr_\lambda,
\]
which implies 
\[
(q^{2\lambda_k}-q^{-2\lambda_k})X_kfr_\lambda=fr_\lambda(X_k+X_k^{-1})-q^{-2\lambda_k}(X_k+X_k^{-1})fr_\lambda.
\]
On the other hand, the relation~\eqref{eqn:highestlowestterm} implies 
\begin{align*}
\gamma_k\sigma &= A^m\tw_{\gamma_k}^+(\sigma)+A^{-m}\tw_{\gamma_k}^-(\sigma)+\dots, \\
\sigma\gamma_k &= A^{-m}\tw_{\gamma_k}^+(\sigma)+A^m\tw_{\gamma_k}^-(\sigma)+\dots,
\end{align*}
where $m$ is the geometric intersection number of $\sigma$ with $\gamma_k$ so that $m=2\lambda_k$ and the ``$\dots$'' in each equation denotes a linear combination of multicurves whose degrees in $\mathcal{S}$ are strictly less than the degree of~$\sigma$. Hence, in $\gr^\bullet\mathcal{S}$ we have 
\[
\bar{\sigma}\bar{\gamma}_k-A^{2m}\bar{\gamma}_k\bar{\sigma}=(A^{-m}-A^{3m})\overline{\tw_{\gamma_k}^+(\sigma)}=A^m(A^{-2m}-A^{2m})\overline{\tw_{\gamma_k}^+(\sigma)}.
\]
Since $\bar{\mathscr{I}}(\bar{\gamma}_k)=X_k+X_k^{-1}$, we obtain 
\begin{align*}
\bar{\mathscr{I}}(\overline{\tw_{\gamma_k}^+(\sigma)}) &= \bar{\mathscr{I}}\left(\frac{A^{-m}}{A^{-2m}-A^{2m}}(\bar{\sigma}\bar{\gamma_k}-A^{2m}\bar{\gamma_k}\bar{\sigma})\right) \\
&= \frac{q^{\lambda_k}}{q^{2\lambda_k}-q^{-2\lambda_k}}(fr_\lambda(X_k+X_k^{-1})-q^{-2\lambda_k}(X_k+X_k^{-1})fr_\lambda) \\
&= q^{\lambda_k}X_kfr_\lambda.
\end{align*}
The proof of the identity $\bar{\mathscr{I}}(\overline{\tw_{\gamma_k}^-(\sigma)})=q^{-\lambda_k}X_k^{-1}fr_\lambda$ is similar.
\end{proof}

\begin{lemma}
\label{lem:dressingassociatedgraded}
For any integer~$k$, we have 
\[
\overline{E_{i,1}[x^k]F_{i,1}[1]}=X_i^kr_{\alpha_i}, \quad \overline{F_{i,1}[x^k]E_{i,1}[1]}=q^{-2k}X_i^{-k}r_{\alpha_i}
\]
in the associated graded $\gr^\bullet\mathcal{A}$.
\end{lemma}

\begin{proof} 
Let us write $G=\mathrm{GL}_2$ and consider the orbit $\Gr_G^{\varpi_1^*}=G_\mathcal{O}z^{\varpi_1^*}\subset\Gr_G$ where $\varpi_1^*$ is the cocharacter of~$G$ given by $t\mapsto\left(\begin{smallmatrix}1 & 0 \\ 0 & t^{-1}\end{smallmatrix}\right)$. There is an isomorphism $\Gr_G^{\varpi_1^*}\stackrel{\sim}{\rightarrow} G_{\mathcal{O}}/\Stab(z^{\varpi_1^*})$ given by $g(z)z^{\varpi_1^*}\mapsto g(z)\Stab(z^{\varpi_1^*})$ where $\Stab(z^\lambda)=G_{\mathcal{O}}\cap z^\lambda G_{\mathcal{O}}z^{-\lambda}\subset G_{\mathcal{O}}$ is the stabilizer of $z^\lambda$ for $\lambda\in\Lambda^+$. There is a further isomorphism $G_{\mathcal{O}}/\Stab(z^{\varpi_1^*})\stackrel{\sim}{\rightarrow}G/B^-$ given by $g(z)\Stab(z^{\varpi_1^*})\mapsto g(0)B^-$ where $B^-\subset G$ is the subgroup of lower triangular matrices. Finally, there is an isomorphism $G/B^-\stackrel{\sim}{\rightarrow}\mathbb{P}^1$ that maps $gB^-\in G/B^-$ to the line $\mathbb{C}\cdot g\left(\begin{smallmatrix}0 \\ 1\end{smallmatrix}\right)\subset\mathbb{C}^2$. Composing these maps, we obtain an isomorphism $\Gr_G^{\varpi_1^*}\stackrel{\sim}{\rightarrow}\mathbb{P}^1$ mapping $g(z)z^{\varpi_1^*}\mapsto\mathbb{C}\cdot g(0)\left(\begin{smallmatrix}0 \\ 1\end{smallmatrix}\right)$. The group $\mathbb{C}^*$ acts on the orbit $\Gr_G^{\varpi_1^*}$ by loop rotation so that a complex number $\zeta\in\mathbb{C^*}$ maps $g(z)z^{\varpi_1^*}\mapsto g(\zeta z)(\zeta z)^{\varpi_1^*}=g(\zeta z)\zeta^{\varpi_1^*}z^{\varpi_1^*}$. Under the above isomorphism $\Gr_G^{\varpi_1^*}\cong\mathbb{P}^1$, this corresponds to the $\mathbb{C}^*$-action on~$\mathbb{P}^1$ given by $\mathbb{C}\cdot h\left(\begin{smallmatrix}0 \\ 1\end{smallmatrix}\right)\mapsto\mathbb{C}\cdot h\left(\begin{smallmatrix}0 \\ \zeta^{-1}\end{smallmatrix}\right)$ for $h\in G$.

As in~Section~5.2.16 of~\cite{CG97}, there is an isomorphism $K^{G\times\mathbb{C}^*}(\mathbb{P}^1)=K^{G\times\mathbb{C}^*}(G/B^-)\stackrel{\sim}{\rightarrow}K^{B^-\times\mathbb{C}^*}(\mathrm{pt})$, mapping the class of a $G\times\mathbb{C}^*$-equivariant coherent sheaf on~$G/B^-$ to its restriction to $1=1B^-\in G/B^-$. If $T\subset G$ is the diagonal subgroup, then there are further isomorphisms $K^{B^-\times\mathbb{C}^*}(\mathrm{pt})\cong\mathbb{C}[X^*(B^-\times\mathbb{C}^*)]\cong\mathbb{C}[X^*(T\times\mathbb{C}^*)]$. Note that if $\mathcal{F}=\mathcal{O}(-1)$ is the tautological vector bundle on~$\mathbb{P}^1$, then the total space of~$\mathcal{F}|_1$ is the line $\mathbb{C}\left(\begin{smallmatrix}0 \\ 1\end{smallmatrix}\right)\subset\mathbb{C}^2$. Composing the above maps, we therefore obtain an isomorphism $K^{G\times\mathbb{C}^*}(\mathbb{P}^1)\stackrel{\sim}{\rightarrow}\mathbb{C}[X^*(T\times\mathbb{C}^*)]$ mapping the class of this vector bundle $\mathcal{F}$ to the character $\left(\left(\begin{smallmatrix}t_1 & 0 \\ 0 & t_2\end{smallmatrix}\right),\zeta\right)\mapsto t_2\zeta^{-1}$ of~$T\times\mathbb{C}^*$. The second equation now follows from the definition of the dressed minuscule monopole operators given in equation~\eqref{eqn:vectorbundles} and the definitions of~$q$ and~$X_i$.

The proof of the first equation is similar. Namely, we consider the orbit $\Gr_G^{\varpi_1}=G_{\mathcal{O}}z^{\varpi_1}\subset\Gr_G$ where $\varpi_1$ is the cocharacter of $G$ given by $t\mapsto\left(\begin{smallmatrix}t & 0 \\ 0 &1\end{smallmatrix}\right)$. In the same way as before, we get isomorphisms $\Gr_G^{\varpi_1}\cong G_{\mathcal{O}}/\Stab(z^{\varpi_1})\cong G/B^-$. There is an isomorphism $G/B^-\stackrel{\sim}{\rightarrow}\mathbb{P}^1$ given by $gB^-\mapsto\mathbb{C}^2/\mathbb{C}\cdot g\left(\begin{smallmatrix}0 \\ 1\end{smallmatrix}\right)$ where we view $\mathbb{P}^1$ as the space of one-dimensional quotients of~$\mathbb{C}^2$. Composing these maps, we obtain an isomorphism $\Gr_G^{\varpi_1}\cong\mathbb{P}^1$, and the $\mathbb{C}^*$-action on~$\Gr_G^{\varpi_1}$ by loop rotation is trivial. We again have the isomorphism $K^{G\times\mathbb{C}^*}(\mathbb{P}^1)\cong K^{B^-\times\mathbb{C}^*}(\mathrm{pt})$ given by restriction to $1\in G/B^-\cong\mathbb{P}^1$ and the isomorphism $K^{B^-\times\mathbb{C}^*}(\mathrm{pt})\cong\mathbb{C}[X^*(T\times\mathbb{C}^*)]$. Let $\mathcal{E}$ is the bundle over~$\mathbb{P}^1$ such that the fiber over a line $\ell\subset\mathbb{C}^2$ is the quotient $\mathbb{C}^2/\ell$. Then the total space of $\mathcal{E}|_1$ is $\mathbb{C}^2/\left(\mathbb{C}\left(\begin{smallmatrix}0 \\ 1\end{smallmatrix}\right)\right)\cong\mathbb{C}\left(\begin{smallmatrix}1 \\ 0\end{smallmatrix}\right)$, and therefore the class of $\mathcal{E}$ maps to the character $\left(\left(\begin{smallmatrix}t_1 & 0 \\ 0 & t_2\end{smallmatrix}\right),\zeta\right)\mapsto t_1$ of~$T\times\mathbb{C}^*$. The first equation now follows from~\eqref{eqn:vectorbundles} and the definition of~$X_i$.
\end{proof}

Using this lemma, we can give a proof of Lemma~\ref{lem:rlambdacommutation}, which was omitted in Section~\ref{sec:GeneratingTheAssociatedGraded}.

\begin{proof}[Proof of Lemma~\ref{lem:rlambdacommutation}]
Using the formulas for the dressed minuscule monopole operators, one sees that 
\[
\overline{E_{j,1}[1]F_{j,1}[1](q^{2k}w_{j,+}+q^{-2k}w_{j_-})} = \overline{q^{2(k+1)}E_{j,1}[x]F_{j,1}[1]} + \overline{q^{-2(k+1)}E_{j,1}[x^{-1}]F_{j,1}[1]}
\]
for any $j=1,\dots, n-1$ and any integer~$k$. By Lemma~\ref{lem:dressingassociatedgraded}, this is equivalent to the identity $r_{\alpha_j}(q^{2k}X_j+q^{-2k}X_j^{-1})=(q^{2(k+1)}X_j+q^{-2(k+1)}X_j^{-1})r_{\alpha_j}$. Applying this identity iteratively, we obtain 
\[
r_{\alpha_j}^{\lambda_j}(X_j+X_j^{-1})=(q^{2\lambda_j}X_j+q^{-2\lambda_j}X_j^{-1})r_{\alpha_j}^{\lambda_j}
\]
for any integer~$\lambda_j\geq0$. The formulas for the monopole operators similarly imply $r_{\alpha_i}(q^{2k}X_j+q^{-2k}X_j^{-1})=(q^{2k}X_j+q^{-2k}X_j^{-1})r_{\alpha_i}$ for $i\neq j$. The lemma now follows since, for any $\lambda=\sum_i\lambda_i\alpha_i\in\Lambda^+$, we can write $r_\lambda=f\prod_ir_{\alpha_i}^{\lambda_i}$ for some $f\in\mathbb{C}[T\times T_F\times\mathbb{C}^*]$ by Proposition~\ref{prop:AssociatedGradedMultiplication} and Corollary~\ref{cor:easymultiplication} after choosing an order for the factors in the product.
\end{proof}

\begin{lemma}
\label{lem:Ibarcalculation}
For all $1\leq i\leq j\leq n-2$, we have 
\[
\bar{\mathscr{I}}(\bar{\sigma}_{i,j+1})=-q^{i-j-1}t_i^{-1}t_{i+1}^{-3}\dots t_j^{-3}t_{j+1}^{-1}X_i^{-2}\dots X_j^{-2}r_{\alpha_{i,j}}
\]
where $\sigma_{i,j+1}$ is the curve defined in Figure~\ref{fig:generators}.
\end{lemma}

\begin{proof}
Let $\gamma_i$ and $\theta_{i,m}$ be the curves defined in Figures~\ref{fig:generators} and~\ref{fig:familycurves}. Then we have 
\[
\bar{\mathscr{I}}(\bar{\gamma}_i)=-(X_i+X_i^{-1})
\]
and 
\[
\bar{\mathscr{I}}(\bar{\theta}_{i,m})=
\begin{cases}
-q^{m-1}t_i^{-1}t_{i+1}^{-1}\overline{E_{i,1}[x^{m-2}]F_{i,1}[1]} & \text{for $i$ even} \\
-q^{3-m}t_i^{-1}t_{i+1}^{-1}\overline{F_{i,1}[x^{2-m}]E_{i,1}[1]} & \text{for $i$ odd}.
\end{cases}
\]
By Lemma~\ref{lem:dressingassociatedgraded}, we have $\overline{E_{i,1}[x^{m-2}]F_{i,1}[1]}=X_i^{m-2}r_{\alpha_i}$ and $\overline{F_{i,1}[x^{2-m}]E_{i,1}[1]}=q^{2m-4}X_i^{m-2}r_{\alpha_i}$. It follows that for any $i$, we have 
\[
\bar{\mathscr{I}}(\bar{\theta}_{i,m})=-q^{m-1}t_i^{-1}t_{i+1}^{-1}X_i^{m-2}r_{\alpha_i}.
\]
In particular, since $\sigma_{i,i+1}=\theta_{i,0}$, we see that $\bar{\mathscr{I}}(\bar{\sigma}_{i,i+1})=-q^{-1}t_i^{-1}t_{i+1}^{-1}X_i^{-2}r_{\alpha_i}$.

Let us consider the integer $k=j-i$. We have just proved that the statement in the lemma is true when $k=0$. Assume inductively that it is true for $k\leq K-2$, and suppose $i$,~$j$ satisfy $j-i=K-1$. Without loss of generality, we may assume that $S_{0,n+2}=S_{0,K+3}$ and $i=1$,~$j=K$. Let us write $\sigma\coloneqq\sigma_{i,j+1}$, $\sigma'\coloneqq\sigma_{i,j}$, $\sigma''\coloneqq\sigma_{i,j-1}$ and consider the curves $\delta_K$, $\rho$, $\zeta$, $\eta$, $\mu$, $\nu$, and~$\tau$ on~$S_{0,K+3}$ depicted in Figures~\ref{fig:generators} and~\ref{fig:skeinsplanarview}. We have seen that these curves obey the matrix equation~\eqref{eqn:skeinmatrix}. In the associated graded, the degree of $\bar{\delta}_{K+1}\bar{\sigma}''$ is strictly less than the degrees of $\bar{\sigma}'\bar{\rho}$ and~$\bar{\rho}\bar{\sigma}'$, and the degree of $\bar{\delta}_0\bar{\delta}_{K+2}$ is strictly less than the degrees of $\bar{\mu}\bar{\nu}$ and~$\bar{\nu}\bar{\mu}$ as $\bar{\delta}_r$ has degree zero for every~$r$. Therefore we obtain an equation 
\[
\begin{pmatrix}
\bar{\sigma}'\bar{\rho} \\
\bar{\rho}\bar{\sigma}' \\
\bar{\mu}\bar{\nu} \\
\bar{\nu}\bar{\mu}
\end{pmatrix}
=
\begin{pmatrix}
A^{-2} & A^2 & \bar{\delta}_K & 0 \\
A^2 & A^{-2} & \bar{\delta}_K & 0 \\
\bar{\delta}_K & 0 & A^2 & A^{-2} \\
\bar{\delta}_K & 0 & A^{-2} & A^2
\end{pmatrix}
\begin{pmatrix}
\bar{\zeta} \\
\bar{\eta} \\
\bar{\sigma} \\
\bar{\tau}
\end{pmatrix}.
\]
Applying Lemma~\ref{lem:twisting} to $\mu=\tw_{\gamma_{j-1}}^+(\sigma')$ and $\nu=\tw_{\gamma_j}^-(\rho)$, we find that $\bar{\mathscr{I}}(\bar{\mu})=qX_{j-1}\bar{\mathscr{I}}(\bar{\sigma}')$ and $\bar{\mathscr{I}}(\bar{\nu})=q^{-1}X_j^{-1}\bar{\mathscr{I}}(\bar{\rho})$. The inductive hypothesis implies $\bar{\mathscr{I}}(\bar{\rho})=-q^{-1}t_j^{-1}t_{j+1}^{-1}X_j^{-2}r_{\alpha_{i,j}}$ and $\bar{\mathscr{I}}(\bar{\sigma}')=-q^{i-j}t_i^{-1}t_{i+1}^{-3}\dots t_{j-1}^{-3}t_j^{-1}X_i^{-2}\dots X_{j-1}^{-2}r_{\alpha_{i,j-1}}$. Applying the map $\bar{\mathscr{I}}$ to all matrix elements in the above equation and factoring in the associated graded using Proposition~\ref{prop:AssociatedGradedMultiplication}, we therefore obtain 
\[
C\cdot\begin{pmatrix}
r_{\alpha_{i,j-1}}*r_{\alpha_j} \\
r_{\alpha_j}*r_{\alpha_{i,j-1}} \\
(X_{j-1}r_{\alpha_{i,j-1}})*(X_j^{-1}r_{\alpha_j}) \\
(X_j^{-1}r_{\alpha_j})*(X_{j-1}r_{\alpha_{i,j-1}})
\end{pmatrix}
=
\begin{pmatrix}
q & q^{-1} & \delta & 0 \\
q^{-1} & q & \delta & 0 \\
\delta & 0 & q^{-1} & q \\
\delta & 0 & q & q^{-1}
\end{pmatrix}
\begin{pmatrix}
\bar{\mathscr{I}}(\bar{\zeta}) \\
\bar{\mathscr{I}}(\bar{\eta}) \\
\bar{\mathscr{I}}(\bar{\sigma}) \\
\bar{\mathscr{I}}(\bar{\tau})
\end{pmatrix}
\]
where $C=q^{-i-j-1}t_i^{-1}t_{i+1}^{-3}\dots t_{j-1}^{-3}t_j^{-2}t_{j+1}^{-1}X_i^{-2}\dots X_j^{-2}$ and $\delta=-t_j-t_j^{-1}$. Comparing this with equation~\eqref{eqn:Coulombmatrix}, we see that $\bar{\mathscr{I}}(\bar{\sigma})=-q^{-i-j-1}t_i^{-1}t_{i+1}^{-3}\dots t_j^{-3}t_{j+1}^{-1}X_i^{-2}\dots X_j^{-2}r_{\alpha_{i,j}}$. The lemma follows by induction.
\end{proof}

\subsection{}

We can now complete the proof of our isomorphism.

\begin{lemma}
\label{lem:surjectiveassociatedgraded}
$\bar{\mathscr{I}}$ restricts to a surjective algebra homomorphism $\gr^\bullet\mathcal{S}\rightarrow\gr^\bullet\mathcal{A}$.
\end{lemma}

\begin{proof}
Recall from Section~\ref{sec:TheAssociatedGradedAlgebraSkein} that $\gr^\lambda\mathcal{S}$ has a basis given by the set of multicurves of degree exactly~$\lambda\in2\mathbb{Z}_{\geq0}^{n-1}$. By Lemma~\ref{lem:specialmulticurve}, there is an element~$\sigma\in\gr^\lambda\mathcal{S}$ of this basis which is a union of curves of the form~$\sigma_{i,j}$. From Lemma~\ref{lem:Ibarcalculation} and the fact that $\bar{\mathscr{I}}$ is a graded algebra homomorphism, we know that $\bar{\mathscr{I}}(\bar{\sigma})\in\gr^\lambda\mathcal{A}$. If $\sigma'\in\gr^\lambda\mathcal{S}$ is any other element of the basis, then $\sigma'$ is obtained from~$\sigma$ by changing the Dehn--Thurston twist parameters, that is, by repeatedly applying the operations $\tw_{\gamma_k}^{\pm}$. Thus Lemma~\ref{lem:twisting} implies $\bar{\mathscr{I}}(\bar{\sigma}')\in\gr^\lambda\mathcal{A}$. This shows that $\bar{\mathscr{I}}$ maps all basis elements into~$\gr^\lambda\mathcal{A}$ and therefore $\bar{\mathscr{I}}(\gr^\lambda\mathcal{S})\subset\gr^\lambda\mathcal{A}$.

To prove surjectivity, note that by Lemma~\ref{lem:Ibarcalculation} and repeated application of Lemma~\ref{lem:twisting}, we have $fr_{\alpha_{i,j}}\in\bar{\mathscr{I}}(\gr^\bullet\mathcal{S})$ for every $f\in\mathbb{C}[T\times T_F\times\mathbb{C}^*]^{W_{\alpha_{i,j}}}$. Arguing as in the proof of Lemma~\ref{lem:generateassociatedgraded}, we see that $fr_\lambda\in\bar{\mathscr{I}}(\gr^\bullet\mathcal{S})$ for every $\lambda\in\Lambda^+$ and every $f\in\mathbb{C}[T\times T_F\times\mathbb{C}^*]^{W_{\alpha_\lambda}}$. By Proposition~\ref{prop:associatedgradedbasis}, elements of the form $fr_\lambda$ form a basis for~$\gr^\bullet\mathcal{A}$. Hence $\bar{\mathscr{I}}$ surjects onto~$\gr^\bullet\mathcal{A}$.
\end{proof}

It now follows from Lemmas~\ref{lem:conditionforiso} and~\ref{lem:surjectiveassociatedgraded} that the map $\mathscr{I}$ of Proposition~\ref{prop:mainlocalized} restricts to an isomorphism $\mathcal{S}\stackrel{\sim}{\rightarrow}\mathcal{A}$ of integral forms. This proves Theorem~\ref{thm:intromain} from the introduction. Corollary~\ref{cor:categorification} is then an immediate consequence of the definition of the Coulomb branch given in Section~\ref{sec:QuantizedCoulombBranchesAndCategorification}.

\bibliographystyle{amsplain}

\end{document}